\documentclass[11pt,a4paper]{compositio}
\usepackage[utf8]{inputenc}
\usepackage{lmodern}
\usepackage{microtype,scalefnt}

\usepackage{mathtools}
\usepackage{amsfonts,amssymb,amsthm,amscd}
\usepackage{mathrsfs}
\usepackage{hyperref}
\usepackage{bookmark}
\usepackage{algorithmic, algorithm}
\usepackage[all]{xy}
\usepackage{color}

\renewcommand{\contentsname}{Contents}

\newcommand{\bZ}{{\mathbf Z}}

\newcommand{\bi}{\mathbf{i}}
\newcommand{\bj}{\mathbf{j}}
\newcommand{\bt}{\mathbf{t}}

\newcommand{\cO}{\mathcal{O}}
\newcommand{\cN}{\mathcal N}

\newcommand{\cH}{\mathcal{H}}

\newcommand{\cG}{\mathcal{G}}

\newcommand{\cL}{\mathcal{L}}

\renewcommand{\L}{\mathcal{L}}

\newcommand{\cM}{\mathcal{M}}

\newcommand{\ra}{\rightarrow}
\newcommand{\xra}[1]{\xrightarrow{#1}}
\newcommand{\hra}{\hookrightarrow}

\newcommand{\vphi}{\varphi}

\newcommand{\Zhat}{\widehat{\Z}}

\newcommand{\comment}[1]{}
\newcommand{\Q}{\mathbf{Q}}

\newcommand{\C}{\mathbf{C}}

\renewcommand{\P}{\mathbf{P}}

\newcommand{\Z}{\mathbf{Z}}
\newcommand{\F}{\mathbf{F}}

\renewcommand{\O}{\mathcal{O}}

\newcommand{\isom}{\cong}

\newcommand{\Akbar}{A_{\overline{k}}}

\DeclareMathOperator{\chainadd}{\texttt{chain\_add}}
\DeclareMathOperator{\chainmult}{\texttt{chain\_mult}}
\DeclareMathOperator{\chainmultadd}{\texttt{chain\_multadd}}

\DeclareMathOperator{\RM}{RM}
\DeclareMathOperator{\Id}{Id}
\DeclareMathOperator{\NS}{\mathbf{NS}}

\DeclareMathOperator{\Pic}{Pic}

\DeclareMathOperator{\Aut}{Aut}

\DeclareMathOperator{\GL}{\mathbf{GL}}

\DeclareMathOperator{\Sp}{\mathbf{Sp}}

\DeclareMathOperator{\PGL}{\mathbf{PGL}}

\DeclareMathOperator{\Gal}{Gal}

\DeclareMathOperator{\diag}{diag}
\DeclareMathOperator{\Hom}{Hom}

\DeclareMathOperator{\End}{End}

\DeclareMathOperator{\Jac}{Jac}

\theoremstyle{plain} 
\newtheorem{thm}{Theorem}[section] 
\newtheorem{prop}[thm]{Proposition}
\newtheorem{cor}[thm]{Corollary}
\newtheorem{lem}[thm]{Lemma}

\theoremstyle{definition} 
\newtheorem{defn}[thm]{Definition} 
\theoremstyle{remark} 
\newtheorem{rem}{Remark}
\newtheorem{notation}{Notation}

\newcommand{\leftexp}[2]{{\vphantom{#2}}^{#1}{#2}}

\newcounter{tasknumber}
\newcommand{\task}[2][]{%
  \addtocounter{tasknumber}{1}%
  \begin{center}%
  \framebox[1.1\width]{\begin{minipage}{0.9\textwidth}%
  \textbf{Task \arabic{tasknumber}} \textit{\if!#1(unassigned)!\else (#1)\fi}: {#2}%
  \end{minipage}}%
  \end{center}%
}

\newcounter{assumptionnumber}
\newcommand{\assumption}[2][]{%
  \addtocounter{assumptionnumber}{1}%
  \begin{center}%
  \framebox[1.1\width]{\begin{minipage}{0.9\textwidth}%
  \textbf{Assumption \arabic{assumptionnumber}} \textit{\if!#1!\else (#1)\fi}: {#2}%
  \end{minipage}}%
  \end{center}%
}

\newcommand{\authnote}[2][]{\noindent {\if!#1!  {\bf TODO} \else {\small \bf #1} \fi: #2} \vspace{0.1in}}

\title[]{Cyclic Isogenies for Abelian Varieties with Real Multiplication}

\author{Alina Dudeanu}
\email{alina.dudeanu@epfl.ch}
\address{Ecole Polytechnique F\'ed\'erale de Lausanne, Switzerland}

\author{Dimitar Jetchev}
\email{dimitar.jetchev@epfl.ch}
\address{Ecole Polytechnique F\'ed\'erale de Lausanne, Switzerland}

\author{Damien Robert}
\email{damien.robert@inria.fr}
\address{Universit\'e de Bordeaux, France}

\author{Marius Vuille}
\email{marius.vuille@epfl.ch}
\address{Ecole Polytechnique F\'ed\'erale de Lausanne, Switzerland}

\begin{document}

\begin{abstract}
We study quotients of principally polarized abelian varieties with real multiplication by finite Galois-stable subgroups and describe when these quotients are principally polarizable. We use this characterization to provide an algorithm to compute explicit cyclic isogenies from kernel for ordinary and simple abelian varieties over finite fields. Our algorithm is polynomial in the number of binary digits of the finite field as well as in the degree of the isogeny and is based on Mumford's theory of theta functions. Recently, the algorithm has been successfully applied to obtain new results on the discrete logarithm problem in genus~2 as well as to study the discrete logarithm problem in genus 3. 
\end{abstract}

\maketitle


%
%
\section{Introduction}

\subsection{Motivation}
Let $f \colon A \ra B$ be a separable isogeny of ordinary and absolutely simple abelian varieties of dimension 
$g \geq 1$ over a finite field $k$ and suppose that $G = \ker (f)$ is a cyclic subgroup defined over $k$ (as a group scheme, not necessarily pointwise). 
The endomorphism algebras $\End^0(A) = \End(A) \otimes_{\Z} \Q$ and $\End^0(B) = \End(B) \otimes_{\Z} \Q$ are isomorphic to a CM field $K$, that is, a totally imaginary quadratic extension of a totally real number field $K_0 / \Q$ with $[K_0 : \Q] = g$.
Suppose that $A$ is equipped with a principal polarization $\lambda_A$ (an algebraic equivalence class of ample line bundles of degree 1 on~$A$). The polarization isogeny $\lambda_A \colon A \ra A^\vee$ (where $A^\vee$ is the dual abelian variety of $A$) determines a Rosati involution $\dagger \colon \End(A) \ra \End(A)$ defined by $\vphi \mapsto \vphi^\dagger = \lambda_A^{-1} \circ \vphi^\vee \circ \lambda_A$. Under any choice of isomorphism $\End^0(A) \xrightarrow{\sim} K$, the complex conjugation on $K$ corresponds to the Rosati involution on~$\End^0(A)$. 

In this paper, we study when the target abelian variety $B$ is principally polarizable and if so, whether one can explicitly compute $B$ as well as evaluate the isogeny $f$ on points. The latter needs a clarification as it is not even clear how the abelian varieties are represented. If $g = 1$ (elliptic curves) then it amounts to computing a Weierstrass equation for $B$ and providing an algorithm to compute images of points under $f$ \cite{velu}. The problem is much more challenging for $g > 1$. One way to represent a principally polarized abelian variety which is a Jacobian of an algebraic curve is via degree-zero divisor classes. This assumes that one knows an explicit model of the curve which need not 
always be the case. Even worse, there is no reason that the non-polarized abelian variety~$B = A / G$ even admits a principal polarization, so working with linear equivalence classes of divisors of degree zero is 
not a suitable option. 

A different approach is via Mumford's theory of projective embeddings of abelian varieties via theta functions. The latter allows us to compute explicit coordinates for $B$ on a certain projective model of the moduli space of principally polarized abelian varieties out of the data for $A$ and $G$. In a series of papers \cite{mumford:eq1, mumford:eq2, mumford:eq3} (see also \cite{mumford1984}), Mumford defines a group associated to an invertible sheaf on an abelian variety (Mumford's theta group). This group has a natural action on the space of global sections of the sheaf, thus yielding a natural linear representation. As Mumford's theta group is abstractly non-canonically isomorphic to a finite analogue of the Heisenberg group, a finite analogue of Stone--von Neuman's theorem then yields that such a representation is unique. Fixing a choice of such an isomorphism between the Heisenberg group and Mumford's theta group yields a system of canonical projective coordinates for the abelian variety.    
It is these coordinates for the target abelian variety that one may try to compute 
explicitly.

\subsection{Main results}
We first provide a criterion for the target abelian variety $B = A / G$ to admit a principal polarization. Let $\End(A)^+ \subset \End(A)$ be the subset of symmetric endomorphisms of $A$, that is, endomorphisms $\varphi \in \End(A)$ for which $\varphi^{\dagger} = \varphi$. Furthermore, let $\End(A)^{++} \subset \End(A)^+$ be the subset of totally positive symmetric endomorphisms (i.e. endomorphisms $\varphi \in \End(A)^+$ that correspond to totally positive elements of $\cO_{K_0}$ under the isomorphism $\End^0(A) \xrightarrow{\sim} K$). If $\cL$ is any ample line bundle on~$A$ (which by abuse of notation we will henceforth also call a polarization) and $\beta \in \End(A)^{++}$, one defines a polarization via the composition 
$\varphi_{\cL} \circ \beta \colon A \ra A^\vee$. We will show in Section \ref{subsec:principal-polarization} that this composition indeed arises as the polarization isogeny of an ample line bundle on $A$, that we denote by~$\cL^\beta$, unique up to algebraic equivalence. 

Let $K(\cL^\beta)$ be the kernel of the polarization isogeny $\varphi_{\cL^\beta} \colon A \ra A^{\vee}$. Associated to the polarization~$\cL^\beta$ is a nondegenerate, alternating bilinear pairing (the commutator pairing) $$e_{\cL^\beta} \colon K(\cL^\beta) \times K(\cL^\beta) \ra \bar{k}^\times$$ (see Section~\ref{par:thetastruct} for the 
precise definition). 

\begin{thm}[(Principal polarizability of $B$)]\label{thm:princ_pol_of_B}
Let $\cL_0$ be a principal polarization on $A$. The a priori non-polarized abelian variety $B = A / G$ admits a principal polarization $\cM_0$ if and only if there exists a totally positive real endomorphism $\beta \in \End(A)^{++}$ such that $G \subset \ker (\beta) = K(\cL_0^\beta)$ and $G$ is a maximal isotropic subgroup for the commutator pairing $e_{\cL_0^\beta}$. 
\end{thm}

\begin{defn}
Let $(A, \cL)$ and $(B, \cM)$ be polarized abelian varieties. Given a totally positive real endomorphism $\beta \in \End(A)^{++}$, an isogeny $f \colon A \to B$ is called a \textit{$\beta$-cyclic isogeny} if $\ker(f)$ is a cyclic subgroup of $\ker(\beta)$ and if $f^* \cM$ is algebraically equivalent to $\cL^\beta$. To indicate the corresponding 
polarizations, we often denote the isogeny by $f \colon (A, \cL^\beta) \to (B, \cM)$
\end{defn}

\noindent Let $f \colon (A, \cL^\beta) \ra (B, \cM)$ be a $\beta$-cyclic isogeny for some $\beta \in \End(A)^{++}$, and suppose that $\cL$ and $\cM$ are very ample. Given a projective embedding $A \hra \mathbf{P}(\Gamma(A, \cL))$, we will be interested in computing a projective embedding $B \hra \mathbf{P}(\Gamma(B, \cM))$. Yet, the data of an abelian variety together with a very ample line bundle does not 
determine a canonical projective embedding in the space of global sections. In addition, one needs to make precise the notion of projective embeddings being compatible under isogeny. 
Mumford \cite{mumford:eq1} resolved these questions by introducing the notion of theta structures (see Section~\ref{sec:theta} for details). Fixing a theta structure $\Theta_{\cL}$ on a polarized abelian variety~$(A, \cL)$ yields canonical projective coordinates for the variety. In addition, one can descend the data of a polarized abelian variety with theta structure under isogeny (see Theorem~\ref{thm:isog}). Our main theorem provides an algorithm that computes the theta coordinates of the image $f(x)$ of a geometric point $x$ of $A$ with respect to a suitable theta structure on $(B, \cM)$ from the theta coordinates of $x$ with respect to~$\Theta_\cL$. The \textit{theta null point} of~$(A, \cL)$ for $\Theta_\cL$ is the image of $0_A$ under the projective embedding of $A$ induced by $\Theta_\cL$.

\begin{thm}[(Computing $\beta$-cyclic isogenies)]\label{thm:main}
Let $(A, \cL_0)$ be a principally polarized ordinary and simple abelian variety of dimension $g$ over the finite field $k$ of cardinality $q = \#k$. Let $n = 2$ or $n = 4$. Suppose that $\cL_0$ is a symmetric line bundle and let $\Theta_{\cL}$ be a symmetric theta structure for the polarization~$\cL = \cL_0^{\otimes n}$ (see Section \ref{par:sym_theta_structures} for the definitions). Let $K$ be the CM field that is the endomorphism algebra of $A$ and let $K_0 \subset K$ be the totally real subfield. Let $\ell$ be an odd prime number different from the characteristic of~$k$ and let $G \subset A(\bar{k})$ be a 
$\pi$-stable cyclic subgroup of order $\ell$ where $\pi \colon A \ra A$ is the Frobenius endomorphism. Assume the following hypotheses:

\begin{enumerate}
\item[H.1] The prime $\ell$ is either split or ramified in $\cO_{K_0}$ and one of the prime ideals of $\cO_{K_0}$ above $\ell$ is principal and generated by a totally positive element $\beta \in \cO_{K_0}$ of norm $\ell$ such that $G \subset A[\beta]$.


\item[H.2] The abelian variety $A$ has maximal local real endomorphism ring at $\ell$, i.e. $\End(A)^+ \otimes_{\Z} \Z_\ell$ is isomorphic to $\cO_{K_0} \otimes_{\Z} \Z_\ell$. 

\item[H.3] The conductor gap $[\cO_{K_0} \colon \Z[\pi + \pi^\dagger]]$ is coprime to $2\ell$. 

\item[H.4] There exists an algorithm $\RM(\alpha, y)$ that computes the action of a real multiplication (RM) endomorphism $\alpha \in \End(A)^+$ on a 4-torsion point $y \in A[4]$. 
\end{enumerate}
Fixing $\beta$ as in H.1 and denoting by $\cM_0$ the principal polarization on $B = A / G$ from Theorem \ref{thm:princ_pol_of_B}, and $\cM = \cM_0^{\otimes n}$, we have:

\vspace{0.1in}

\noindent (i) Given the theta null point of $(A, \cL)$ for $\Theta_{\cL}$ and the theta coordinates of a generator of $G$ with respect to $\Theta_\cL$, there exists an algorithm polynomial in $\log q$ and $\ell$ that computes the theta null point of $(B, \cM)$ for some theta structure $\Theta_{\cM}$. 

\vspace{0.1in}

\noindent (ii) Given a point $x \in A(k)$ of order coprime to $\ell$ in theta coordinates for $\Theta_\cL$, there exists an algorithm that computes the theta coordinates of $f(x)$ with respect to $\Theta_{\cM}$ in time polynomial in $\log q$ and $\ell$ where $\Theta_{\cM}$ is the theta structure from (i). 
\end{thm}

\begin{rem}
If $k$ is an arbitrary field of positive characteristic, Theorem~\ref{thm:main} (i) provides an algorithm to compute the theta null point of $(B, \cM)$ for some theta structure $\Theta_{\cM}$, and Theorem~\ref{thm:main} (ii) provides an algorithm to evaluate a cyclic isogeny on points $x$ whose order is finite and coprime to the degree $\ell$, provided $\End^+(A)$ acts on $x$ by scalar multiplication.
\end{rem}

\begin{rem}
In genus 2 and 3, knowing the theta null point for $(B, \cM, \Theta_{\cM})$ is enough to determine the equation of the underlying curve $C'$ of $B = \Jac(C')$. For the case of a hyperelliptic Jacobian, see~\cite{CossetRobert} and for the case of a Jacobian of a smooth plane quartic, see \cite{fiorentino}. 
\end{rem}

\begin{rem}
For all applications that we will consider, hypothesis H.4 is not too restrictive. In the case where $A$ is the Jacobian variety of a hyperelliptic curve, one can use Mumford's representation to compute this action (e.g., in this case, we know that the 2-torsion points are coming from the Weierstrass points on the curve). For $g = 2$, a more general method for $\RM$ is based on computing $\sqrt{D_0}$ as a $(D_0, D_0)$-isogeny (using the method of \cite{CossetRobert}) where $D_0$ is the discriminant of the real quadratic field.
\end{rem}

\begin{rem}
A natural question is to what extent is the theta structure $\Theta_{\cM}$ uniquely determined by 
the theta structure $\Theta_{\cL}$. As we will discuss in more detail in Section~\ref{sec:computing_theta_null_point}, this theta structure need not be unique. 
\end{rem}

\subsection{Known results and applications of the main theorem}

Computing isogenies in higher dimensions has been of considerable interest. The idea of using theta coordinates to compute isogenies from kernel is certainly not new and already appears in several prior works \cite{drobert:thesis}, \cite{CossetRobert}, \cite{lubicz-robert:isogenies}, 
\cite{couveignes-ezome}. Other more geometric methods have been considered as well \cite{richelot:37}, 
\cite{dolgachev-lehavi}, \cite{smith:genus3}, \cite{bruin-flynn-testa} and \cite{flynn:55}. 

Yet, in all of these works, the computed isogenies are special in the sense that the principal polarizability of the target variety does not depend on the existence of totally positive real endomorphisms of prime $K_0/ \Q$-norm, but is linked to multiplication by $\ell$, hence is always guaranteed.
In the case of cyclic isogenies, the target abelian variety need not be principally polarizable, so the problem requires much more careful analysis and novel ideas. 

Computing explicit isogenies have been fundamentally important in both computational number theory and 
mathematical cryptology. The basic point-counting algorithm of Schoof--Elkies--Atkin (SEA) relies heavily on isogeny computations (see \cite{fouquet-morain}). Kohel's algorithm for computing the endomorphism ring of an elliptic curve over a finite field \cite{kohel:thesis} as well as the analogous algorithm for genus 2 due to Bisson \cite{bisson} critically used isogenies. On a more practical cryptographic level, computing isogenies from kernel was the key idea behind the post-quantum cryptographic system proposed by de Feo, Jao and Plut \cite{defeo-jao-plut}. Isogenies received considerable attention in computing Hilbert class polynomials and the CM method \cite{lauter-robert}, modular polynomials in genus 2 \cite{milio:hal1} and \cite{milio:hal2} as well as pairings \cite{lubicz-robert:pairings}. 

Theorem~\ref{thm:main} has already been applied to prove a worst case to average case reduction in isogeny classes for the discrete logarithm problem for Jacobians of curves of genus 2 \cite{jetchev-wesolowski}. In addition, it enabled improvements to existing \emph{going-up} algorithms that appear crucial in the CM method in genus 2 for computing Hilbert class polynomials \cite{brooks-jetchev-wesolowski}. 

It is expected that Theorem~\ref{thm:main} can be used to efficiently reduce the discrete logarithm problem in genus 3 from the Jacobian of a hyperelliptic curve to the Jacobian of a quartic curve, thus generalizing the arguments of Smith \cite{smith:genus3}. Since the 
problem on Jacobians of quartic curves is known to be solvable faster than the generic Pollard's rho method, this could allow us to show that the problem is also easier on Jacobians of hyperelliptic curves of genus 3. 

\subsection{Overview and organization of the paper}
We establish the criterion for polarizability of the target abelian variety in Section~\ref{sec:polarizability} (see Proposition~\ref{lem:princpolB}). Section~\ref{sec:theta} reviews basic notions from Mumford's theta theory such as theta groups, theta structures, symmetric and totally symmetric line bundles as well as the main tool for transferring data under isogeny - the isogeny theorem (Theorem~\ref{thm:isog}). We apply 
the theory in Section~\ref{sec:computing_theta_null_point} to first compute the theta null point of the 
target variety for a suitably chosen totally symmetric line bundle and symmetric theta structure on the target variety. This section requires several novel ideas: first of all, in order to apply the isogeny theorem, one needs to work on an $r$-fold product $B^r$ of the target variety $B$ for some suitable $r >1$. One can thus only get the theta coordinates for some theta structure on $B^r$ that is not an $r$-fold product theta structure (i.e., that comes from a single copy of $B$). To remedy this, in Section~\ref{subsec:metaplect} we apply a metaplectic isomorphism to modify appropriately the theta structure and extract theta coordinates for a single copy of $B$. Section~\ref{sec:isogpts} explains how to evaluate the isogeny on points. Section~\ref{sec:complexity} analyzes the complexity of the two algorithms. Finally, we provide an explicit example and discuss our implementation of the algorithm in Section~\ref{sec:examples}.

%
%
\section{Polarizability}\label{sec:polarizability}
Let $A$ be an abelian variety of dimension $g$ over the finite field $k$.

\subsection{Principal polarization on an abelian variety}\label{subsec:principal-polarization}

Write $\Akbar$ for the base change $A \otimes_k \bar{k}$.
The Picard group $\Pic(\Akbar)$ of $\Akbar$ is the group of isomorphism classes of line bundles $\cL$ on $\Akbar$ under the tensor product $\otimes$. Denote by~ $\Pic^0(\Akbar)$ the subgroup of isomorphism classes of degree 0 line bundles. Every line bundle $\cL$ on~$A$ induces a map $\varphi_{\cL} \colon A(\bar{k}) \ra \Pic^0(\Akbar), \, x \mapsto t_x^*\cL \otimes \cL^{-1}$ which is a homomorphism by the theorem of the square. If $\cL$ is ample, then $\varphi_{\cL} \colon A \ra A^\vee$ is an isogeny, where $A^\vee$ is the dual abelian variety, representing isomorphism classes of degree 0 line bundles on $A$. However, not every homomorphism $A \ra A^\vee$ is of the form $\varphi_{\cL}$ for some $\cL \in \Pic(\Akbar)$. For a criterion, see \cite[Prop.16.6]{milne:abvars}.
Recall that two line bundles $\cL_1$ and $\cL_2$ on $A$ are \textit{algebraically equivalent} if there exists $\cL^0 \in \Pic^0(A)$ such that $\cL_2 \isom \cL_1 \otimes \cL^0$. If $\cL_1$ is ample, this is
equivalent to saying that $\cL_2 \isom t_x^* \cL_1$ for some $x \in A(\bar k)$
(indeed~$\cL_1$ is ample, hence $\varphi_{\cL_1} \colon A \to A^\vee$ is an
  isogeny, so the line bundle $\cL^0$ from above can be written as $t_x^*
\cL_1 \otimes \cL_1^{-1}$ for some $x \in A(\bar k)$). Note that the isogeny $\varphi_{\cL} \colon A \ra A^\vee$ depends only on the algebraic equivalence class of $\cL$. The algebraic equivalence classes of line bundles on $A$ form a group under $\otimes$, the N\'eron-Severi group $\NS(A)$. 

A \textit{polarization} $\lambda$ on $A$ is the algebraic equivalence class of an ample line bundle. The polarization is called \textit{principal} if the induced isogeny $A \ra A^\vee$, also denoted by $\lambda$, is an isomorphism. 
If $\lambda$ is a principal polarization then $\lambda$ determines a Rosati involution $\dagger$ on $\End(A)$ defined by  
$\vphi \mapsto \vphi^{\dagger} := \lambda^{-1} \circ \vphi^{\vee} \circ \lambda$. Again consider the subset  
$\End(A)^+ \subset \End(A)$ of endomorphisms of $A$ stable under $\dagger$.
If $\End(A)^+ \otimes \Q$ is a totally real field of degree~$g$ over $\Q$ (the dimension of $A$), we say that $A$
has real multiplication~(RM). This is the case e.g. for an ordinary and simple abelian variety over the finite field $k$.
 
\paragraph{Principal polarizations and totally positive real endomorphisms.}\label{sec:RM}
Suppose that $A$ is ordinary and simple and let $\lambda$ be a principal polarization on $A$. Let $K$ be the CM-field isomorphic to the endomorphism algebra of $A$, with totally real subfield $K_0 \subset K$.
The following result describes all ample line bundles on~$A$ up to algebraic equivalence in terms of totally positive real endomorphisms. 

\begin{prop}\label{prop:ppav}
The map $\cL \mapsto \lambda^{-1} \circ \varphi_{\cL}$ yields an isomorphism of groups $\Phi_{\lambda} \colon \NS(A) \xra{\sim} \End(A)^+$.  
This isomorphism induces a bijection between the polarizations on $A$ and the set $\End(A)^{++} \subset \End(A)^+$ 
of symmetric (for the Rosati involution) totally positive endomorphisms of $A$. Under this bijection, the polarizations of 
degree $d$ correspond to endomorphisms of degree $d^2$. In particular, if $d = 1$ then principal polarizations on $A$ correspond to totally positive symmetric units in $\End(A)$. 
\end{prop}
\begin{proof}
We prove the proposition by considering the canonical lift of $A$ to the ring of Witt vectors, embedding the fraction field of the ring of Witt vectors (in any way) to $\C$ and then reducing the problem to a principally polarized complex abelian variety, for which the proof can be found in \cite[Prop.5.2.1 and Thm.5.2.4]{birkenhake-lange}. 
\end{proof}

\begin{cor}
  \label{cor:endo_polarisation}
  If $\alpha$ is an endomorphism of $A$, then $\alpha^\ast \cL$ is
  algebraically equivalent to $\cL^{\alpha^ \dagger\alpha}$.
\end{cor}


\begin{rem}\label{rem:totsym}
Proposition~\ref{prop:ppav} is not quite
strong since often we will need to rigidify the line
bundles up to isomorphism (linear equivalence) rather than up to algebraic equivalence.
But this rigidification will actually be automatic, since we
will be working with symmetric theta structures on totally symmetric line bundles (see Section \ref{par:sym_theta_structures}). 
A line bundle is \emph{totally symmetric} if and only if it is the square
of a symmetric line bundle, so there is at most one totally symmetric line
bundle in its algebraic equivalence class.
In particular, if there is an isogeny $f \colon (A,\cL) \to (B,\cM)$ such that~$f^\ast \cM$ is algebraically equivalent to $\cL$ and both $\cL$ and $\cM$
are totally symmetric, then $f^\ast \cM$ (which is totally symmetric) is
linearly equivalent to $\cL$.
\end{rem}

\subsection{Principal polarizability via isogenies}\label{subsec:inducedpp}
Let $f \colon A \ra B$ be an isogeny of abelian varieties whose kernel is $\ker(f) = G \subset A(\bar{k})$. Proposition~\ref{prop:ppav} can be used to decide whether 
$B \isom A / G$ is principally polarizable. Assuming that there exists a principal polarization $\cM_0$ on $B$, we would like to compute the theta null point of $\cM = \cM_0^{\otimes n}$ 
for some~$n$ and for a suitably chosen theta structure on $(B, \cM)$. The exact criterion for deciding principal polarizability is summarized in the 
following lemma.

\begin{prop}
  \label{lem:princpolB}
  Let $(A,\cL_0)$ be a principally polarized ordinary and simple abelian variety defined over the
  field $k$. Let $G \subset A(\bar{k})$ be a finite subgroup, and $f \colon A \to
  B=A/G$ the corresponding separable isogeny. Then $B$ admits a principal
  polarization if and only if there exists a totally positive real endomorphism 
  $\beta \in \End(A)^{++}$ such that $G$ is a maximal isotropic subgroup for
  the commutator pairing $e_{\cL_0^\beta}$.
\end{prop}

\begin{proof}
  If $B$ admits a principal polarization $\cM_0$, we apply the proposition to
  $f^\ast \cM_0$, so there exists an
  endomorphism $\beta$ making the following diagram commute:
\begin{equation}
\xymatrix{
A  \ar[rd]_{\varphi_{\cL_0}} & \ar[l]_{\beta} A  \ar[r]^{f} \ar[d]^{\varphi_{f^\ast \cM_0}} &B \ar[d]^{\varphi_{\cM_0}} \\
&  A^\vee &\ar[l]_{f^\vee}B^\vee.
} 
\label{eq:cyclic}
\end{equation}
It is easy to check that $\beta$ is symmetric, it is totally positive because
$\cL_0^\beta \isom f^\ast \cM_0$ is ample (see Proposition~\ref{prop:ppav}), and $G$ is maximal isotropic
inside $\ker(\varphi_{\cL^\beta})$ for degree reasons.

Conversely, given an endomorphism $\beta$ satisfying the conditions of the
lemma, let $\cL_0^\beta$ be an element of the algebraic equivalence class representing the polarization associated to $\beta$. Then $\cL_0^\beta$ descends under $f$ into a polarization
$\cM_0$ by descent theory \cite[Prop.2.4.7]{drobert:thesis}, and the following
diagram is commutative
\begin{equation}
\xymatrix{
A  \ar[rd]_{\varphi_{\cL_0}} & \ar[l]_{\beta} A  \ar[r]^{f} \ar[d]^{\varphi_{\cL_0^\beta}} &B \ar[d]^{\varphi_{\cM_0}} \\
&  A^\vee &\ar[l]_{f^\vee}B^\vee.
} 
\label{diag:cyclic}
\end{equation}
So $f^*(\cM_0)$ is algebraically equivalent to $\cL_0^{\beta}$; moreover since the degree of $\beta$ is equal to \linebreak $\deg f  \cdot \deg f^\vee=(\deg f)^2$, then
the degree of $\varphi_{\cM_0}$ is 1 and hence, $\cM_0$ is a principal polarization. 
\end{proof}

\begin{rem}
If $G$ is cyclic of order $\ell$, we will apply the above lemma to elements $\beta \in \End(A)^{++}$ that are real endomorphisms of degree $\ell^2$. Then $G$ is automatically a maximal isotropic subgroup for the commutator pairing $e_{\cL_0^\beta}$.

 In dimension~$2$, we have that $\End(A)^+$ is a real quadratic order. Again, for $G$ cyclic of prime order $\ell$, there is a
 principal polarization on $A/G$ if and only if there exists a totally
 positive real endomorphism $\beta \in \End(A)^{++}$ of $K_0 / \Q$-norm $\ell$ such that $G \subset A[\beta]$ (which is then automatically maximal isotropic for $e_{\cL_0^\beta}$). In particular, $(\ell)$ splits into $(\beta)(\beta^c)$ in $\End(A)^+$ (where $\beta^c$ is the real conjugate of $\beta$), and it is easy to see that this can happen only if $\End(A)^+$ is maximal locally at~$\ell$.
  In particular, unless the real multiplication is locally maximal at
  $\ell$, there is no cyclic isogeny of degree $\ell$ between $A$ and
  another principally polarized abelian variety.

  We remark that while an isogeny between (non-polarized) abelian varieties
  can always be written as the composition of cyclic isogenies of prime
  degrees, Proposition~\ref{lem:princpolB} shows that this is not true if we
  require the abelian varieties to be principally polarized.
\end{rem}

%
%
\section{Theta Coordinates and Theta Structures}\label{sec:theta}
An abelian variety $A$ together with a polarization $\cL$ is not sufficient to get a canonical projective embedding of $A$. One needs to add more data in order to single out canonical coordinates. In this section, we recall the notion of a theta structure and explain how it yields canonical theta coordinates. We discuss symmetric theta structures and recall how canonical theta coordinates are related under isogenies (the isogeny theorem). Finally, we explain how to choose an appropriate
symmetric theta structure on the original abelian variety given to our algorithm. For this section we fix $k$ an algebraically closed field of positive characteristic. 

\subsection{The theta group and theta structures}\label{subsec:thetagroup}
Let $A$ be an abelian variety over $k$.
Following \cite{mumford:eq1}, let $\cL$ be an ample line bundle on $A$ and let~$K(\cL)$ be the kernel of the polarization isogeny $\vphi_{\cL} \colon A \ra A^\vee,\, x \mapsto t_x^*\cL \otimes \cL^{-1}$. For any element $x$ of $K(\cL)$, the line bundles $\cL$ and $t_x^*\cL$ are isomorphic, but not necessarily in a unique way.

\paragraph{Mumford's theta group.} 
The Mumford theta group is
$$
\cG(\cL) = \{(x, \phi_x) \colon x \in K(\cL), \phi_x \colon \cL \xra{\sim} t_x^* \cL \}, 
$$
under the group law 
$(x,\phi_x) \cdot (y,\phi_y)=(x+y,t_x^*\phi_y \circ \phi_x)$ for $(x,\phi_x), (y,\phi_y) \in \mathcal G(\cL)$.   
The inverse of~$(x, \phi_x)$ under this group law is $(-x, t_{-x}^* \phi_x^{-1})$. We have an exact sequence 
$$
0 \ra k^\times \ra \cG(\cL) \ra K(\cL) \ra 0, 
$$ 
where $\alpha \in k^\times \mapsto (0, \alpha) \in \cG(\cL)$ (here, $\alpha \colon \cL \xra{\sim} t_0^*\cL = \cL$ is the multiplication-by-$\alpha$ automorphism of~$\cL$) and $\cG(\cL) \ra K(\cL)$ is the forgetful map $(x, \phi_x) \mapsto x$. 

\paragraph{Theta structures.}\label{par:thetastruct} The ample line bundle $\cL$ on $A$ gives rise to a non-degenerate symplectic form $e_{\cL}$ on $K(\cL)$ defined as follows: for any $x, y \in K(\cL)$, let $\widetilde{x}, \widetilde{y} \in \cG(\cL)$ be arbitrary lifts and set $e_{\cL}(x, y) := \widetilde{x} \widetilde{y} \widetilde{x}^{-1} \widetilde{y}^{-1}$. As lifts of elements of $K(\cL)$ are defined up to scalars, and since $k^\times$ is the center of $\cG(\cL)$, the form $e_{\cL}$ is well defined. Moreover, $\widetilde{x} \widetilde{y} \widetilde{x}^{-1} \widetilde{y}^{-1}$ being in the kernel of $\cG(\cL) \ra K(\cL)$, we can see $e_{\cL}(x, y)$ as an element of $k^\times$. We call $e_{\cL}$ the commutator pairing. If $K(\cL) = K_1(\cL) \oplus K_2(\cL)$ is a symplectic decomposition of $K(\cL)$ with respect to the commutator pairing $e_{\cL}$ (here, $K_i(\cL)$ are maximal isotropic subspaces of $K(\cL)$), then $K_1(\cL) \isom \Z(\delta) := \bigoplus_{i = 1}^g \Z / \delta_i \Z$ where $\delta_1 \mid \delta_2 \mid \dots \mid \delta_g$ are the elementary divisors of $K_1(\cL)$. We then say that the polarization (ample line bundle) $\cL$ is of type $\delta = (\delta_1, \dots, \delta_g)$. 

Given a tuple $\delta = (\delta_1, \dots, \delta_g) \in \Z^g$ with $\delta_1 \mid \dots \mid \delta_g$, let $K(\delta) = \Z(\delta) \oplus \widehat{\Z}(\delta)$ where $\Zhat(\delta) = \Hom(\Z(\delta), k^\times)$. Then $K(\delta)$ is equipped with the standard pairing $e_\delta$ coming from the duality, i.e., 
$$
e_\delta((x_1, y_1), (x_2, y_2)) = \frac{y_2(x_1)}{y_1(x_2)} \in k^\times. 
$$ 
Moreover, let $\cH(\delta)$ be the Heisenberg group, that is, the group whose underlying set is $k^\times \times K(\delta)$ and whose group law is 
\begin{equation}\label{def:Heisenberg}
(\alpha_1, x_1, y_1) \cdot (\alpha_2, x_2, y_2) = (\alpha_1 \alpha_2  y_2(x_1), 
x_1 + x_2, y_1 + y_2), \qquad \forall \alpha_i \in k^\times, x_i \in \Z(\delta), y_i \in \Zhat(\delta).   
\end{equation}
Note that the Heisenberg group fits into the exact sequence
$$0 \ra k^\times \ra \cH(\delta) \ra K(\delta) \ra 0,$$
where $\alpha \in k^\times \mapsto (\alpha, 0, 0)$ and $(\alpha, x, y) \in \cH(\delta) \mapsto (x,y)$. Also, note that the inverse of $(\alpha, x, y) \in \cH(\delta)$ is given by $(\alpha^{-1} y(x), -x, -y)$.

A \textit{theta structure} $\Theta_{\cL}$ of type $\delta$ is an isomorphism of central extensions 
$$
\Theta_{\cL} \colon \cH(\delta) \xra{\sim}  \cG(\cL). 
$$
One can show that the pairing $e_\delta$ on $K(\delta)$ is induced from the commutator pairing on 
$\cH(\delta)$ and thus, the induced isomorphism $\overline{\Theta}_{\cL} \colon K(\delta) \xra{\sim} K(\cL)$ which makes the following diagram commutative
\[
\xymatrix{
0 \ar[r] & k^\times \ar[r] \ar[d]_{\text{id}} & \cH(\delta) \ar[d]_{\Theta_{\cL}} \ar[r] & K(\delta) \ar[r] \ar[d]_{\overline{\Theta}_{\cL}} & 0 \\
0 \ar[r] & k^\times \ar[r] & \cG(\cL) \ar[r] & K(\cL) \ar[r] & 0    
}
\]
is a symplectic isomorphism since $\Theta_{\cL}$ pulls back commutators.
There is a canonical map (of sets) $s_\delta \colon K(\delta) \ra \cH(\delta)$ given by $(x, y) \mapsto (1, x, y)$ that is a homomorphism of groups when restricted to $K_1(\delta) := \Z(\delta)$ and $K_2(\delta) := \Zhat(\delta)$. Via the theta structure, 
these homomorphisms yield two sections $s_{K_1(\cL)} \colon K_1(\cL) \ra \cG(\cL)$ and $s_{K_2(\cL)} \colon K_2(\cL) \ra \cG(\cL)$, where $K_i(\cL) := \overline{\Theta}_{\cL} (K_i(\delta))$ for $i = 1, 2$. Conversely \cite[Prop. 3.3.3]{drobert:thesis}, given a symplectic isomorphism $\overline{\Theta}_{\cL} \colon K(\delta) \ra K(\cL)$ together with two (group) sections $s_{K_1(\cL)}$, $s_{K_2(\cL)}$ we get a unique theta structure $\Theta_{\cL}$ above $\overline{\Theta}_{\cL}$ inducing these two sections. In particular, for any symplectic isomorphism $\overline{\Theta}_{\cL} \colon K(\delta) \ra K(\cL)$ (equivalently, a Witt basis for $K(\cL)$ with respect to $e_{\cL}$), there is a theta structure $\Theta_{\cL}$ above that isomorphism.  This follows from \cite[Prop. 3.2.6]{drobert:thesis} : the projection $\cG(\cL) \ra K(\cL)$ admits a section above $K \subset K(\cL)$ if and only if $K$ is isotropic for the pairing $e_{\cL}$.

\paragraph{Totally symmetric line bundles and symmetric theta structures.}\label{par:sym_theta_structures}
Let $(A, \cL)$ be a polarized abelian variety. A priori, there is no way to fix a particular choice of a representative in the algebraic equivalence class of $\cL$. One way to do that is to introduce the notions of symmetric and totally symmetric line bundles. A line bundle $\cL$ on $A$ is called \emph{symmetric} if $[-1]^* \cL \isom \cL$.  

Suppose now that $\cL$ is a symmetric line bundle and fix an isomorphism $\psi \colon \cL \xra{\sim} [-1]^* \cL$ (such an isomorphism is unique up to a scalar in $k^\times$). This means that we have isomorphisms on fibers $\psi(x) \colon \cL(x) \xra{\sim} [-1]^*\cL(x) = \cL(-x)$. We assume that $\psi$ is normalized in the sense that $\psi(0) \colon \cL(0) \xra{\sim} \cL(0)$ is the identity map (otherwise rescale $\psi$). If $x \in A[2]$ then $\cL(-x) = \cL(x)$ and hence, $\psi$ is given on $\cL(x)$ by multiplication by a scalar $e^{\cL}_*(x) \in k^\times$. Note that $e^{\cL}_*(x) =\pm 1$ for all $x\in A[2]$. We call the line bundle $\cL$ \emph{totally symmetric} if $e^{\cL}_*(x) =1$ for all $x \in A[2]$. The notion of totally symmetric line bundles is useful for making a canonical choice of an isomorphism class of line bundles within an algebraic equivalence class. More precisely, if $2 \mid \delta_1, \dots, \delta_g$ where $\delta = (\delta_1, \dots , \delta_g)$ is the type of $\cL$, then there exists a unique totally symmetric line bundle in the algebraic equivalence class of $\cL$ \cite[Prop.4.2.4]{drobert:thesis}. 

Suppose now that $\cL$ is symmetric and let $\psi \colon \cL \xra{\sim} [-1]^* \cL$ be the isomorphism of $\cL$ with $[-1]^* \cL$. Assume that $\psi$ is normalized, i.e. the restriction of $\psi$ to the fiber $\cL(0)$ of 0 is the identity.   
We then have an automorphism $\gamma_{-1} \colon \cG(\cL) \ra \cG(\cL)$ given by 
$$
\gamma_{-1} (x, \phi) = \left (-x, (t_{-x}^* \psi)^{-1} \circ ([-1]^* \phi) \circ \psi \right ). 
$$ 
In addition, we have a metaplectic automorphism $\gamma_{-1} \colon \cH(\delta) \ra \cH(\delta)$ of the Heisenberg group 
given by $\gamma_{-1} (\alpha, x, y) = (\alpha, -x, -y)$. A theta structure $\Theta_{\cL}$ on $(A, \cL)$ is called \textit{symmetric} if $\gamma_{-1} \circ \Theta_{\cL} = \Theta_{\cL} \circ \gamma_{-1}$. Suppose $\cL$ is a totally symmetric line bundle on $A$ of type $\delta$. Then, according to \cite[Remark 2, p.318]{mumford:eq1}, every symplectic isomorphism $\overline{\Theta} \colon K(\delta) \to K(\cL)$ is induced by a symmetric theta structure $\Theta_{\cL} \colon \cH(\delta) \to \cG(\cL)$.

The reason why this notion will be useful (see the key application in Section~\ref{par:unfolding}) is the following result proved in \cite[Prop.4.3.1]{drobert:thesis}: 

\begin{prop}\label{prop:symtheta}
Let $\cL$ be a totally symmetric line bundle on $A$ of type $\delta$. Let $\overline{\Theta} \colon K(\delta) \ra K(\cL)$ be a symplectic isomorphism. In order to fix a symmetric theta structure on $(A, \cL)$ that induces $\overline{\Theta}$ it suffices to fix a symplectic isomorphism $K(2\delta) \ra K(\cL^2)$ that restricts to $\overline{\Theta}$ on $K(\delta)$.
\end{prop}

Note that for the above proposition, we identified $K(\delta)$ with a subgroup of $K(2\delta)$ in the following way. The elements $(x_1,\dots, x_g)\in \Z(\delta) \subset K(\delta)$ are sent to $(2x_1, \dots , 2x_g) \in \Z(2\delta)$, whereas for each $y \in \widehat{\Z}(\delta) \subset K(\delta)$ there exists a unique $y' \in \widehat{\Z}(2\delta)$ such that $y'(x) = y(2x)$ for all $x \in K(2\delta)$.

\subsection{Theta coordinates} 
Assume now that the line bundle $\cL$ on $A$ is very ample. This means that each $k$-basis of the space of global sections $\Gamma(A,\cL)$ yields a projective embedding $A \hra \mathbf{P}^{d-1}_k$, where $d = \deg \cL$. This embedding is only defined up to the action of $\PGL_d(k)$. In order to fix an embedding, we need to fix canonical coordinates on $A$ (i.e. a canonical basis for $\Gamma(A, \cL)$). This choice will come precisely from the choice of a theta structure.  
Once we have fixed a theta structure $\Theta_{\cL}$ on $(A, \cL)$, one gets 
canonical theta functions $\{\theta_i^{\Theta_{\cL}}\}_{i \in K_1(\delta)}$ forming a basis for the space of global sections 
$\Gamma(A, \cL)$ and a canonical theta null point $(\theta^{\Theta_{\cL}}_i(0))_{i \in K_1(\delta)}$ of  $(A, \cL)$. We will get this canonical basis via the representation theory of 
the Heisenberg group $\cH(\delta)$. 

Recall that the Heisenberg group $\cH(\delta)$ has a natural irreducible representation on 
the space $V(\delta)$ of $k$-valued functions on $K_1(\delta)$ that is given by  
$$
((\alpha, x, y) \cdot f) (z) =  \alpha y(z) f(z + x). 
$$
One can show that any representation $V$ of $\cH(\delta)$ with a natural action of $k^\times$ (as in \cite[Prop.2]{mumford:eq1}) is a direct sum of $r$ copies of $V(\delta)$ where $r = \dim_k V^{\widetilde{K}}$, for $\widetilde{K} \subset \cH(\delta)$ any maximal level subgroup. 
To explain why the theta structure yields a canonical embedding, note that the Mumford theta group $\cG(\cL)$ acts on the space $V = \Gamma(A, \cL)$, where the action is given by 
$$
(x, \phi) \cdot s = t_{-x}^* \phi(s). 
$$
One can show that this action is irreducible \cite[Prop.3.4.3]{drobert:thesis}, and one easily sees that $k^\times \hookrightarrow \cG(\cL)$ acts in the natural way. It follows that the 
theta structure $\Theta_{\cL}$ determines a unique (up to a scalar multiple) $\cH(\delta)$-equivariant isomorphism $\varphi \colon V(\delta) \ra \Gamma(A, \cL)$. Since $V(\delta)$ has a canonical basis $\{\gamma_i\}_{i \in K_1(\delta)}$ given by
$$
K_1(\delta) \ni j \mapsto \gamma_i (j) = 
\begin{cases}
1 & \text{ if } i = j, \\
0 & \text{otherwise,}
\end{cases}
$$
the theta structure $\Theta_{\cL}$ yields a canonical basis $\left \{\theta_i^{\Theta_{\cL}} :=  \varphi(\gamma_i) \mid i \in K_1(\delta)\right \}$ for $\Gamma(A, \cL)$, up to scalar multiples.

If we summarise the above we get: Let $(A, \cL, \Theta_{\cL})$ be a polarized abelian variety with theta structure and suppose $\cL$ is very ample and of type $\delta = (\delta_1, \dots, \delta_g)$. Let $\left \{\theta_i^{\Theta_{\cL}} =  \varphi(\gamma_i) \mid i \in K_1(\delta)\right \}$ be the basis of $\Gamma(A, \cL)$ induced by the theta structure $\Theta_{\cL}$ and fix once and for all an ordering $a_1, \dots, a_d $ of the elements of $K_1(\delta)$, where $d = \# K_1(\delta) = \delta_1\cdots \delta_g$. This ordering determines an ordering $\theta_1^{\Theta_{\cL}}, \dots , \theta_d^{\Theta_{\cL}}$ of the theta functions, which then yields an embedding
$$A \hookrightarrow \P_k^{d-1}, \, x \mapsto (\theta_1^{\Theta_{\cL}}(x) : \cdots :  \theta_d^{\Theta_{\cL}}(x)).$$

Since we always keep the same ordering of the elements of $K_1(\delta)$, we will subsequently write the embedding as $x \mapsto (\theta_i^{\Theta_{\cL}}(x))_{i \in K_1(\delta)}$. Moreover, via the symplectic isomorphism $\overline{\Theta}_{\cL}$, we may consider indexing the theta functions $\theta_i^{\Theta_{\cL}}$ by $K_1(\cL)$.

\paragraph{The isogeny theorem.}\label{subsec:isogThm}
We now recall a theorem (see \cite[\S 3.6]{drobert:thesis} or \cite[\S 6.5]{birkenhake-lange}) that relates theta coordinates on two isogenous polarized abelian varieties with theta structure. Suppose that $(A,\cL)$ and $(B, \cM)$ are two polarized abelian varieties over a field $k$ and let $f \colon (A, \cL) \ra (B, \cM)$ be a separable $k$-isogeny of polarized abelian varieties, i.e. $f^*\cM$ is linearly equivalent to $\cL$. Let~$G \subset A$ be the finite kernel of $f$ so that $B \isom A /G$. It is explained in \cite[\S 3.6]{drobert:thesis} what it means for two theta structures $\Theta_{\cL}$ and $\Theta_{\cM}$ to be compatible with respect to the isogeny $f$. One condition is that the symplectic decomposition on $K(\cL)$ induced by $\Theta_{\cL}$ is $f$-compatible with the symplectic decomposition on $K(\cM)$ induced by $\Theta_{\cM}$  in the sense that $f(K_i(\cL)) \cap K(\cM) = K_i(\cM)$ for $i = 1,2$. 

The polarized abelian varieties with theta structure $(A,\cL, \Theta_{\cL})$ and $(B, \cM, \Theta_{\cM})$ yield~canonical embeddings of the varieties $A$ and $B$ respectively into projective space, provided $\cL$ and $\cM$ are very ample. The following theorem tells us how to compute the canonical theta coordinates of $f(x)$ out of the canonical theta coordinates of $x$ for $x \in A(\bar{k})$.

\begin{thm}\label{thm:isog} 
Let $f \colon (A,\cL, \Theta_{\cL}) \to (B, \cM, \Theta_{\cM})$ be an isogeny of polarized abelian varieties with theta structure. There exists $\lambda\in \bar{k}^\times$ such that for all $i\in K_1(\cM)$ and $x\in A(\bar{k})$, we have 
\begin{equation}\label{eq:isogthm}
\displaystyle \theta_i^{\Theta_{\cM}}(f(x))=\lambda \cdot \sum_{\substack{j \in K_1(\cL) \\ f(j) = i}}\theta_j^{\Theta_{\cL}}(x). 
\end{equation}
\end{thm}

In addition, we can state an affine version of the isogeny theorem as follows: suppose that we have fixed affine coordinates on both $\widetilde{A}$ and $\widetilde{B}$ (i.e. we chose $f$-compatible very ample line bundles and theta structures on $A$ and $B$ respectively and see the closed points via the canonical embeddings). Then there exists a lifting $\widetilde{f} \colon \widetilde{A} \ra \widetilde{B}$ 
such that the following diagram is commutative: 
\[
\xymatrix{
\widetilde{A} \ar[r]^{\widetilde{f}} \ar[d]^{p_A} & \widetilde{B} \ar[d]^{p_B} \\ 
A \ar[r]^{f} & B  \\
}
\]

\paragraph{The action of the Heisenberg group on theta coordinates.}\label{subsubsec:canaffine}

Suppose we are given $(A, \cL, \Theta_{\cL})$ with~$\cL$ very ample, which determines the embedding $A \hookrightarrow \P_k^{d-1}$, $x \mapsto(\theta_i^{\Theta_{\cL}}(x))_{i \in K_1(\delta)}$. Let $p \colon \mathbf{A}_k^d \backslash \{0\} \ra \mathbf{P}_k^{d-1}$ be the natural projection map and consider the affine cone $\widetilde{A} = p^{-1}(A)$. The action of $\cG(\cL)$ on $\Gamma(A, \cL)$ induces an action of $\cG(\cL)$ on $A$ in terms of theta coordinates, which is above the translation by elements of $K(\cL)$. To be more precise, given $x\in A$ and $\widetilde{x} \in \mathbf{A}^d(\bar{k})$ above $(\theta_i^{\Theta_{\cL}}(x))_{i \in K_1(\delta)} \in \P^{d-1}(\bar{k})$, then for $(w, \phi_w) \in \cG(\cL)$, the element $(w, \phi_w) \cdot \widetilde{x}$ is an affine lift of $x+w$. 

Let us give a description of this action for $\cH(\delta)$ (via the theta structure $\Theta_{\cL}$). Let $(\alpha, x, y) \in \cH(\delta)$ and suppose it is mapped to $(w,\phi_w)$ via $\Theta_{\cL}$. Consider $\theta_i^{\Theta_{\cL}}$ for $i \in K_1(\delta)$. \cite[Prop.6.4.2.]{birkenhake-lange} describes the action of $(\alpha, x, y)$ on $\theta_i^{\Theta_{\cL}}$ above the translation by $-w$, and by the slight adaption $(\alpha, x, y) \leftrightarrow (\alpha, -x, -y)$ we obtain an action above the translation by $w$, which is given by
\begin{equation}\label{eq:actionHeisenberg}
(\alpha, x, y) \cdot \theta_i^{\Theta_{\cL}} = \alpha e_\delta((i+x,0),(0,-y))\theta_{i+x}^{\Theta_{\cL}} = \alpha y(-i-x) \theta_{i+x}^{\Theta_{\cL}}.
\end{equation}
Hence we see that translation by an element of $K_1(\cL)$ acts on $\widetilde{A}$ as permutation of the theta coordinates, whereas  translation by an element of $K_2(\cL)$ acts on $\widetilde{A}$ as dilatation of the theta coordinates. This proves the following.

\begin{lem}\label{lem:canaffine}
A choice of an affine lift $\widetilde{x}$ of an element $x = (\theta_i^{\Theta_{\cL}}(x))_{i \in K_1(\delta)}$ gives a section of the projection $\widetilde{A} \to A $ above $x+K(\cL)$. That is to say, once an affine lift $\widetilde{x}$ of $x$ is fixed, the action of the Heisenberg group $\cH(\delta)$ on $\widetilde{A}$ determines a lift above each $x + w$, for $w \in K(\cL)$.
\end{lem}
\noindent Note that for the above lemma, we used the canonical embedding of $K_i(\delta) \hookrightarrow \cH(\delta)$, for $i = 1,2$.

\paragraph{Product line bundles and product theta structures.}\label{par:rfold}
Let $(A, \cL, \Theta_{\cL})$ be a polarized abelian variety of type $\delta = (\delta_1, \dots, \delta_g)$ with theta structure, and let $r\ge 1$ be a nonnegative integer . There is a natural polarization $\cL ^{\star r}$ on $A^r$ defined by 
$$
\cL^{\star r} = p_1^*\cL \otimes \cdots \otimes p_r^* \cL,
$$
where $p_i \colon A^r \ra A$ is the projection of the $i$-th factor of $A^r$ for $i = 1, \dots , r$. A polarization (or ample line bundle) $\cL$ on the variety $A^r$ is called a \textit{product polarization} if $\cL$ is isomorphic to $p_1^* \cL_{A} \otimes \cdots \otimes p_r^* \cL_{A}$ for some polarization $\cL$ on $A$. According to \cite[Lem.1, p.323]{mumford:eq1} we have 
$$\cG(\cL^{\star r}) \cong \cG(\cL)^{\times r} / \{ (\alpha_1, \dots, \alpha_r) : \alpha_i \in k^\times \subset \cG(\cL), \alpha_1 \cdots \alpha_r = 1 \},$$ where
$((x_1, \varphi_1 \colon \cL \xrightarrow{\sim} t_{x_1}^*\cL), \dots , (x_r, \varphi_r \colon \cL \xrightarrow{\sim} t_{x_r}^*\cL))$ is mapped to \\ $((x_1,\dots, x_r), p_1^*\varphi_1 \otimes \cdots \otimes p_r^* \varphi_r \colon \cL^{\star r} \xrightarrow{\sim} t_{(x_1, \dots, x_r)}^* \cL^{\star r})$.

The type $\delta^{\star r}$ of $\cL^{\star r}$ is easily seen to be
$$
\delta^{\star r} = \left (\underbrace{\delta_1, \dots, \delta_1}_r, \underbrace{\delta_2, \dots, \delta_2}_r, \dots, \underbrace{\delta_g, \dots, \delta_g}_r \right ) \in \Z^{gr}, 
$$
since $\Z(\delta^{\star r}) \cong \Z(\delta)^r$, and $K(\delta^{\star r})  \cong K(\delta)^r$ is equipped with the symplectic pairing
$$e_{\delta^{\star r}}((z_1, \dots , z_r), (z_1', \dots , z_r')) = e_\delta(z_1, z_1') \cdots e_\delta(z_r, z_r') \in k^\times.$$ 
The Heisenberg group $\cH(\delta^{\star r})$ is then defined as in (\ref{def:Heisenberg}). The theta structure $\Theta_{\cL} \colon \cH(\delta) \to \cG(\cL)$ induces in a natural way a $k^\times$-isomorphism
$$\Theta_{\cL^{\star r}} \colon \cH(\delta^{\star r}) \to \cG(\cL^{\star r}),$$
given by
$$(\alpha, (x_1, y_1), \dots, (x_r, y_r)) \mapsto (\alpha \cdot \Theta_{\cL}(1,x_1,y_1), \dots, \Theta_{\cL}(1, x_r, y_r))$$
(actually, we can put the scalar $\alpha$ in any coordinate). The canonical coordinates for the $r$-fold product theta structure 
$\Theta_{\cL^{\star r}}$ are simply given by 
\begin{equation}
\theta_{\mathbf{i}}^{\Theta_{\cL^{\star r}}}(\mathbf{x}) = \theta_{i_1}^{\Theta_{\cL}}(x_1) \cdots  
\theta_{i_r}^{\Theta_{\cL}}(x_r), 
\end{equation}
where $\mathbf{i} = (i_1, \dots, i_r) \in K_1(\cL)^r = K_1(\cL^{\star r})$ and $\mathbf{x} = (x_1, \dots, x_r) \in A^r(\bar{k})$. 

We call a theta structure on $(A^r, \cL^{\star r})$ an $r$-fold product theta structure if it arises via the above construction for some polarization $\cL$ on $A$. Note that $r$-fold product theta structures are key for our algorithm as they will allow us to deduce data about the polarized abelian variety $(B, \cM)$ from data about the polarized abelian $r$-fold product abelian variety $(B^r, \cM^{\star r})$. The following lemma will be useful in the sequel.

\begin{lem}\label{lem:productthetastructure}
A theta structure $\Theta^{\star r} \colon \cH(\delta^{\star r}) \ra \cG(\cL^{\star r})$ is of product form if and only if the induced symplectic isomorphism $\overline{\Theta}^{\star r} \colon K(\delta^{\star r}) \ra K(\cL^{\star r})$ is of product form.
\end{lem}

\begin{proof}
Let $\Theta^{\star r} \colon \cH(\delta^{\star r}) \ra \cG(\cL^{\star r})$ be a theta structure such that the induced symplectic isomorphism $\overline{\Theta}^{\star r} \colon K(\delta^{\star r}) \ra K(\cL^{\star r})$ is of product form. Denote by $\overline{\Theta} \colon K(\delta) \to K(\cL)$ the restriction of $\overline{\Theta}^{\star r}$ to a single factor and suppose $(x,y) \in K(\delta)$ is mapped via $\overline{\Theta}$ to $z \in K(\cL)$. Then $\Theta^{\star r}$ must send
$$(1,(x,y),\dots,(x,y)) \mapsto ((z,\varphi_1),\dots, (z,\varphi_r)),$$
where the isomorphisms $\varphi_1,\dots, \varphi_r$ satisfy $\varphi_2 = \alpha_2 \cdot \varphi_1, \dots, \varphi_r = \alpha_r \cdot \varphi_1$, with $\alpha_2, \dots , \alpha_r \in k^\times$. Define $\Theta \colon \cH(\delta) \to \cG(\cL)$ pointwise by 
$$(1,x,y) \mapsto (z, \sqrt[r]{\alpha_2}\cdots\sqrt[r]{\alpha_r} \cdot \varphi_1).$$
One carefully checks that $\Theta$ is a $k^\times$-isomorphism and that $\Theta^{\star r}$ is equal to the $r$-fold product of~$\Theta$.
\end{proof}

%
%
\section{Computing a Theta Null Point for the Target Abelian Variety}\label{sec:computing_theta_null_point}
Let $(A, \cL_0)$ be a principally polarized ordinary and simple abelian variety of dimension $g$ over the finite field $k$. For what follows we set $n = 2$ or $n = 4$. Let $\cL$ be the totally symmetric line bundle in the algebraic equivalence class of $\cL_0^{\otimes n}$ and let $\Theta_\cL$ be a symmetric theta structure on $(A, \cL)$. Let $\ell$ be an odd prime number different from the characteristic of $k$. Let $\beta \in \End(A)^{++}$ be a totally positive symmetric endomorphism of degree $\ell^2$ and let $G \subset \ker(\beta)$ be a $\Gal(\bar{k} / k)$-stable cyclic subgroup of order $\ell$.
We now compute a theta null point for the target abelian variety $(B = A/G, \cM)$, where~$\cM_0$ is the induced principal polarization on $B$ from Section~\ref{subsec:inducedpp} and $\cM = \cM_0^{\otimes n}$. To do that, we apply Theorem~\ref{thm:isog} to the $\beta$-contragredient isogeny $\widehat{f} \colon (B, \cM^\beta) \ra (A, \cL)$ to express the canonical coordinates of $(B, \cM^\beta)$ (with respect to a compatible theta structure $\Theta_{\cM^\beta}$) in terms of the canonical coordinates of $(A, \cL, \Theta_{\cL})$ and then use one more time the isogeny theorem for a suitably chosen isogeny of polarized abelian varieties with theta structures 
$$
F \colon (B^r, (\cM^{\beta})^{\star r}, \Theta_{(\cM^{\beta})^{\star r}}) \ra (B^r, \cM^{\star r}, \widetilde{\Theta}_{\cM^{\star r}}) 
$$
to obtain the canonical coordinates of $(B^r, \cM^{\star r}, \widetilde{\Theta}_{\cM^{\star r}})$. Since the target theta structure $\widetilde{\Theta}_{\cM^{\star r}}$ is not necessarily a product theta structure, we 
cannot a priori use it to recover a theta null point for~$(B, \cM)$. We thus convert it first to a theta structure 
of the form $\Theta_{\cM} \star \cdots \star \Theta_{\cM}$ ($r$-times) for a theta structure $\Theta_{\cM}$ on $(B, \cM)$ via a suitable choice of a metaplectic automorphism (an automorphism of the corresponding Heisenberg group) and then use a transformation formula for the theta constants to recover the theta null point for $(B, \cM, \Theta_{\cM})$.   

\subsection{Applications of the isogeny theorem}\label{subsec:isogthmapp}
 Starting from Theorem~\ref{thm:isog}, we wish to use the theta null point of $A$ for the symmetric theta structure~$\Theta_{\cL}$ on $(A, \cL)$ to compute the theta null point of $B$ for a suitably chosen symmetric theta structure $\Theta_{\cM}$ on $(B, \cM)$.  One possible way is to first compute the theta null point of $A$ for a symmetric theta structure $\Theta_{\cL^\beta}$ on $(A, \cL^\beta)$ in order to apply the isogeny theorem to $f \colon (A, \cL^\beta) \ra (B, \cM)$ which will give us the theta null point of $B$ for a symmetric theta structure $\Theta_{\cM}$ compatible with $\Theta_{\cL^\beta}$ (via~$f$). There are two major problems with this approach: 

\begin{enumerate}
\item There is no obvious isogeny of polarized abelian varieties between $(A, \cL^\beta)$ and $(A, \cL)$ available\footnote{One might misleadingly think that $\beta \colon A \ra A$ is such an isogeny. Yet, the degree of $\beta^* \cL$ is 
$n^g\deg(\beta) = n^g \ell^2$ whereas the degree of $\cL^\beta$ is $n^g \ell$, so $\beta^* \cL$ is not isomorphic to $\cL^\beta$}. 

\item Even if one knows such an isogeny, the isogeny theorem applied to $(A, \cL^\beta) \ra (A, \cL)$ expresses the theta coordinates for $\Theta_{\cL}$ as polynomials in the theta coordinates for $\Theta_{\cL^\beta}$, thus requiring one to solve a polynomial system which may be expensive in practice.  
\end{enumerate}

\paragraph{Addressing i).}\label{par:issue1}
To address i), one may try to express $\beta$ as $u \bar{u}$ for some $u \in \End(A)$, where $\bar u$ is the $K/K_0$-conjugate of $u$, and then apply the isogeny theorem for $u \colon A \ra A$ instead. In this case,~$u^* \cL$ will be algebraically equivalent to $\cL^\beta$ and hence, we will have an isogeny $u \colon (A, \cL^\beta) \ra (A, \cL)$ of polarized abelian varieties. The problem is that
$\beta$ need not be in the image of the norm map $N_{K/K_0} \colon \cO_K \ra \cO_{K_0}$ and even if it were (i.e. $\beta = u \bar{u}$ for some $u \in \cO_K$), such a $u$ need not be an easily computable endomorphism
of $A$. 

Instead, we use an idea appearing in \cite{CossetRobert} and motivated by Zahrin's trick \cite[Rem.16.12]{milne:abvars} used to show that for any abelian variety $A$, the abelian variety $(A \times A^\vee)^4$ is principally polarizable. Note that for any integer $r \geq 1$, $\beta$ induces an endomorphism 
$\beta^r \colon A^r \ra A^r$ of the $r$-fold product~$A^r$. Choosing $r > 1$ allows us to search for a matrix $F \in M_r(\End(A))$ that satisfies $\beta^r = \bar{F} F$. If for example $r=4$, one knows by \cite{siegel} that any totally positive element of $\cO_{K_0}$ is a sum of 4 algebraic numbers in the same field, i.e. there exist $\alpha_1, \dots, \alpha_4 \in K_0$ such that 
$\beta = \alpha_1^2 + \dots + \alpha_4^2$. 
In general, the $\alpha_i$'s need not be integral and hence, need not be in
$\End(A)^+$. Yet, assuming that they yield endomorphisms of the $\beta$-torsion and the 
$n$-torsion points (i.e. the denominators are coprime to $n\ell$), one can take $F$ to be the matrix
corresponding to multiplication by $\alpha_1 + \alpha_2 i + \alpha_3 j + \alpha_4 k$ on the
Hamilton quaternions over $K_0$ and observe that ${F}^t F = \bar{F} F = \beta I_4$. By looking at
the corresponding isogeny $F \colon A^4 \ra A^4$, we see that $F^* \cL^{\star 4}$ is algebraically equivalent to $(\cL^\beta)^{\star 4}$ and hence, one can apply the isogeny theorem for $F \colon (A^4, (\cL^\beta)^{\star 4}) \ra (A^4, \cL^{\star 4})$. But this reduces to the same problem as in ii). 

\paragraph{Addressing ii).}\label{par:issue2} We can avoid solving a polynomial system by using the $\beta$-contragredient isogeny $\widehat{f} \colon B \ra A$ in an appropriate way, as we will see in the next section. Therefore we first consider the following lemma.

\begin{lem}\label{lem:mbeta}
Let $\cM_0$ be the (induced) principal polarization on $B$ defined in Section~\ref{subsec:inducedpp} and let~$\cM_0^\beta$ be the ample line bundle on $B$ whose polarization isogeny $\varphi_{\cM_0^\beta}$ is $\vphi_{\cM_0} \circ \beta \colon B \ra B^\vee$. 
Then $\widehat{f}^*\cL_0$ is algebraically equivalent to $\cM_0^\beta$.
\end{lem}

\begin{proof}
By Proposition~\ref{prop:ppav} applied to $(B, \cM_0)$, there exists $\gamma \in \End(B)^{++}$ such
that $\widehat{f}^\ast \cL_0$ is algebraically equivalent to $\cM_0^\gamma$.
Now, we have the following algebraic equivalences $f^\ast \cM_0^\gamma \sim \cL_0^{\beta \gamma}$
and $\cL_0^{\beta \gamma} \sim (f \circ \widehat{f})^\ast \cL_0=\beta^\ast
\cL_0 \sim \cL_0^{\beta^2}$ where the last equivalence comes from
Corollary~\ref{cor:endo_polarisation} and the fact that $\beta$ is a real
endomorphism. By applying Proposition~\ref{prop:ppav} again, we get that
$\beta=\gamma$.
\end{proof}

\begin{rem}
The line bundles satisfy $(\widehat{f}^*\cL_0)^{n} \sim \widehat{f}^*\cL_0^{n} \sim \widehat{f}^*\cL$, as $n \nmid \ell$, and $(\cM_0^\beta)^{n} \sim \cM^{\beta}$. Therefore,~$\widehat{f}^*\cL$ is algebraically equivalent to $\cM^{\beta}$. But then, both line bundles being totally symmetric implies that $\widehat{f}^\ast \cL$ is linearly equivalent to $\cM^{\beta}$.
\end{rem}

Using Lemma~\ref{lem:mbeta} we can recover the theta coordinates for $(B, \cM^\beta)$ with respect to a suitably chosen theta structure $\Theta_{\cM^\beta}$ without solving systems of polynomial equations. Then using the same idea as in i) applied to $F \colon (B^4, (\cM^\beta)^{\star 4}) \ra (B^4, \cM^{\star 4})$, we compute the theta null point for $(B^4, \cM^{\star 4})$ for some compatible theta structure $\widetilde{\Theta}_{\cM^{\star 4}}$. In general, we can recover the theta null point for a single polarized factor $(B, \cM)$ only after a symplectic transformation of the theta coordinates (induced by a metaplectic automorphism turning $\widetilde{\Theta}_{\cM^{\star 4}}$ into a product structure), as explained in Section \ref{subsec:metaplect}.

Note that in some cases $\beta$ can be written as the sum of 2 squares of real algebraic integers, therefore in the sequel we will consider $F$ as a real endomorphism of $B^r$ for $r = 2$ or $r = 4$.

\paragraph{Isogeny theorem for $\widehat{f}$.} 
Lemma~\ref{lem:mbeta} shows that $\widehat{f} \colon (B, \cM^\beta) \ra (A, \cL)$ is an isogeny of polarized abelian varieties. Let $\Theta_{\cM^{\beta}}$ be a theta structure on $(B, \cM^\beta)$ compatible with the theta structure $\Theta_{\cL}$, i.e. such that the isogeny $\widehat{f} \colon (B, \cM^\beta, \Theta_{\cM^\beta}) \ra (A, \cL, \Theta_{\cL})$ is an isogeny of polarized abelian varieties with theta structures. Then $\Theta_{\cM^\beta}$ induces a symplectic decomposition $K(\cM^{\beta}) = K_1(\cM^\beta) \oplus K_2(\cM^\beta)$. 
Since $K_i(\cM^{\beta})=K_i(\cM^\beta)[\beta]\oplus K_i(\cM^\beta)[n]$, we have a symplectic decomposition $K_1(\cM^\beta)[n]\oplus K_2(\cM^\beta)[n]$ of $B[n]$ which yields (via $\widehat{f}$) the symplectic decomposition 
on $K(\cL) = A[n]$ determined by the theta structure $\Theta_{\cL}$. 
If we assume that the kernel $\widehat{G}$ of $\widehat{f}$ is contained in $K_2(\cM^\beta)[\beta]$ (which we can always do since the compatibility requirement on $\Theta_{\cM^\beta}$ is only on the $n$-torsion and not on the $\beta$-torsion points), the kernel $G$ of $f$ coincides with $\widehat{f}(K_1(\cM^\beta)[\beta])$. Moreover, the isogeny $\widehat{f}$  induces an isomorphism between $K_1(\cM^\beta)[n]$ and $K_1(\cL)$ and significantly simplifies the formula appearing in the isogeny theorem when applied to $\widehat{f}$. Indeed, there exists a constant $\lambda \in \bar{k}^\times$ such that for all points $y \in B(\bar{k})$ and all $i\in K_1(\cL)$, 
\begin{equation}\label{eq:fhat}
\theta_i^{\Theta_\cL}(\widehat{f}(y))=\lambda \cdot \theta_j^{\Theta_{\cM^\beta}}(y), 
\end{equation}
where $j\in K_1(\cM^\beta)[n]$ is the unique preimage of $i$ via $\widehat{f}$. Specializing to $y = 0_B$, we obtain 
\begin{equation}\label{eq:fhat_null}
\theta_i^{\Theta_\cL}(0_A)=\lambda \cdot \theta_j^{\Theta_{\cM^\beta}}(0_B).  
\end{equation}

\paragraph{Isogeny theorem for $F$.}\label{par:isogF}
Let $F \colon B^r \ra B^r$ be as in Section~\ref{par:issue2}.  
We first prove the following: 
\begin{lem}\label{lem:F}
The line bundles $F^*\cM_0^{\star r}$ and $(\cM_0^\beta)^{\star r}$ are algebraically equivalent. 
\end{lem}

\begin{proof}
  From Corollary~\ref{cor:endo_polarisation}, 
  $F^* \cM_0^{\star r}$ is algebraically equivalent to $(\cM_0^{\star
  r})^{F^\dagger  F}$. Here $F^\dagger$ denotes the action of the Rosati
  involution on $\End(B^r)$, which is given component wise as the
  Rosati involution (of $\End(B)$) on the coefficients of the transpose of $F$.
  Since $F$ is composed of totally real endomorphisms, we have that $F^\dagger=F^t$, so that $F^t F=\beta \Id_r$. Furthermore, comparing polarization isogenies we have that
  $(\cM_0^{\star r})^{\beta \Id_r}$ is algebraically equivalent to $(\cM_0^\beta)^{\star
  r}$.
\end{proof}

As $F^*\cM_0^{\star r}$ is algebraically equivalent to $(\cM_0^\beta)^{\star r}$ and as $\cM^{\star r}$ and 
$(\cM^{\beta})^{\star r}$ are both totally symmetric, we have that $F^*\cM^{\star r}$ is linearly equivalent to $(\cM^\beta)^{\star r}$. 

Consider the $r$-fold product theta structure $\Theta_{(\cM^\beta)^{\star r}}$ on $(B^r,(\cM^{\beta})^{\star r})$ (determined by $\Theta_{\cM^\beta}$) and let $\widetilde\Theta_{\cM^{\star r}}$ be a compatible (for the isogeny $F$) theta structure on $(B, \cM^{\star r})$. According to Theorem~\ref{thm:isog} applied to 
$$
F \colon (B^r, (\cM^\beta)^{\star r}, \Theta_{(\cM^\beta)^{\star r}}) \ra (B^r, \cM^{\star r}, \widetilde{\Theta}_{\cM^{\star r}}), 
$$ 
there exists $\lambda \in \bar{k}^\times$ such that for every $\mathbf{y} = (y_1, \dots, y_r) \in B^r(\bar{k})$ and $\mathbf{k} \in K_1(\cM^{\star r})$,  
\begin{equation}\label{eq:isogF}
\displaystyle\theta^{\widetilde{\Theta}_{\cM^{\star r}}}_{\mathbf{k}}( F(\mathbf{y})) = \lambda \cdot \displaystyle\sum_{\substack{\mathbf{t} \in K_1((\cM^\beta)^{\star r})[\beta] \\ F(\mathbf{t}) = 0}}
\theta^{\Theta_{(\cM^\beta)^{\star r}}}_{\bj+\mathbf{\bt}} (\mathbf{y}) = \lambda \cdot \displaystyle\sum_{\substack{\mathbf{t} \in K_1((\cM^\beta)^{\star r})[\beta] \\ F(\mathbf{t}) = 0}}
\prod_{s=1}^r \theta^{\Theta_{\cM^\beta}}_{j_s+t_s} ({y_s}). 
\end{equation}
Here, we are using the decomposition $K_1((\cM^{\beta})^{\star r}) = K_1((\cM^{\beta})^{\star r})[n] \oplus K_1((\cM^{\beta})^{\star r})[\beta]$, with $\mathbf{j} = (j_1, \dots, j_r)$ being the unique element in $K_1((\cM^{\beta})^{\star r})[n]$ that satisfies $F(\mathbf{j}) = \mathbf{k}$, and $\mathbf{t} = (t_1, \dots, t_r)$ is in the kernel of $F$. Specializing to $\mathbf{y} = 0$, we obtain 
\begin{equation}\label{eq:Fnull}
\displaystyle\theta^{\widetilde{\Theta}_{\cM^{\star r}}}_{\mathbf{k}}(0_{B^r}) = \lambda \cdot \displaystyle\sum_{\substack{\mathbf{t} \in K_1((\cM^\beta)^{\star r})[\beta] \\ F(\mathbf{t}) = 0}}
\prod_{s=1}^r \theta^{\Theta_{\cM^\beta}}_{j_s+t_s} (0_{B}). 
\end{equation}

\subsection{Computing the theta null point $(\theta^{\widetilde{\Theta}_{\cM^{\star r}}}_{\mathbf{k}}(0_{B^r}))_{\mathbf{k} \in K_1(\cM^{\star r})}$ of $B^r$}\label{subsec:complifts}

To evaluate the right-hand side of \eqref{eq:Fnull}, we need to know the theta null point $(\theta^{\Theta_{\cM^\beta}}_{j} (0_{B}))_{j \in K_1(\cM^\beta)}$. The input of the algorithm only provides us with the (projective) theta null point $(\theta^{\Theta_{\cL}}_{i} (0_{A}))_{i \in K_1(\cL)}$ and the (projective) theta coordinates $(\theta^{\Theta_{\cL}}_{i} (t))_{i \in K_1(\cL)}$, where $t$ is a generator of the kernel of $f$.

Equation \eqref{eq:fhat_null} will recover some, but not all of the projective theta coordinates for the theta null point $(\theta^{\Theta_{\cM^\beta}}_{j} (0_{B}))_{j \in K_1(\cM^\beta)}$. More precisely, we will recover $\{\theta^{\Theta_{\cM^\beta}}_{j} (0_{B})\}_{j \in K_1(\cM^\beta)[n]}$ up to a projective factor. To recover the rest of the coordinates, we are hoping to use the action of the Heisenberg group $\cH(\delta_{\cM^\beta})$ on the affine cone $\widetilde{B}$ (as described in Section \ref{subsubsec:canaffine}), which implies that for each $t'' \in K_1(\cM^{\beta})[\beta]$ we have  
\begin{equation}\label{eq:tors-translate}
\theta^{\Theta_{\cM^\beta}}_{j + t''}(0_B) = \lambda_{t''} \cdot \theta^{\Theta_{\cM^\beta}}_{j}(t''), \quad \forall j \in K_1(\cM^{\beta})[n] 
\end{equation}
up to a projective factor $\lambda_{t''} \in \bar{k}^\times$. This shows that one can recover the missing projective coordinates by using the theta coordinates of all the $\ell$-torsion points $t'' \in K_1(\cM^\beta)[\beta]$. Letting $t' \in K_1(\cM^\beta)[\beta]$ be the unique preimage in $K_1(\cM^\beta)[\beta]$ of $t$ under $\widehat{f}$, we can compute (up to projective factors) the coordinates
$$\{\theta^{\Theta_{\cM^\beta}}_{j}(t')\}_{j\in K_1(\cM^\beta)[n]}, \, \{\theta^{\Theta_{\cM^\beta}}_{j}(2t')\}_{j\in K_1(\cM^\beta)[n]}, \dots ,\, \{\theta^{\Theta_{\cM^\beta}}_{j}((\ell-1)t')\}_{j\in K_1(\cM^\beta)[n]}$$ 
using \eqref{eq:fhat} as
\begin{equation}\label{eq:isogthm_fhat}
\theta_j^{\Theta_{\cM^\beta}}(ut') = \lambda_u \cdot \theta_{\widehat{f}(j)}^{\Theta_\cL}(ut), \text{ for } 1 \le u \le \ell-1.
\end{equation}
The latter can be computed by computing the point $ut$ in Mumford coordinates and then converting to theta coordinates $(\theta^{\Theta_\cL}_i(ut))_{i \in K_1(\cL)}$ for $(A,\cL,\Theta_{\cL})$. The problem is that we cannot simply patch the projective coordinates $(\theta^{\Theta_\cL}_i(0_A))_{i \in K_1(\cL)}, (\theta^{\Theta_\cL}_i(t))_{i \in K_1(\cL)}$ ,..., $(\theta^{\Theta_\cL}_i((\ell-1)t))_{i \in K_1(\cL)}$ together and obtain the theta null point of $B$ for $\Theta_{\cM^\beta}$. That is, knowing $\ell$ projective points in $\mathbb{P}^{n^2-1}$, there is no natural map
$$\underbrace{\mathbb{P}^{n^2-1}\times \cdots \times \mathbb{P}^{n^2-1} }_{\ell \text{ times}}\to \mathbb{P}^{\ell n^2-1}$$ 
giving a projective point in $\mathbb{P}^{\ell n^2-1}$. Hence, for computing the right-hand side of \eqref{eq:Fnull} we cannot simply substitute the theta coordinates of $0_A, t, \dots , (\ell-1)t$ in the product $\prod_{s=1}^r \theta^{\Theta_{\cM^\beta}}_{j_s+t''_s} (0_{B})$. Still we would like to make use of Equations \eqref{eq:tors-translate} and \eqref{eq:isogthm_fhat} to compute the right-hand side of \eqref{eq:Fnull}. To do so, we will have to work with affine lifts $\widetilde{0}_B \in \mathbb{A}^{\ell n^2}(\bar{k})$ and $\widetilde{0}_A, \widetilde{t}, \dots , \widetilde{(\ell-1)t} \in \mathbb{A}^{n^2}(\bar{k})$ of the points $0_B$ and $0_A, t, \dots , (\ell-1)t$ respectively.
\begin{notation} 
Let $(A,\cN,\Theta_\cN)$ be a polarized abelian variety with a theta structure. Let $x\in A(\bar{k})$ be a point on $A$ with projective theta coordinates $\left( \theta_i^{\Theta_\cN} (x) \right) _{i \in K_1(\cN)}$. Let $\widetilde{x}$ be an affine lift of the theta coordinates of $x$. Then we write $\theta_i^{\Theta_\cN} (\widetilde{x})$ for the $i$-th coordinate of $\widetilde{x}$.
\end{notation}

Equations \eqref{eq:tors-translate} and \eqref{eq:isogthm_fhat} are equalities between coordinates of projective points. This means, knowing the coordinates of one projective point, we know the coordinates of the second one up to a projective factor. For \eqref{eq:tors-translate}, fixing an affine lift $\widetilde{0}_B$ of $0_B$, we can set the projective factor $\lambda_{t''} = 1$, which determines an affine lift of $t''$. And for \eqref{eq:isogthm_fhat}, writing $ut' = t''$, an affine lift of $ut'$ for $\Theta_{\cM^\beta}$ determines an affine lift of $ut$ for $\Theta_\cL$ by setting $\lambda_u=1$. Combining \eqref{eq:tors-translate} and \eqref{eq:isogthm_fhat}, if we were given an affine lift $\widetilde{0}_B = \left( \theta_{j}^{\Theta_{\cM^\beta}}(\widetilde{0}_B) \right)_{j \in K_1(\cM^\beta)}$ of the theta null point of $B$ for $\Theta_{\cM^\beta}$, we could define affine lifts 
$$\widetilde{0}_A = \left(\theta_i^{\Theta_\cL}(\widetilde{0}_A) \right)_{i \in K_1(\cL)}, \widetilde{t} = \left(\theta_i^{\Theta_\cL}(\widetilde{t}) \right)_{i \in K_1(\cL)}, \dots , \widetilde{(\ell-1)t} = \left(\theta_i^{\Theta_\cL}(\widetilde{(\ell-1)t}) \right)_{i \in K_1(\cL)}$$ of $0_A, t, \dots , (\ell-1)t$ as
\begin{align}\label{def:affinelifts}
\theta_i^{\Theta_\cL}(\widetilde{0}_A) &:= \theta_j^{\Theta_{\cM^\beta}}(\widetilde{0}_B),\nonumber \\ 
\theta_i^{\Theta_\cL}(\widetilde{t}) &:= \theta_{j+t'}^{\Theta_{\cM^\beta}}(\widetilde{0}_B),\\
&\hspace{3mm}\vdots \nonumber \\
\theta_i^{\Theta_\cL}(\widetilde{(\ell-1)t}) &:= \theta_{j+(\ell-1)t'}^{\Theta_{\cM^\beta}}(\widetilde{0}_B), \nonumber
\end{align}
where $i$ runs over $K_1(\cL)$ and $j=\widehat{f}^{-1}(i)$ runs over $K_1(\cM^\beta)[n]$. Unfortunately, the input of the algorithm only provides us with an affine lift $\widetilde{0}_A$ of $0_A$ and not with an affine lift of $0_B$. However, by the affine version of \eqref{eq:isogthm_fhat} where we set the scalar to 1, a choice of lift $\widetilde{0}_A$ determines a lift of $0_B$, which in return determines lifts of $t, \dots , (\ell-1)t$ as in \eqref{def:affinelifts}. If now out of all the possible lifts of $t, \dots , (\ell-1)t $ we were able to determine precisely the ones induced in this way, we could patch their coordinates together and obtain an affine lift of $0_B$ for $\Theta_{\cM^\beta}$, and then using (\ref{eq:Fnull}) we can compute an affine lift of the theta null point of $B^r$ with respect to $\widetilde{\Theta}_{\cM^{\star r}}$. Of course we cannot take arbitrary lifts of $t, \dots, (\ell-1)t$ and hope they correspond to the ones induced by the fixed lift $\widetilde{0}_A$, but we will show that the induced lifts satisfy some compatibility condition, and that we can compute them up to $\ell$-th roots of unity.

\paragraph{Affine lifts of $\widehat{f}$.}\label{par:affine_isogthm}
The isogeny theorem applied to the isogeny of polarized abelian varieties
$$
\widehat{f} \colon (B, \cM^\beta, \Theta_{\cM^\beta}) \to (A, \L, \Theta_\cL)
$$
implies that there exists $\lambda \in \bar{k}^\times$ such that for all $i \in K_1(\cL)$ and for all $y \in B(\bar{k})$,
$$\theta_i^{\Theta_\cL}(\widehat{f}(y)) = \lambda \cdot \theta_j^{\Theta_{\cM^\beta}}(y),$$
where $j\in K_1(\cM^\beta)[n]$ is the unique index such that $\widehat{f}(j) =  i$. The scalar $\lambda$ is just a projective factor of the projective point $\widehat{f}(y) \in A(\bar{k})$, meaning that replacing $\lambda$ by any $\lambda' \in \bar{k}^\times$, the statement of the isogeny theorem remains true. We can lift the isogeny $\widehat{f}$ to an ``affine" isogeny 
$$\widetilde{\widehat{f}}_\lambda \colon \widetilde{B} \to \widetilde{A}, \, \widetilde{y} \mapsto \widetilde{\widehat{f}}(\widetilde{y}),$$
where the $i$-th coordinate of $\widetilde{\widehat{f}}(\widetilde{y})$ (for $i \in K_1(\cL)$) is given by
$$\theta_i^{\Theta_\cL}(\widetilde{\widehat{f}}_\lambda(\widetilde{y})) = \lambda \cdot \theta_{\widehat{f}^{-1}(i)}^{\Theta_{\cM^\beta}}(\widetilde{y}).$$
Moreover, every choice of $\lambda \in \bar{k}^\times$ determines an affine lift $\widetilde{\widehat{f}}_\lambda$ of $\widehat{f}$ by the above.
For now, fix $\widetilde{\widehat{f}} = \widetilde{\widehat{f}}_1$ the affine lift of $\widehat{f}$ where we set $\lambda = 1$. Then $\widetilde{\widehat{f}}$ satisfies
\begin{equation}
\widetilde{\widehat{f}}(\xi \cdot \widetilde{y}) = \xi \cdot \widetilde{\widehat{f}}(\widetilde{y}) \text{ for } \widetilde{y} \in \widetilde{B} \text{ and } \xi \in \bar{k}^\times.
\end{equation}
Note that a choice of affine lift $\widetilde{0}_A$ of $0_A$ determines an affine lift $\widetilde{0}_B$ of $0_B$ by the relation
$$\widetilde{0}_B = \widetilde{\widehat{f}}(\widetilde{0}_A).$$

\paragraph{The action of the Heisenberg group on $\widetilde{B}$.}
Let $\delta_{\cM^\beta} = (n,\dots, n, \ell n) \in \Z^g$ be the type of $\cM^\beta$ and consider the Heisenberg group $\mathcal{H}(\delta_{\cM^\beta})$, whose underlying set is given by $k^\times \times \Z(\delta_{\cM^\beta}) \times \widehat{\Z}(\delta_{\cM^\beta})$, where $\Z(\delta_{\cM^\beta}) = (\Z/ n\Z)^{g-1} \times \Z / \ell n \Z$. By \eqref{eq:actionHeisenberg} we know that $\mathcal{H}(\delta_{\cM^\beta})$ acts on $\widetilde{B}$ as
$$(\alpha, i, k) \cdot \theta_j^{\Theta_{\cM^\beta}} = \alpha k(-j-i) \theta_{j+i}^{\Theta_{\cM^\beta}}$$
and that $(1,i,k) \cdot \widetilde{y}$ is a lift of $y+z$, where $z = \overline{\Theta}_{\cM^\beta}((i,k))$.

Let $i_{t'} \in \Z(\delta_{\cM^\beta})$ be the unique preimage of $t' \in K_1(\cM^\beta)[\beta]$ under $\overline{\Theta}_{\cM^\beta}$. Note that $ui_{t'} = i_{ut'}$, so that $\ell i_{t'} = 0$. The action of $\mathcal{H}(\delta_{\cM^\beta})$ on $\widetilde{B}$ determines affine lifts $\widetilde{t'}, \widetilde{2t'}, \dots, \widetilde{(\ell-1)t'}$ of $t', 2t', \dots, (\ell-1)t'$ as follows
$$\widetilde{ut'} = (1, ui_{t'}, 0) \cdot \widetilde{0}_B \text{, for }1 \le u \le \ell-1.$$
This means that we have equality between the coordinates
$$\theta_{j}^{\Theta_{\cM^\beta}}(\widetilde{ut'}) = \theta_{j }^{\Theta_{\cM^\beta}}((1, ui_{t'},0) \cdot \widetilde{0}_B)  = \theta_{j + ut' }^{\Theta_{\cM^\beta}}(\widetilde{0}_B).$$
Therefore, fixing a lift $\widetilde{0}_A$ of $0_A$, which determines a lift $\widetilde{0}_B$ of $0_B$ by the relation $\widetilde{0}_B = \widetilde{\widehat{f}}(\widetilde{0}_A)$, the lifts $\widetilde{t}, \dots, \widetilde{(\ell-1)t}$ of \eqref{def:affinelifts} satisfy
$$\widetilde{t} = \widetilde{\widehat{f}}((1, i_{t'}, 0) \cdot \widetilde{0}_B) , \dots , \widetilde{(\ell-1)t} = \widetilde{\widehat{f}}((1, (\ell-1)i_{t'}, 0) \cdot \widetilde{0}_B).$$

\begin{defn} 
Let $\widetilde{0}_A$ be a fixed affine lift of $0_A$. Let $\widetilde{0}_B$ be the lift induced by $\widetilde{\widehat{f}}$ and $\widetilde{0}_A$. For  $1 \le u \le \ell-1$, the induced lift
$$\widetilde{\widehat{f}}((1, ui_{t'}, 0) \cdot \widetilde{0}_B) $$
is called \textbf{right} lift of $ut$, and is denoted by $\widetilde{ut}_{\text{right}}$.
\end{defn}

The terminology comes from the following: knowing $\widetilde{0}_A$ and the \textbf{right} lifts $\widetilde{t}_{\text{right}}, \dots , \widetilde{(\ell-1)t}_{\text{right}}$, we can patch their coordinates together and obtain an affine lift $\widetilde{0}_B$ of $0_B$ for $\Theta_{\cM^\beta}$, from which we can compute an affine lift of the theta null point of $B^r$ for $\widetilde{\Theta}_{\cM^{\star r}}$. In general we cannot hope finding the \textbf{right} lifts, but as we will see, the \textbf{right} lifts satisfy some compatibility conditions, so that we can compute them up to $\ell$-th roots of unity.

\paragraph{Excellent lifts.}\label{par:excellent} 
We will show that the \textbf{right} lifts are \textit{excellent} lifts, following the definition of \cite[\S 7.4]{drobert:thesis}. We will therefore briefly recall the notion of the pseudo-operations $\chainadd$, $\chainmultadd$ and $\chainmult$ on the affine cone of an abelian variety. For  more details we refer to \cite[\S 4.4]{drobert:thesis}. Let $(A, \cL, \Theta_{\cL})$ be a polarized abelian variety with theta structure and let $\widetilde{A}$ be the affine cone associated to the projective embedding induced by $\Theta_{\cL}$. Let $\widetilde{0}_A\in \widetilde{A}$ be a fixed lift of $0_A$. Then,
\begin{itemize}
\item given affine lifts $\widetilde{x}, \widetilde{y}, \widetilde{x-y} \in \widetilde{A}$ of $x, y, x-y \in A$,
$$\chainadd(\widetilde{x}, \widetilde{y}, \widetilde{x-y})$$
is an algorithm that computes the affine lift $\widetilde{x+y}$ of $x+y$ so that $\widetilde{0}_A, \widetilde{x}, \widetilde{y}, \widetilde{x+y}, \widetilde{x-y}$ satisfy the Riemann relations \cite[Thm.4.4.6.]{drobert:thesis}.
\item given an integer $m \ge 1$ and affine lifts $\widetilde{x}, \widetilde{y}, \widetilde{x+y} \in \widetilde{A}$, 
$$\chainmultadd(m, \widetilde{x+y}, \widetilde{x}, \widetilde{y})$$
is an algorithm that computes an affine lift of $mx+y$. It is defined by recursive calls of $\chainadd$. If $m<0$ we set 
$$\chainmultadd(m, \widetilde{x+y}, \widetilde{x}, \widetilde{y}):=\chainmultadd(-m, -\widetilde{x+y}, -\widetilde{x}, -\widetilde{y}).$$
\item given an affine lift $\widetilde{x} \in \widetilde{A}$,
$$\chainmult(m,\widetilde{x}) := \chainmultadd(m, \widetilde{x}, \widetilde{x}, \widetilde{0}_A)$$
is an algorithm that computes an affine lift of $mx$. 
\end{itemize}

\begin{defn}
Suppose that $\widetilde{0}_A$ is a fixed affine lift of the theta null point of $A$ for $\Theta_{\cL}$. 
We call an affine lift $\widetilde{t}_e$ of $t \in G$ \emph{excellent} with respect to $(A, \cL, \Theta_{\cL}, \widetilde{0}_A)$ if 
\begin{equation}\label{eq:excellent}
\chainmult(m+1, \widetilde{t}_e) = -\chainmult(m, \widetilde{t}_e),  
\end{equation}
where $\ell = 2m +1$.
\end{defn}

Throughout this section, suppose that we have fixed a lift $\widetilde{0}_A$ of $0_A$. To compute an excellent lift of $t$, take any affine lift $\widetilde{t}$ and look for a scalar $\lambda_t \in \bar{k}^\times$ such that $\widetilde{t}_e = \lambda_t \cdot \widetilde{t}$ is excellent. Here, $\lambda_t \cdot \widetilde{t}$ is the affine point with $i$-th coordinate equal to $\lambda_t \cdot \theta^{\Theta_\cL}_i(\widetilde{t})$. Indeed, using that  
$$
\chainmult(m+1, \lambda_t \cdot \widetilde{t}) = \lambda_t^{(m+1)^2} \cdot \chainmult(m+1, \widetilde{t}) 
$$  
and 
$$
\chainmult(m, \lambda_t \cdot \widetilde{t}) = \lambda_t^{m^2} \cdot \chainmult(m, \widetilde{t}), \bar{k}
$$
we obtain that $\widetilde{t}_e = \lambda_t \cdot \widetilde{t}$ will be excellent if 
$$
\lambda_t^\ell \cdot \chainmult(m+1, \widetilde{t}) = -\chainmult(m, \widetilde{t}). 
$$
This determines $\lambda_t^\ell$ precisely; yet, one still has to take an $\ell$-th root of unity, thus introducing some ambiguity in the choice of the affine lift. 

We will show that if we fix a lift $\widetilde{0}_A$, then the \textbf{right} lift $\widetilde{t}_{\text{right}}$ is an excellent lift, and for all $u = 2, \dots ,\ell-1$, we have $\widetilde{ut}_{\text{right}} = \chainmult(u,\widetilde{t}_{\text{right}})$.

\paragraph{The \textbf{right} lifts are excellent.}\label{par:right_is_excellent} Let us first prove the following lemma about the compatibility of chain operations with the action of the Heisenberg group. Suppose that we have fixed an affine lift $\widetilde{0}_B$ of $0_B$.
\begin{lem}\label{lem:compatibility}
Let $i \in \Z(\delta_{\cM^\beta})$ and $k \in \widehat{\Z}(\delta_{\cM^\beta})$ and let $m \in \Z_{\ge 0}$. Then, for $\widetilde{y} \in \widetilde{B}$ we have
$$\chainmult(m, (1,i,k)\cdot \widetilde{y}) = (1, mi, mk)\cdot \chainmult(m, \widetilde{y}).$$
Moreover, for $\widetilde{y'} \in \widetilde{B}$
$$
\chainmultadd(m, (1,i,k)\cdot \widetilde{y+y'}, (1,i,k)\cdot \widetilde{y'}, \widetilde{y}) = (1,mi,mk)\cdot \chainmultadd(m,\widetilde{y+y'}, \widetilde{y'}, \widetilde{y}).
$$
\end{lem}

\begin{proof}
We first recall the following: for $n \in \Z_{\ge 0 }$ and $\widetilde{x} \in \widetilde{B}$, by definition
$$\chainmult(n, \widetilde{x}) := \chainmultadd(n, \widetilde{x}, \widetilde{x}, \widetilde{0}_B)$$
and we compute $\chainmultadd(n, \widetilde{x}, \widetilde{x}, \widetilde{0}_B)$ recursively as
$$\chainmultadd(n, \widetilde{x}, \widetilde{x}, \widetilde{0}_B) := \chainadd(\chainmultadd(n-1, \widetilde{x}, \widetilde{x}, \widetilde{0}_B), \widetilde{x}, \chainmultadd(n-2, \widetilde{x}, \widetilde{x}, \widetilde{0}_B)).$$
It follows that
$$\chainmult(n, \widetilde{x}) = \chainadd(\chainmult(n-1, \widetilde{x}), \widetilde{x}, \chainmult(n-2, \widetilde{x})).$$
The proof is by induction on $m$. For $m = 1$ the statement is precisely \cite[Prop.3.11]{lubicz-robert:isogenies}. Assume that the statement is true for all $n \leq m$ and write 
\begin{small}
\begin{align*}
& \chainmult(m + 1, (1,i,k)\cdot \widetilde{y}) \\
& \stackrel{\text{by def.}}{=}  \chainadd( \chainmult(m, (1,i,k)\cdot \widetilde{y}), (1,i,k)\cdot \widetilde{y}, \chainmult(m-1, (1,i,k)\cdot \widetilde{y}))  \\
& \stackrel{\text{by ind.}}{=}  \chainadd( (1, mi, mk) \cdot \chainmult(m, \widetilde{y}), (1,i,k)\cdot \widetilde{y}, (1, (m-1) i, (m-1) k) \cdot \chainmult(m-1, \widetilde{y}))  \\
& \stackrel{\tiny \cite[3.11]{lubicz-robert:isogenies}}{=} (1, (m+1) i, (m+1) k) \cdot \chainadd( \chainmult(m, \widetilde{y}), \widetilde{y}, \chainmult(m-1, \widetilde{y})) \\
& \stackrel{\text{by def.}}{=} (1, (m+1) i, (m+1) k) \cdot \chainmult(m+1, \widetilde{y}). 
\end{align*}
\end{small}
This proves the induction hypothesis. The proof for $\chainmultadd$ is similar except that we use 
\begin{small}
\begin{align*}
& \chainmultadd(m+1, \widetilde{y + y'}, \widetilde{y'}, \widetilde{y}) \\ 
& = \chainadd(\chainmultadd(m, \widetilde{y + y'}, \widetilde{y'}, \widetilde{y}), \widetilde{y'}, \chainmultadd(m-1, \widetilde{y + y'}, \widetilde{y'}, \widetilde{y})).
\end{align*}
\end{small}
\end{proof}
Now, using Lemma \ref{lem:compatibility} from above, and Lemma 3.9 and Corollary 3.17 of \cite{lubicz-robert:isogenies}, we can show the following key result.

\begin{prop}\label{prop:excellentlift}
Let $\widetilde{0}_A$ be a fixed affine lift of $0_A$. Then the \textbf{right} lift $\widetilde{t}_{\text{right}}$ is an excellent lift of~$t$, and for $2 \le u \le \ell-1$ we have
$$\widetilde{ut}_{\text{right}} = \chainmult(u, \widetilde{t}_{\text{right}}).$$
It follows that $\widetilde{2t}_{\text{right}}, \dots , \widetilde{(\ell-1)t}_{\text{right}}$ are excellent lifts too.
\end{prop}

\begin{proof}
Let $\widetilde{0}_B$ be the lift of $0_B$ induced by $\widetilde{0}_A$ and $\widetilde{\widehat{f}}$ (c.f. Section \ref{par:affine_isogthm}). Let 
$$\widetilde{t'} = (1, i_{t'}, 0) \cdot \widetilde{0}_B, \dots , \widetilde{(\ell-1)t'} = (1, (\ell-1)i_{t'}, 0) \cdot \widetilde{0}_B$$
be the lifts of $t' , \dots , (\ell-1)t'$ induced by the action of the Heisenberg group on $\widetilde{B}$. Let $\ell = 2m +1$. Observe that $(1,(m+1)i_{t'}, 0) \cdot \widetilde{0}_B = (1, -m i_{t'}, 0) \cdot \widetilde{0}_B$, which follows from $\ell i_{t'}=0$. Now,

\begin{align*}
\chainmult(m+1, \widetilde{t}_{\text{right}}) 
&= \chainmult(m+1, \widetilde{\widehat{f}}(\widetilde{t'}))\\
&= \widetilde{\widehat{f}}(\chainmult(m+1, \widetilde{t'})) \hspace{4mm} \text{by } \cite[3.17]{lubicz-robert:isogenies}\\
&= \widetilde{\widehat{f}}(\chainmult(m+1, (1,i_{t'},0) \cdot \widetilde{0}_B))\\
&= \widetilde{\widehat{f}}((1,(m+1)i_{t'},0) \cdot \chainmult(m+1, \widetilde{0}_B)) \hspace{4mm} \text{by Lemma } \ref{lem:compatibility}\\
&= \widetilde{\widehat{f}}((1,-m i_{t'},0) \cdot \widetilde{0}_B)\\
&= \widetilde{\widehat{f}}(-(1,m i_{t'},0) \cdot \widetilde{0}_B) \hspace{4mm} \text{by } \cite[3.9]{lubicz-robert:isogenies}\\
&= \widetilde{\widehat{f}}(-(1,m i_{t'},0) \cdot \chainmult(m, \widetilde{0}_B)) \\
&= \widetilde{\widehat{f}}(-\chainmult(m,(1,i_{t'},0) \cdot \widetilde{0}_B)) \hspace{4mm} \text{by Lemma }\ref{lem:compatibility}\\
&= \widetilde{\widehat{f}}(-\chainmult(m, \widetilde{t'}))\\
&= -\chainmult(m, \widetilde{\widehat{f}}(\widetilde{t'})) \hspace{4mm} \text{by } \cite[3.9]{lubicz-robert:isogenies} \text{ and } \cite[3.17]{lubicz-robert:isogenies}\\
&= -\chainmult(m,\widetilde{t}_{\text{right}}),
\end{align*}
and therefore $\widetilde{t}_{\text{right}}$ is an excellent lift. The equality $\widetilde{ut}_{\text{right}} = \chainmult(u, \widetilde{t}_{\text{right}})$ is shown in a similar way, using that $\widetilde{ut}_{\text{right}}= \widetilde{\widehat{f}}((1,ui_{t'},0), \widetilde{0}_B)$.
\end{proof}

Given the input $\left( \theta_i ^{\Theta_\cL} (0_A) \right)_{i \in K_1(\cL)}$ and $\left( \theta_i ^{\Theta_\cL} (t) \right)_{i \in K_1(\cL)}$ of the algorithm, when fixing an affine lift $\widetilde{0}_A$ of $0_A$, we can compute an excellent lift $\widetilde{t}$ of $t$ for $\widetilde{0}_A$. This lift will differ from $\widetilde{t}_{\text{right}}$ by an $\ell$-th root of unity $\zeta_t$, i.e., $\widetilde{t} = \zeta_t \cdot \widetilde{t}_{\text{right}}$. By Proposition \ref{prop:excellentlift}, the lift $\widetilde{ut} = \chainmult(u, \widetilde{t})$ of $ut$, for $2 \le u \le \ell-1$, will differ from the \textbf{right} lift of $ut$ by $\zeta_t^{u^2}$. When fixing an excellent lift of $t$ we commit an error of an $\ell$-th root of unity, compared to fixing the \textbf{right} lift. We will show that this ambiguity does not affect the computation of the theta null point of $B^r$ for $\widetilde{\Theta}_{\cM^{\star r}}$.

\paragraph{Independence of the choice of an excellent lift of $t$.}\label{par:indep_of_excellent_lift}
We prove the case $r=4$. The case $r=2$ is easier and can be proven in a similar way. 

It is not hard to show that 
\begin{equation}\label{eq:kerFcapG4}
\ker(F) \cap G^4 = \{ F^t(t_1, t_2, 0, 0) \colon t_1, t_2 \in G\},  
\end{equation} 
where the endomorphism $F^t$ is represented by the matrix 
$$
F^t=
\begin{pmatrix}
\alpha_1&-\alpha_2&-\alpha_3&-\alpha_4\\
\alpha_2&\alpha_1&\alpha_4&-\alpha_3\\
\alpha_3&-\alpha_4&\alpha_1&\alpha_2\\
\alpha_4&\alpha_3&-\alpha_2&\alpha_1
\end{pmatrix}.  
$$
Fix a lift $\widetilde{0}_A$ of $0_A$ and fix an excellent lift $\widetilde{t}$ of $t$ that differs from the \textbf{right} lift by an $\ell$-th root of unity $\zeta_t$. Fix lifts $\widetilde{2t} = \chainmult(2, \widetilde{t}), \dots , \widetilde{(\ell-1)t} = \chainmult(\ell-1, \widetilde{t})$ of $2t, \dots , (\ell-1)t$. As we have seen, these lifts differ from the \textbf{right} lifts by $\zeta_t^{2^2}, \dots , \zeta_t^{(\ell-1)^2}$ respectively.  Fix $\mathbf{k} \in K_1(\cM^{\star 4})$ and consider the sum
\begin{equation}\label{eq:sum}
\sum_{\substack{\mathbf{t}'' \in K_1((\cM^\beta)^{\star 4})[\beta] \\ F(\mathbf{t}'') = 0}}
\prod_{s=1}^4 \theta^{\Theta_{\cM^\beta}}_{j_s+t''_s} (0_{B}),
\end{equation}
where $\mathbf{j} \in K_1((\cM^{\beta})^{\star 4})[n]$ is the unique index such that $F(\mathbf{j}) = \mathbf{k}$.

In order to evaluate the above sum, we have to make sense of the terms we substitute for each $\theta^{\Theta_{\cM^\beta}}_{j_s+t''_s} (0_{B})$. Let $t' \in K_1(\cM^\beta)[\beta] \subset B$ be the unique preimage of $t$ under $\widehat{f}$. Let $u_s$ be such that $t''_s = u_s t'$, then we evaluate \eqref{eq:sum} by substituting $\theta^{\Theta_{\cL}}_{\widehat{f}(j_s)}(\widetilde{u_st})$ for $\theta^{\Theta_{\cM^\beta}}_{j_s + u_st'}(0_B)$.

Note that since $\widetilde{t}$ is an excellent lift, we have that $\widetilde{0}_A = \chainmult(\ell,\widetilde{t})$, so that the lifts $\widetilde{0t}$ and~$\widetilde{0}_A$ are the same.

Recall that the kernel $\widehat{G}$ of $\widehat{f}$ equals $K_2(\cM^\beta)[\beta]$, so that $\widehat{f}$ induces an isomorphism of $K_1(\cM^\beta)[\beta]$ onto $\widehat{f}(K_1(\cM^\beta)[\beta]) = G$, the kernel of $f$. Hence, looping over~$\mathbf{t}'' = (t_1'', \dots , t_4'')\in K_1((\cM^\beta)^{\star 4})[\beta] \cap \ker (F)$ for computing $\displaystyle \prod_{s=1}^4 \theta^{\Theta_{\cM^\beta}}_{j_s+u_st'} (0_B) = \prod_{s=1}^4 \theta^{\Theta_{\cL}}_{\widehat{f}(j_s)} (\widetilde{u_st})$ (again, writing $t''_s = u_s t'$) is equivalent to looping over $\mathbf{t} = (t_1, \dots, t_4) \in G^4 \cap \ker (F)$ and computing $\displaystyle \prod_{s=1}^4 \theta^{\Theta_{\cL}}_{\widehat{f}(j_s)} (\widetilde{t_s})$. Therefore, using \eqref{eq:kerFcapG4}, the sum \eqref{eq:sum} becomes

\begin{equation}
\sum_{t_1,t_2 \in G}\theta^{\Theta_{\cL}}_{i_1} (\widetilde{\alpha_1 t_1 -\alpha_2 t_2})\theta^{\Theta_{\cL}}_{i_2} (\widetilde{\alpha_2 t_1 +\alpha_1 t_2}) 
\theta^{\Theta_{\cL}}_{i_3} (\widetilde{\alpha_3 t_1 -\alpha_4 t_2})  \theta^{\Theta_{\cL}}_{i_4} (\widetilde{\alpha_4 t_1 +\alpha_3 t_2}), 
\end{equation}
where $i_s=\widehat f(j_s)\in K_1(\cL)$. Writing the action of $\alpha_i$ on $t$ as multiplication by the scalar $a_i$  (considered modulo $\ell$) and writing $t_1 = u_1t$ and $t_2 = u_2t$, with $0 \le u_1, u_2 \le \ell-1$, we are reduced to compute the sum

\begin{equation}\label{eq:0_r=4}
\sum_{0\le u_1, u_2 \le \ell-1}\theta^{\Theta_{\cL}}_{i_1} (\widetilde{(a_1u_1-a_2u_2)t})\theta^{\Theta_{\cL}}_{i_2} (\widetilde{(a_2u_1+a_1u_2)t}) 
\theta^{\Theta_{\cL}}_{i_3} (\widetilde{(a_3u_1-a_4u_2)t})  \theta^{\Theta_{\cL}}_{i_4} (\widetilde{(a_4u_1 + a_3u_2)t}).
\end{equation}
We know that if we compute \eqref{eq:sum} as 
\begin{equation}\label{eq:0_r=4_right}
\sum_{0\le u_1, u_2 \le \ell-1}\theta^{\Theta_{\cL}}_{i_1} (\widetilde{(a_1u_1-a_2u_2)t}_{\text{right}})\theta^{\Theta_{\cL}}_{i_2} (\widetilde{(a_2u_1+a_1u_2)t}_{\text{right}}) 
\theta^{\Theta_{\cL}}_{i_3} (\widetilde{(a_3u_1-a_4u_2)t}_{\text{right}})  \theta^{\Theta_{\cL}}_{i_4} (\widetilde{(a_4u_1 + a_3u_2)t}_{\text{right}}),
\end{equation}
then we compute the theta null point of $B^4$ for $\widetilde{\Theta}_{\cM^{\star 4}}$ correctly. We will show that with our choice of lifts of $t, \dots , (\ell-1)t$, the sums \eqref{eq:0_r=4} and \eqref{eq:0_r=4_right} are the same.

\begin{lem}\label{lem:lambda}
Fix $u_1, u_2 \in \{ 0, \dots , \ell-1\}$. Then the terms 
$$\theta^{\Theta_{\cL}}_{i_1} (\widetilde{(a_1u_1-a_2u_2)t})\theta^{\Theta_{\cL}}_{i_2} (\widetilde{(a_2u_1+a_1u_2)t}) 
\theta^{\Theta_{\cL}}_{i_3} (\widetilde{(a_3u_1-a_4u_2)t})  \theta^{\Theta_{\cL}}_{i_4} (\widetilde{(a_4u_1 + a_3u_2)t}) \text{ and}$$
$$\theta^{\Theta_{\cL}}_{i_1} (\widetilde{(a_1u_1-a_2u_2)t}_{\text{right}})\theta^{\Theta_{\cL}}_{i_2} (\widetilde{(a_2u_1+a_1u_2)t}_{\text{right}}) 
\theta^{\Theta_{\cL}}_{i_3} (\widetilde{(a_3u_1-a_4u_2)t}_{\text{right}})  \theta^{\Theta_{\cL}}_{i_4} (\widetilde{(a_4u_1 + a_3u_2)t}_{\text{right}})$$
are the same. That is to say, fixing an excellent lift $\widetilde{t}$ of $t$ and substituting the coordinates of $\widetilde{0}_A, \widetilde{t}, \widetilde{2t} = \chainmult(2, \widetilde{t}), \dots , \widetilde{(\ell-1)t} = \chainmult(\ell-1, \widetilde{t})$ in the formula \eqref{eq:sum}, we correctly compute an affine lift of the theta null point of $B^4$ for $\widetilde{\Theta}_{\cM^{\star 4}}$.
\end{lem}

\begin{proof}
The two terms differ by
$$\zeta_t^{(a_1u_1-a_2u_2)^2 + (a_2u_1+a_1u_2)^2 + (a_3u_1-a_4u_2)^2 + (a_4u_1 + a_3u_2)^2}.$$
But
$$(a_1u_1-a_2u_2)^2 + (a_2u_1+a_1u_2)^2 + (a_3u_1-a_4u_2)^2 + (a_4u_1 + a_3u_2)^2 = (a_1^2+ \cdots+ a_4^2)( u_1^2+ u_2^2)$$
and $(a_1^2+ \cdots + a_4^2)$ is a multiple of $\ell$, since it is given by the scalar of the action of $\beta = \alpha_1^2 + \cdots + \alpha_4^2$ on $t$.
\end{proof}

\subsection{Modification of $\widetilde{\Theta}_{\cM^{\star r}}$ on $(B^r, \cM^{\star r})$ via a metaplectic isomorphism}
\label{subsec:metaplect}
The theta null point for the symmetric theta structure $\widetilde{\Theta}_{\cM^{\star r}}$ on $\cM^{\star r}$ from Section~\ref{subsec:complifts} does not automatically recover the theta null point for 
$(B, \cM)$. It would do so if $\widetilde{\Theta}_{\cM^{\star r}}$ were of the form $\Theta_{\cM} \star \Theta_{\cM^{\star (r-1)}}$ 
for theta structures $\Theta_{\cM}$ and $\Theta_{\cM^{\star (r-1)}}$ on 
$(B, \cM)$ and $(B^{r-1}, \cM^{\star (r-1)})$ respectively. 

In order to obtain information about a single polarized factor $(B, \cM)$, we need to modify 
$\widetilde{\Theta}_{\cM^{\star r}}$ via a suitably chosen metaplectic automorphism (an automorphism of the corresponding Heisenberg group) so that it has the above form. We explain how to do that now.

\paragraph{Transforming $\widetilde{\Theta}_{\cM^{\star r}}$ to a product theta structure via a metaplectic automorphism.}

\begin{lem}\label{lem:exist}
There exists a metaplectic automorphism $M \in \Aut_{k^\times}(\cH(\delta_\cM^{\star r}))$ such that the theta structure 
$\widetilde{\Theta}_{\cM^{\star r}} \circ M$ is a product theta structure. 
\end{lem}

\begin{proof}
There exists a symmetric theta structure $\Theta_{\cM}$ on $(B, \cM)$. We can then form the ($r$-fold) product theta structure $\Theta_{\cM^{\star r}} = \Theta_{\cM} \star \dots \star \Theta_{\cM}$ on $(B^r, \cM^{\star r})$. Define $M := \widetilde{\Theta}_{\cM^{\star r}}^{-1} \circ \Theta_{\cM^{\star r}}$, which is clearly an element of $\Aut_{k^\times}(\cH(\delta_\cM^{\star r}))$ and satisfies the above property.  
\end{proof}

\paragraph{Explicit computation of a metaplectic automorphism $M$.}\label{par:unfolding}
Lemma~\ref{lem:exist} shows that $\widetilde{\Theta}_{\cM^{\star r}}$ can be transformed into a product theta structure via an automorphism $M \in \Aut_{k^\times}(\cH(\delta_\cM^{\star r}))$, but does not provide such an $M$. 
We will look for an automorphism $M$ that transforms $\widetilde{\Theta}_{\cM^{\star r}}$ into a theta structure of the form $\Theta_{\cM} \star \cdots \star \Theta_{\cM}$.
By Proposition~\ref{prop:symtheta}, to give a symmetric theta structure $\cH(\delta_\cM^{\star r}) \ra \cG(\cM^{\star r})$ it suffices to give a symplectic isomorphism $K(2 (\delta_\cM^{\star r})) \ra K((\cM^{\star r})^2)$. Observe that $K(2 (\delta_\cM^{\star r})) = K((2\delta_\cM)^{\star r}) = K(2\delta_\cM) \times \cdots \times K(2\delta_\cM)$ and that $K((\cM^{\star r})^2) = K((\cM^2)^{\star r}) = K(\cM^2) \times \cdots \times K(\cM^2)$. We have the following proposition.

\begin{prop}
Suppose that a symplectic isomorphism $K(2 (\delta_\cM^{\star r})) \ra K((\cM^{\star r})^2)$ is of product form. Then the induced symmetric theta structure $\cH(\delta_\cM^{\star r}) \ra \cG(\cM^{\star r})$ is also of product form.
\end{prop}
\begin{proof}
Observe that the induced symplectic isomorphism $K(\delta_\cM^{\star r}) \ra K(\cM^{\star r})$ is of product form, then use Lemma \ref{lem:productthetastructure}.
\end{proof}

We will now explain how to find the symplectic isomorphism $K(2 (\delta_\cM^{\star r})) \ra K((\cM^{\star r})^2)$ that turns~$\widetilde{\Theta}_{\cM^{\star r}}$ into a product theta structure. We will therefore work on $A^r$ instead.
 
Let $\Theta_\cL$ be as in Section \ref{subsec:isogthmapp} and let $\Theta_{\cL^{2}}$ be a symmetric theta structure on $(A, \cL^2)$ such that $(\Theta_\cL, \Theta_{\cL^2})$ is a compatible pair of symmetric theta structures for $(\cL, \cL^{2})$, as defined in \cite[p.317]{mumford:eq1}. Consider  the $r$-fold product symplectic basis $\{e_i^{'}, e_i^{''}\}_{i =1}^{gr}$ for the $2n$-torsion points of $(A^r, (\cL^2)^{\star r})$ determined by (the $r$-fold product of) the theta structure $\Theta_{\cL^2}$. Let $\{x_i', x_i''\}_{i=1}^{gr}$ be the $r$-fold product basis on $K((\cM^{2\beta})^{\star r})[2n]$ corresponding to the $r$-fold product theta structure $\Theta_{(\cM^{2\beta})^{\star r}}$. We know that $\widehat{f}^{\star r}(x_i') = e_i'$ (and same for $x_i''$). Let $y_i' = F(x_i')$ and $y_i'' = F(x_i'')$. The basis 
$\{y_i', y_i''\}_{i=1}^{gr}$ is not of product form, but is symplectic for the Weil pairing on $B^r[2n]$. Define $f_i' = (f^{\star r})^{-1}(y_i')$ and 
$f_i'' = (f^{\star r})^{-1}(y_i'')$ where we use the fact that $f^{\star r}|_{K((\cL^{2\beta})^{\star r})[2n]}$ is invertible. 
Note that $\{f_i', f_i''\}_{i=1}^{gr}$ is a non $r$-fold symplectic basis for 
$K((\cL^2)^{\star r})$. The following diagram might be helpful. Note that $\{e_i^{'}, e_i^{''}\}$ and $\{f_i', f_i''\}$ are bases of $A^r[2n]$, whereas $\{x_i', x_i''\}$ and $\{y_i', y_i''\}$ are bases of $B^r[2n]$.
$$\xymatrix{ 
\{e_i', e_i''\} \mbox{: $r$-fold, symplectic for } e_{(\cL^2)^{\star r}}\ar@{-->}[d]_{F\beta^{-1}}& \ar[l]_{\widehat{f}^{\star r}} \{x_i', x_i''\}\mbox{: $r$-fold, symplectic for }e_{(\cM^{2\beta})^{\star r}}\ar@{-->}[d]_{F}\\
\{f_i', f_i''\}\mbox{: non $r$-fold, symplectic for } e_{(\cL^{2\beta})^{\star r}}\ar[r]_{f^{\star r}}&\{y_i', y_i''\} \mbox{: non $r$-fold, symplectic for } e_{(\cM^2)^{\star r}}
}$$
A simple diagram chasing shows that indeed
$F\beta^{-1} (e_i') = f_i'$ 
and $F\beta^{-1} (e_i'') = f_i''$, so the non $r$-fold basis $\{f_i', f_i''\}$ of $A^r[2n]$ can be computed from the basis $\{e_i', e_i''\}$.

\begin{lem}\label{lem:switch}
Suppose that $\overline S_A \in \Sp(A^r[2n])$ is a symplectic automorphism (for the pairing $e_{(\cL^{2\beta})^{\star r}}$ restricted to $A^r[2n]$) such that the basis 
$\{\overline S_A(f_i'), \overline S_A(f_i'')\}_{i=1}^{gr}$ is an $r$-fold product basis. Then the basis $\{\overline S_B(y_i'), \overline{S}_B(y_i'')\}_{i=1}^{gr}$ is an $r$-fold product symplectic basis of $K((\cM^2)^{\star r})$ for $e_{(\cM^2)^{\star r}}$, where $\overline{S}_B = f^{\star r} \circ \overline{S}_A \circ (f^{\star r})^{-1} \in \Sp(B^r[2n])$. 
\end{lem} 

\begin{proof}
Write 
$$
\overline S_B (y_i') = f^{\star r}(\overline S_A (f_i')) \qquad \text{and} \qquad  \overline S_B (y_i'') = f^{\star r}(\overline S_A (f_i'')). 
$$
Since $\{\overline S_A(f_i'), \overline S_A (f_i'')\}$ is an $r$-fold product basis, so is 
$\{\overline S_B(y_i'), \overline S_B(y_i'')\}$. 
\end{proof}

\noindent To compute such an $\overline{S}_A \in \Sp(A^r[2n])$ in practice, we do the following: 

\begin{itemize}
\item Let $M_{F\beta^{-1}} \in \GL_{2gr}(\Z / 2n \Z)$ be the matrix corresponding to the action of $F \beta^{-1}$ on $\{e_i', e_i''\}_{i=1}^{gr}$ (i.e. the one corresponding to the change of basis from $\{e_i', e_i''\}_{i=1}^{gr}$ to $\{f_i', f_i''\}_{i=1}^{gr}$). 

\item For each $N \in \GL_{2g}(\Z/2n\Z)$ let $\Delta(N)$ be the image of $N$ under the 
standard diagonal embedding $$\Delta \colon \GL_{2g}(\Z/2n\Z) \hookrightarrow \GL_{2gr}(\Z/2n\Z).$$
Then test whether $\overline{S}_A = \Delta(N) M_{F\beta^{-1}}^{-1} \in \Sp_{2gr}(\Z / 2n\Z)$ where the symplectic group $\Sp_{2gr}(\bZ/2n\bZ) \subset \GL_{2gr}(\Z / 2n\Z)$ is defined with respect to the pairing $\left( \begin{array}{cc} & I_{gr} \\ -I_{gr} &\end{array} \right)$ for the standard basis of $(\Z / 2n\Z)^{2gr}$. 
\end{itemize}
The complexity of the computation depends on  $g = \dim A$, $r$ and $n$ and has no dependency in $q = \# k$ and~$\ell$. In practice we will have $n=2$ and for most cases $r = 4$.

\subsection{Transforming theta coordinates under metaplectic automorphisms}\label{subsec:transform}
In this section we explain how to transform canonical theta coordinates under metaplectic automorphisms. 
We first treat the case where $(A, \cL)$ is a $g$-dimensional abelian variety over a field $k$ with a totally symmetric ample line bundle $\cL$ of type $\delta_A = (4, \dots, 4) \in \Z^g$, i.e. $\cL$ is algebraically equivalent to the 4th tensor power of a principal polarization. Let $\Theta_{\cL}$ and $\Theta_{\cL}'$ be two symmetric theta structures on $(A, \cL)$. 
We will explain how to compute the theta null point of $A$ with respect to $\Theta_{\cL}'$ from the theta null point of $A$ with respect to $\Theta_{\cL}$. More generally, for a geometric point $x \in A(\bar{k})$, we will 
explain how to express the canonical theta coordinates of $x$ with respect to $\Theta_{\cL}'$ out of the canonical theta coordinates of $x$ with respect to $\Theta_{\cL}$.  Then we will stick to the case $(A, \cL)$ where $\cL$ is totally symmetric of type $(2,\dots,2)$ and see how to relate theta coordinates on $A$ with respect to different symmetric theta structures, using the transformation law developed for $\cL^2$ (i.e. the one for totally symmetric line bundles of type $(4, \dots , 4)$).

\paragraph{Analytic transformation formula of Igusa for totally symmetric line bundles of type $(4,\dots,4)$.}
Here we restrict to the case $k = \C$. Let $(A = \C^g / \Lambda,\cL)$ be a complex abelian variety of dimension~$g$ with a totally symmetric line bundle $\cL$ on $A$ of type $\delta_A = (4,\dots,4)$ and let $\Theta_{\cL}$ and $\Theta_{\cL}'$ be two symmetric theta structures on $(A, \cL)$. The theta structures $\Theta_{\cL}$ and $\Theta_{\cL}'$ induce level $\delta_A$-structures on $(A, \cL)$ (see \cite[Ch.8.3]{birkenhake-lange}), which in return are induced by period matrices $\Omega \in \cH_g$ and $\Omega' \in \cH_g$ respectively. According to \cite[Prop.3.1.6]{Cosset} there exist bases $\left \{ \displaystyle\vartheta\begin{bmatrix}\mathbf{a} \\ \mathbf{b}\end{bmatrix} \right \}$ and $\left \{ \displaystyle\vartheta'\begin{bmatrix}\mathbf{a} \\ \mathbf{b}\end{bmatrix} \right \}$ of $\Gamma(A, \cL)$ determined by $\Omega$ and $\Omega'$ respectively, with indices $\mathbf{a}=(a_1,\ldots,a_g), \mathbf{b}=(b_1,\ldots,b_g) \in \frac{1}{2}\Z^g$ that run over all  representatives of elements of $\displaystyle \left ( \frac{1}{2}\bZ/\bZ \right )^g$. We want to see how these bases are related. Recall that $\Sp_{2g}(\Z)$ acts on $\C^g \times \cH_g$ via
\begin{equation} 
S \cdot (z, \Omega) = (S_{\Omega} \cdot z, S \cdot \Omega),
\end{equation}
where $S_{\Omega} \cdot z = \leftexp{t}{(C\Omega + D)}^{-1} z$ and $S \cdot \Omega=(A\Omega+B)(C\Omega+D)^{-1}$.
Igusa \cite[\S 5, Thm.2]{Igusa1972} proved that for any $(z,\Omega)\in \C^g\times \cH_g$ the following relation holds:

\begin{equation}
\displaystyle\theta\begin{bmatrix}\mathbf{a'} \\ \mathbf{b'}\end{bmatrix}(S_\Omega \cdot z,S \cdot \Omega)=\xi_{S}\cdot \xi_{z,S} \cdot \xi_{\mathbf a, \mathbf b}\cdot \sqrt{\det (C\Omega+D)} \cdot \theta\begin{bmatrix}\mathbf{a}\\\mathbf{b}\end{bmatrix}(z,\Omega), 
\end{equation}
where 
\begin{itemize}
\item $\xi_S$ is an $8$-th root of unity, 
\item $\xi_{z,S}=\exp \left (\pi i z^t(C\Omega+D)^{-1}Cz \right )$ and is $1$ for the case of $z=0$.
\item $\xi_{\mathbf{a}, \mathbf{b}}=\exp\left (-\pi i (\mathbf{a}^tAB^t \mathbf{a}+\mathbf{b}^tCD^t \mathbf{b})-2\pi i (A^t \mathbf{a}+C^t \mathbf{b}+\mathbf{e'})^t\mathbf{e''}-2 \pi i \mathbf{a}^tBC^t\mathbf{b}\right )$ 
\item  
$\begin{pmatrix}
\mathbf{a'}\\
\mathbf{b'}
\end{pmatrix}:=(S^t)^{-1}\cdot 
\begin{pmatrix}
\mathbf{a}-\mathbf{e'}\\
\mathbf{b}-\mathbf{e''}
\end{pmatrix}$, where $\displaystyle \mathbf{e}' = \frac{1}{2} \diag \left ( A^t C \right )$ and 
$\displaystyle \mathbf{e''} = \frac{1}{2} \left ( D^t B\right )$.   
\end{itemize}

Now consider the metaplectic automorphism $M:= \Theta_{\cL}^{-1} \circ \Theta_{\cL}' \in \Aut_{\C^\times}(\cH(\delta_A))$. Let $\overline S \in \Sp(K(\delta_A))$ be the induced symplectic automorphism. With respect to the standard symplectic basis of $K(\delta_A)$, the automorphism $\overline S$ can be expressed as a $2g$-by-$2g$ symplectic matrix with coefficients in $\Z / 4\Z$. Lift~$\overline{S}$ to a symplectic matrix $S = \left( \begin{array}{cc} A & B \\C & D \end{array}\right)\in \Sp_{2g}(\Z)$ in order that $\Omega' = S \cdot \Omega$. Then the canonical bases are related as

\begin{equation}\label{eq:isS}
\displaystyle\vartheta'\begin{bmatrix}\mathbf{a'} \\ \mathbf{b'}\end{bmatrix}(z)=\xi_{S}\cdot \xi_{z,S} \cdot \xi_{\mathbf a, \mathbf b}\cdot \sqrt{\det (C\Omega+D)} \cdot \vartheta\begin{bmatrix}\mathbf{a}\\\mathbf{b}\end{bmatrix}(z). 
\end{equation}

The important point is that the factor $\lambda=\xi_S \cdot \xi_{z,S} \cdot \sqrt{\det(C\Omega+D)}$ is independent of $\mathbf{a}$ and $\mathbf{b}$, so it gets absorbed when working with projective coordinates. 

\paragraph{Algebraic transformation formula for totally symmetric line bundles of type $(2,\dots,2)$.}
We now want to give a transformation law for canonical bases on $(A, \cL)$, where $\cL$ is totally symmetric of type $\delta_A = (2, \dots, 2)$. Suppose we have fixed two pairs of compatible symmetric theta structures $(\Theta_{\cL}, \Theta_{\cL^2})$ and $(\Theta_{\cL}', \Theta_{\cL^2}')$ for $(\cL, \cL^2)$ (as defined in \cite[p.317]{mumford:eq1}). Note that by \cite[\S2, Prop.7]{mumford:eq1} every symmetric theta structure $\Theta_{\cL}$ on $(A, \cL)$ can be extended to a compatible pair of symmetric theta structures $(\Theta_{\cL}, \Theta_{\cL^2})$ on $(\cL, \cL^2)$. 
Let $\overline{S} \in \Sp(K(\delta_A))$ and $\overline{S}_2 \in \Sp(K(2\delta_A))$ be the symplectic automorphisms induced by the metaplectic automorphisms $\Theta_{\cL}^{-1} \circ \Theta_{\cL}'$ and $\Theta_{\cL^2}^{-1} \circ \Theta_{\cL^2}'$ respectively. Then $\overline{S}_2$ restricts to $\overline{S}$ on $K(\delta_A)$. Denote by $S_2$ an arbitrary lift of $\overline{S}_2$ to $\Sp_{2g}(\Z)$.

Given the theta coordinates of a point $z$ with respect to the theta structure $\Theta_{\cL}$, we will explain how to compute the theta coordinates of the point $z'=S_{\Omega}\cdot z$  with respect to the theta structure~$\Theta'_{\cL}$. First note that the theta structures $\Theta_{\cL}$ and $\Theta_{\cL}'$ determine the squares of the theta coordinates for the compatible theta structures $\Theta_{\cL^2}$ and $\Theta_{\cL^2}'$ respectively. Hence, we write the analogue of \eqref{eq:isS} for the squares of the algebraic theta coordinates for $\Theta_{\cL^2}$ and $\Theta_{\cL^2}'$. 
There exists a constant $\lambda$ such that for every $\mathbf{a},\mathbf{b},\mathbf{a'},\mathbf{b'}\in Rpr\left(\frac{1}{2}\Z/\Z\right)^g$, with $\scalefont{0.85}{\begin{pmatrix}
\mathbf{a'}\\
\mathbf{b'}
\end{pmatrix}=(S_2^t)^{-1}\cdot 
\begin{pmatrix}
\mathbf{a}-\mathbf{e'}\\
\mathbf{b}-\mathbf{e''}
\end{pmatrix}}$, 
\begin{equation}\label{eq:isS_F}
\displaystyle\left(\theta^{\Theta_{\cL^2}'}_{\mathbf a', \mathbf b'}(z')\right)^2=\displaystyle \lambda\cdot \xi^2_{\mathbf a, \mathbf b}\cdot\left(\theta^{\Theta_{\cL^2}}_{\mathbf a, \mathbf b}(z)\right)^2.
\end{equation}

The theta coordinates of any point $z$ of $(A, \cL, \Theta_{\cL})$ and indexed by $\bi$ are transformed into squares of type $2\delta_A$ theta coordinates, denoted by $\theta^{\Theta_{\cL^2}}_{\mathbf{a},\mathbf{b}}(z)$ with indices 
$\displaystyle \mathbf{a}, \mathbf{b}\in Rpr\left (\frac{1}{2}\bZ/\bZ \right )^g$, via \cite[eq.(3.13)]{Cosset}:

\begin{equation}\label{eq:d_2To22}
\left (\theta^{\Theta_{\cL^2}}_{\mathbf{a}, \mathbf{b}}(z) \right )^2=\frac{1}{2^{g}}
\sum_{\mathbf{i}\in \left ( \frac{1}{2}\bZ/\bZ \right)^g}
\exp(4 \pi i\mathbf{a}^t\bi)\theta^{\Theta_{\cL}}_{\mathbf{b}+\bi}(z)\theta^{\Theta_{\cL}}_{\bi}(0). 
\end{equation}

Then apply \eqref{eq:isS_F} with $\scalefont{0.85}{\begin{pmatrix}
\mathbf{a'}\\
\mathbf{b'}
\end{pmatrix}=(S_2^t)^{-1}\cdot 
\begin{pmatrix}
\mathbf{a}-\mathbf{e'}\\
\mathbf{b}-\mathbf{e''}
\end{pmatrix}}$ for given indices $\mathbf{a},\mathbf{b}\in Rpr\left(\frac{1}{2}\Z/\Z\right)^g$.
Finally we use \cite[eq.(3.12)]{Cosset} to go back to theta coordinates for $\Theta_{\cL}'$ and deduce: 
\begin{equation}
\begin{array}{lll}
\displaystyle\theta^{\Theta_{\cL}'}_{\mathbf{b'}}(z') \theta^{\Theta_{\cL}'}_{\mathbf{0}}(0)& =&\displaystyle \sum_{\mathbf{a'} \in \left ( \frac{1}{2}\Z / \Z\right )^g} \left ( \theta_{\mathbf{a'}, \mathbf{b'}}^{\Theta_{\cL^2}'}(z') \right )^2\\
&=&\displaystyle\sum_{\mathbf{a}\in  \left ( \frac{1}{2}\Z / \Z\right )^g} \lambda\xi_{\mathbf{a},\mathbf{b}}^2\left(\theta_{\mathbf{a},\mathbf{b}}^{\Theta_{\cL^2}}(z)\right)^2\\
&=&\displaystyle\frac{\lambda}{2^{g}}\sum_{\mathbf{a} \in \left ( \frac{1}{2}\Z / \Z\right )^g}\xi^2_{\mathbf{a},\mathbf{b}}
\sum_{\mathbf{i}\in \left ( \frac{1}{2}\bZ/\bZ \right)^g}
\exp(4\pi i\mathbf{a}^t\mathbf{i})\theta^{\Theta_{\cL}}_{\mathbf{b}+\mathbf{i}}(z)\theta^{\Theta_{\cL}}_{\mathbf{i}}(0). 
\end{array}\end{equation}

\noindent We summarise the above in the following proposition.

\begin{prop}\label{prop:transform}
Let $(A, \cL)$ be a polarized abelian variety of dimension $g$ with a totally symmetric line bundle $\cL$ of type $(2, \dots , 2)$. Let $(\Theta_{\cL}, \Theta_{\cL^2})$ and $(\Theta_{\cL}', \Theta_{\cL^2}')$ be two pairs of compatible symmetric theta structures for $(\cL, \cL^2)$. Let $\overline{S} \in \Sp(\Z / 4\Z)$ be the symplectic automorphism induced by $\Theta_{\cL^2}^{-1} \circ \Theta_{\cL^2}'$ and let 
$S = \left( \begin{array}{cc} A & B \\C & D \end{array}\right)\in \Sp_{2g}(\Z)$ be an arbitrary lift of $\overline{S}$. Let $\mathbf{e}' = \frac{1}{2} \diag ( A^t C)$ and 
$ \mathbf{e''} = \frac{1}{2} \diag( D^t B)$. Then there exists a constant $\lambda$ such that 
 \begin{equation}\label{eq:d_Stransf}
\begin{array}{lll}
\displaystyle\theta^{\Theta_{\cL}'}_{\mathbf{b'}}(z') \theta^{\Theta_{\cL}'}_{\mathbf{0}}(0)& =&\displaystyle\frac{\lambda}{2^{g}}\sum_{\mathbf{a} \in \left ( \frac{1}{2}\Z / \Z\right )^g}\xi^2_{\mathbf{a},\mathbf{b}}
\sum_{\mathbf{i}\in \left ( \frac{1}{2}\bZ/\bZ \right)^g}
\exp(4 \pi i\mathbf{a}^t\mathbf{i})\theta^{\Theta_{\cL}}_{\mathbf{b}+\mathbf{i}}(z)\theta^{\Theta_{\cL}}_{\mathbf{i}}(0),
\end{array}\end{equation}
\end{prop}
\noindent where $\mathbf{a},\mathbf{b},\mathbf{a'},\mathbf{b'}\in Rpr\left(\frac{1}{2}\Z/\Z\right)^g$ satisfy $\scalefont{0.85}{\begin{pmatrix}
\mathbf{a'}\\
\mathbf{b'}
\end{pmatrix}=(S^t)^{-1}\cdot 
\begin{pmatrix}
\mathbf{a}-\mathbf{e'}\\
\mathbf{b}-
\mathbf{e''}
\end{pmatrix}}$.

\subsection{Computing the theta null point of $(B, \cM, \Theta_\cM)$}\label{subsec:unfolding} 
Let $(A, \cL)$ be a polarized abelian $g$-fold with $\cL$ totally symmetric of type $\delta = (2,\dots,2)$. Suppose we have fixed a compatible pair of symmetric theta structures $(\Theta_\cL, \Theta_{\cL^2})$ for $(\cL, \cL^2)$. Let $\overline{S}_A \in \Sp(A^r[4])$ be computed as in Section \ref{par:unfolding}, where the symplectic pairing on $A^r[4]$ is the restriction of $e_{(\cL^{2\beta})^{\star r}}$ to $A^r[4]$. Via $\overline{\Theta}_{(\cL^{2\beta})^{\star r}}$ we might see $\overline{S}_A$ as a matrix $\overline{S} \in \Sp_{2gr}(\Z / 4\Z)$.

On the other hand, $(\Theta_\cL, \Theta_{\cL^2})$ induces the compatible pair of symmetric theta structures $(\Theta_\cM, \Theta_{\cM^2})$ for the totally symmetric line bundles $(\cM, \cM^2)$ on $B$. According to \cite[Rem.2, p.318]{mumford:eq1} there exists a symmetric theta structure $\widetilde{\Theta}_{(\cM^2)^{\star r}}$ above $\widetilde{\Theta}_{\cM^{\star r}}$ such that the symplectic automorphism induced by $\widetilde{\Theta}_{(\cM^2)^{\star r}}^{-1} \circ \Theta_{(\cM^2)^{\star r}}$ (when expressed in the canonical symplectic basis of $K(2\delta^{\star r})$) equals $\overline{S}$.

Now suppose we are given the theta null point $(\theta^{\widetilde{\Theta}_{\cM^{\star r}}}_{\mathbf{b}} (0_{B^r}))_{\mathbf{b}\in K_1(\cM^{\star r})}$ of $(B^r, \cM^{\star r}, \widetilde{\Theta}_{\cM^{\star r}})$. If we apply Proposition \ref{prop:transform} to $(B^r, \cM^{\star r})$ and the two pairs of compatible symmetric theta structures $(\widetilde{\Theta}_{\cM^{\star r}}, \widetilde{\Theta}_{(\cM^2)^{\star r}})$ and $(\Theta_{\cM^{\star r}}, \Theta_{(\cM^2)^{\star r}})$, then we can compute the theta null point for the product theta structure $\Theta_{\cM^{\star r}}$ by \eqref{eq:d_Stransf}. Let $S = \left( \begin{array}{cc} A & B \\C & D \end{array}\right)\in \Sp_{2gr}(\Z)$ be an arbitrary lift of~$\overline{S}$. Let $\mathbf{e}' = \frac{1}{2} \diag ( A^t C)$ and 
$ \mathbf{e''} = \frac{1}{2} \diag( D^t B)$. Then there exists a constant $\lambda$ (containing the factor~$\theta^{\Theta_{\cM^{\star r}}}_{\mathbf{0}}(0_{B^r})$) for which the new theta coordinates are
 \begin{equation}\label{eq:isS_gen}
\theta^{\Theta_{\cM^{\star r}}}_{\mathbf{b'}}(0_{B^r}) =\frac{\lambda}{2^{2r}}\sum_{\mathbf{a} \in \left ( \frac{1}{2}\Z^g / \Z^g \right )^r} \xi^2_{\mathbf{a},\mathbf{b}}\sum_{\mathbf{i}\in \left ( \frac{1}{2}\Z^g/\Z^g \right)^r}
\exp(4 \pi i\mathbf{a}^t\mathbf{i})\theta^{\widetilde \Theta_{\cM^{\star r}}}_{\mathbf{b}+\mathbf{i}}(0_{B^r})\theta^{\widetilde \Theta_{\cM^{\star r}}}_{\mathbf{i}}(0_{B^r}), 
\end{equation}
where $\mathbf{a},\mathbf{b},\mathbf{a'},\mathbf{b'}\in Rpr  \left ( \frac{1}{2}\Z^g/\Z^g \right )^r$ are related by $\scalefont{0.85}{\begin{pmatrix}
\mathbf{a'}\\
\mathbf{b'}
\end{pmatrix}=\begin{pmatrix} D & -C\\ -B & A \end{pmatrix}\cdot 
\begin{pmatrix}
\mathbf{a}-\mathbf{e'}\\
\mathbf{b}-
\mathbf{e''}
\end{pmatrix}}$, and 
$$
\xi_{\mathbf{a}, \mathbf{b}}=\exp \left(-\pi i (\mathbf{a}^t A B^t \mathbf{a}+\mathbf{b}^t C D^t \mathbf{b})- 2\pi i(A^t \mathbf{a}+C^t \mathbf{b}+\mathbf{e'})^t\mathbf{e''}-2\pi i \mathbf{a}^t B C^t\mathbf{b}\right).
$$ 

Finally, we have that for any $\mathbf{k} = (k_1, \dots , k_r) \in Rpr  \left ( \frac{1}{2}\Z^g/\Z^g \right )^r$
$$\theta^{\Theta_{\cM^{\star r}}}_{\mathbf{k}}(0_{B^r}) = \theta^{\Theta_{\cM}}_{k_1}(0_B) \cdot \theta^{\Theta_{\cM^{\star (r-1)}}}_{(k_2, \dots, k_r)}(0_{B^{r-1}}).$$
Fixing $(k_2, \dots, k_r) \in Rpr  \left ( \frac{1}{2}\Z^g/\Z^g \right )^{r-1}$ such that $\theta^{\Theta_{\cM^{\star (r-1)}}}_{(k_2, \dots, k_r)}(0_{B^{r-1}}) \ne 0$, we obtain the projective theta null point $(\theta^{\Theta_\cM}_k (0_B))_{k \in K_1(\cM)}$ for $(B, \cM, \Theta_\cM)$.

\begin{rem}
In order to apply Lemma \ref{lem:switch} and Proposition \ref{prop:transform}, and hence to compute the theta null point of $(B, \cM, \Theta_\cM)$, it is important to note that we need not only a symmetric theta structure $\Theta_\cL$ for $(A, \cL)$, where $\cL$ is totally symmetric and of type $(2, \dots, 2)$, but also a compatible symmetric theta structure $\Theta_{\cL^2}$ for $(A,\cL^2)$. In the case where $A$ is the Jacobian variety of a hyperelliptic curve~$H$ of genus $g$, Thomae's formula \cite[Theorem 8.1.]{mumford1984} gives a way to fix such a pair of compatible symmetric theta structures $(\Theta_\cL, \Theta_{\cL^2})$, and hence to apply the above algorithm.
\end{rem}

%
%
\section{Evaluating the Isogeny on Points}\label{sec:isogpts}

Let $x\in A(k)$ be a point of order $N$ coprime to $\ell$. We want to express the theta coordinates of~$y = f(x) \in B(k)$ with respect to the theta structure 
$\Theta_{\cM}$ in terms of the theta coordinates of $x$ with respect to the theta structure $\Theta_{\cL}$.

\subsection{The preimage of $(y, 0, \dots, 0)$ under $F$}\label{subsec:preim}
Consider the following subgroup of $A^r(k)$:  
$$
X = \{ (a_1 x, \dots, a_r x) \colon (a_1, \dots, a_r) \in \Z^r\}.  
$$ 
As $\beta=F\circ F^t$ is an automorphism of $X$, then both $F$ and $F^t$ are automorphisms of $X$. 

\begin{lem}\label{lem:preim}
Let $(x_1,\ldots,x_r)=F^t(x,0\ldots,0)\in X$ and let $(x_1',\ldots,x'_r)=F^{-1}(x,0,\ldots,0)\in X$. 
If~$y_i=f(x'_i)\in B$
then, $$(f(x),0,\ldots,0)=F(y_1,\ldots,y_r) \text{ and } \widehat f^{\star r}(y_1,\ldots,y_r)=(x_1,\ldots,x_r).$$
\end{lem}

\begin{proof}
Let $(y_1',\ldots,y_r')=F^t(y,0,\ldots,0)$.
Then it follows:
\[
\xymatrix{
(x_1', \dots, x'_r) & \in & A^r(k)  \ar[d]^{F}\ar[r]^{f^{\star r}}& B^r(k) \ar[d]^{F}   & \ni & (y_1, \dots, y_r)  \\ 
(x, 0,\dots, 0) & \in & A^r(k) \ar[d]^{F^t}\ar[r]^{f^{\star r}}& B^r(k) \ar[d]^{F^t} & \ni & (y, 0, \dots, 0)\\
(x_1,\ldots,x_r)& \in & A^r(k) \ar[r]^{f^{\star r}}& B^r(k) & \ni &(y_1',\ldots,y_r').
}
\]
and consequently 
$$
F(y_1,\ldots,y_r)=(y,0\ldots,0)
$$ 
and 
$$
(x_1,\ldots,x_r)=\beta(x_1',\ldots,x_r')=\widehat f^{\star r}(f^{\star r}(x'_1,\ldots,x'_r))=\widehat f^{\star r}(y_1,\ldots,y_r).
$$ 
\end{proof}

\subsection{Computing the point $\left(\theta^{\widetilde{\Theta}_{\cM^{\star r}}}_{\mathbf{k}}(y,0,\dots,0)\right)_{\mathbf{k} \in K_1(\cM^{\star r})}$}\label{sebseq:thetacoord_B^r}

\paragraph{The isogeny theorem for $F$.} Similarly to the computation of the theta null point of $B^r$ with respect to $\widetilde{\Theta}_{\cM^{\star r}}$, we use the isogeny of polarized abelian varieties
$$F \colon (B^r, (\cM^\beta)^{\star r}, \Theta_{(\cM^\beta)^{\star r}}) \ra (B^r, \cM^{\star r}, \widetilde{\Theta}_{\cM^{\star r}})$$
to compute the coordinates of $(y,0,\dots,0)$ with respect to $\widetilde{\Theta}_{\cM^{\star r}}$.
Equation~\eqref{eq:isogF} implies that up to a projective factor $\lambda \in \bar{k}^\times$, one has that for every $\mathbf{k} \in K_1(\cM^{\star r})$,

\begin{equation}\label{eq:isogPts}
\theta^{\widetilde{\Theta}_{\cM^{\star r}}}_{\mathbf{k}}( y, 0, \dots, 0)
= \lambda \cdot \displaystyle\sum_{\substack{\mathbf{t}'' \in K_1((\cM^\beta)^{\star r})[\beta] \\ F(\mathbf{t}'') = 0}}
\prod_{s=1}^r \theta^{\Theta_{\cM^\beta}}_{j_s + t_s''} (y_s),
\end {equation}
where $\mathbf{j} \in K_1((\cM^\beta)^{\star r})[n]$ is the unique index for which $F(\mathbf{j}) = \mathbf{k}$. 

We encounter the same difficulty as in Section~\ref{subsec:complifts}. That is, the input of the algorithm only provides us with the theta coordinates $\left( \theta_i ^{\Theta_\cL} (0_A) \right)_{i \in K_1(\cL)}$ of $0_A$, the theta coordinates $\left( \theta_i ^{\Theta_\cL} (t) \right)_{i \in K_1(\cL)}$ of~$t$ (a generator of the kernel of $f$) and the theta coordinates $\left( \theta_i ^{\Theta_\cL} (x) \right)_{i \in K_1(\cL)}$ of $x$. As in the computation of the theta null point of $B^r$ with respect to $\widetilde{\Theta}_{\cM^{\star r}}$, we would like to use a combination of the isogeny theorem for $F$ and of the isogeny theorem for $\widehat{f}$. Following Lemma \ref{lem:preim}, we consider $x_1 = \alpha_1 x, \dots , x_r = \alpha_r x$ which satisfy $x_1 = \widehat{f}(y_1), \dots , x_r = \widehat{f}(y_r)$, where $(y_1 , \dots , y_r)\in B^r(k)$ is the unique point such that $(y,0, \dots , 0) = F(y_1, \dots , y_r)$. 
As before, we would like to compute \eqref{eq:isogPts} by making substitutions
$$
\theta_{j_s + u_st'}^{\Theta_{\cM^\beta}}(y_s) = \theta_{i_s}^{\Theta_\cL}(x_s + u_st),
$$
where $i_s = \widehat{f}(j_s)$ and where we write $u_st'$ for $t_s''$, but we have to carefully work this out.
One of the major difference in computing 
$\displaystyle \prod_{s=1}^r \theta^{\Theta_{\cM^\beta}}_{j_s + t_s''} (y_s)$
compared to computing 
$\displaystyle \prod_{s=1}^r \theta^{\Theta_{\cM^\beta}}_{j_s+t''_s} (0_{B})$
is that we have to substitute affine coordinates not only of one point $\widetilde{0}_B$ but of $r$ points $\widetilde{y_1}, \dots , \widetilde{y_r}$. Now this might look like a restriction, but it turns out to actually simplify things, since we are free to choose affine lifts of $y_1, \dots , y_r$, a priori without any relation amongst them.

\paragraph{Right lifts and suitable lifts.}\label{par:suitable}
Let $\widetilde{\widehat{f}}$ be the affine lift of $\widehat{f}$ of Section \ref{par:affine_isogthm}. Fixing lifts $\widetilde{x_1}, \dots , \widetilde{x_r}$ of $x_1, \dots, x_r$ determines lifts $\widetilde{y_1} , \dots , \widetilde{y_r}$ of $y_1, \dots , y_r$ by the relation
$$\widetilde{x_1} = \widetilde{\widehat{f}}(\widetilde{y_1}), \dots , \widetilde{x_r} = \widetilde{\widehat{f}}(\widetilde{y_r}).$$
If we were able to compute the lifts $\widetilde{y_1} , \dots , \widetilde{y_r}$, then we could substitute their coordinates in \eqref{eq:isogPts} and compute an affine lift of $(y, 0, \dots , 0)$ for $\widetilde{\Theta}_{\cM^{\star r}}$. Unfortunately we cannot directly compute those from the input.

We make the following crucial observation. The lifts $\widetilde{y_1}, \dots , \widetilde{y_r}$ determine lifts of $x_1 + u_1t, \dots , x_r + u_rt$, where $1 \le u_1, \dots , u_r \le \ell-1$, as
$$\widetilde{x_s + u_st} = \widetilde{\widehat{f}}((1, u_si_{t'}, 0) \cdot \widetilde{y}_s).$$
These lifts will play an important role in the sequel.

\begin{defn} 
Let $\widetilde{0}_A$ and $\widetilde{x_1}, \dots , \widetilde{x_r}$ be fixed affine lifts of $0_A$ and $x_1, \dots , x_r$ respectively. Let $\widetilde{0}_B$ and $\widetilde{y_1}, \dots , \widetilde{y_r}$ be the lifts induced by $\widetilde{\widehat{f}}$ and $\widetilde{0}_A, \widetilde{x_1}, \dots, \widetilde{x_r}$. For $s=1, \dots , r$ and $1 \le u, u_s \le \ell-1$, the induced lifts 
$$\widetilde{\widehat{f}}((1, u_si_{t'}, 0) \cdot \widetilde{y_s}) \text{ and } \widetilde{\widehat{f}}((1, ui_{t'}, 0) \cdot \widetilde{0}_B)$$
are called \textbf{right} lifts of $x_s + u_st$ and $ut$ respectively, and are denoted by
$$\widetilde{x_s + u_st}_{\text{right}} \text{ and } \widetilde{ut}_{\text{right}}.$$
\end{defn}

The terminology is motivated by the following: fixing lifts $\widetilde{0}_A$ and $\widetilde{x_1}, \dots , \widetilde{x_r}$, if we knew the \textbf{right} lifts of all the points $x_s + u_st$ (with respect to $\widetilde{x_1}, \dots , \widetilde{x_r}$), we could patch their affine coordinates together and recover affine lifts of $\widetilde{y_1}, \dots , \widetilde{y_r}$. Surely we do not know the \textbf{right} lift, since when we convert the point $x_s + u_st$ from Mumford to theta coordinates, we do an arbitrary choice of affine lift. Hence, the lift differs from the \textbf{right} lift by a scalar $\lambda_{x_s + u_st} \in \bar{k}^\times$.

But we can show, similar to the notion of excellent lifts from Section \ref{par:excellent}, that the \textbf{right} lifts of $\{x_s + u_st : s = 1, \dots, r \text{ and } u_s = 1, \dots, \ell-1\}$ satisfy some compatibility conditions, reducing the ambiguity when fixing a lift to a choice of an $\ell$-th root of unity.

\begin{defn}[(suitable lifts)]
Let $x \in A(k)$ (not necessarily the input of the algorithm) and let $\widetilde x$ be a fixed affine lift of $x$. Let $\widetilde{t} $ be an affine lift of 
$t\in G$ (we can assume that $\widetilde{t}$ is excellent). We call an affine lift $\widetilde{x + t}$ of $x+t$ \emph{suitable} for $\widetilde{t}$ and $\widetilde{x}$ if
$$
\chainmultadd(\ell, \widetilde{x+t}, \widetilde{t}, \widetilde{x}) = \widetilde{x}. 
$$

\end{defn}
For the algorithms $\chainmult$ and $\chainmultadd$, see Section \ref{par:excellent}. The computation of a suitable lift of $x+t$ is similar to the computation of excellent lifts in the previous section: we take any lift $\widetilde{x+t}$ and search for a scalar $\lambda_{x+t} \in \bar{k}^\times$ such that $\lambda_{x+t} \cdot \widetilde{x+t}$ is suitable. Using \cite[Lem.4.8]{lubicz-robert:isogenies}, we obtain that in order for 
$\lambda_{x+t} \cdot \widetilde{x+t}$ to be suitable, we need 
$$
\widetilde{x} = \lambda_{x+t}^\ell \cdot \chainmultadd(\ell, \widetilde{x+t}, \widetilde{t}, \widetilde{x}). 
$$
The latter determines $\lambda_{x+t}^\ell$ uniquely (since $\chainmultadd(\ell, \widetilde{x+t}, \widetilde{t}, \widetilde{x})$ can be computed using \cite[Alg.4.6]{lubicz-robert:isogenies}). This determines $\lambda_{x+t}$ up to an $\ell$-th root of unity.

We will now show that the notion of suitable lift is the correct notion.

\begin{prop}
Let $\widetilde{0}_A$ and $\widetilde{x_1}, \dots , \widetilde{x_r}$ be fixed affine lifts of $0_A$ and $x_1, \dots , x_r$ respectively. Then, the \textbf{right} lift $\widetilde{t}_{\text{right}}$ is an excellent lift of $t$ and for $2 \le u \le \ell-1$,
$$\widetilde{ut}_{\text{right}} = \chainmult(u, \widetilde{t}_{\text{right}}).$$
Moreover, for $s=1, \dots , r$, the \textbf{right} lift $\widetilde{x_s + t}_{\text{right}}$  of $x_s + t$ is suitable for $\widetilde{t}_{\text{right}}$ and $\widetilde{x_s}$ and we have
$$\widetilde{x_s + u_st}_{\text{right}} = \chainmultadd(u_s, \widetilde{x_s + t}_{\text{right}}, \widetilde{t}_{\text{right}}, \widetilde{x_s}), \text{ for } u_s = 2, \dots , \ell-1.$$
\end{prop}

\begin{proof}
We proved in Proposition \ref{prop:excellentlift} that $\widetilde{t}_{\text{right}}$ is an excellent lift of $t$ and that $\widetilde{ut}_{\text{right}} = \chainmult(u, \widetilde{t}_{\text{right}})$ for $u = 2, \dots , \ell-1$. Let $\widetilde{0}_B$ and $\widetilde{y_1}, \dots, \widetilde{y_r}$ be the lifts of $0_B$ and $y_1, \dots , y_r$ induced by $\widetilde{\widehat{f}}$ and $\widetilde{0}_A, \widetilde{x_1}, \dots , \widetilde{x_r}$ respectively. Then, 
\begin{align*}
& \chainmultadd(u_s, \widetilde{x_s + t}_{\text{right}}, \widetilde{t}_{\text{right}}, \widetilde{x_s}) \\
& =\chainmultadd(u_s, \widetilde{\widehat{f}}((1, i_{t'}, 0) \cdot \widetilde{y_s}), \widetilde{\widehat{f}}((1, i_{t'}, 0) \cdot \widetilde{0}_B), \widetilde{\widehat{f}}(\widetilde{y_s})) \\
& =\widetilde{\widehat{f}}(\chainmultadd(u_s, (1, i_{t'}, 0) \cdot \widetilde{y_s}, (1, i_{t'}, 0) \cdot \widetilde{0}_B, \widetilde{y_s})) \\
& =\widetilde{\widehat{f}}((1, u_si_{t'}, 0) \cdot \chainmultadd(u_s, \widetilde{y_s}, \widetilde{0}_B, \widetilde{y_s})) \\
& =\widetilde{\widehat{f}}((1, u_si_{t'}, 0) \cdot \widetilde{y_s}) \\
& =\widetilde{x_s + u_st}_{\text{right}}.
\end{align*}
\end{proof}

Given the input $\left( \theta_i ^{\Theta_\cL} (0_A) \right)_{i \in K_1(\cL)}$, $\left( \theta_i ^{\Theta_\cL} (t) \right)_{i \in K_1(\cL)}$ and $\left( \theta_i ^{\Theta_\cL} (x) \right)_{i \in K_1(\cL)}$ of the algorithm, when fixing an affine lift $\widetilde{0}_A$ of $0_A$, we can compute an excellent lift $\widetilde{t}$ of $t$ for $\widetilde{0}_A$. This lift will differ from~$\widetilde{t}_{\text{right}}$ by an $\ell$-th root of unity $\zeta_t$, i.e. $\widetilde{t} = \zeta_t \cdot \widetilde{t}_{\text{right}}$. 

\begin{lem}\label{lem:suitlifts}
Let $\widetilde{t}$ (not necessarily excellent) and $\widetilde{x}$ be fixed affine lifts of $t$ and $x$ respectively. Let~$\widetilde{x+t}$ be an affine lift of $x+t$. Then, $\widetilde{x+t}$ is suitable for $\widetilde{t}$ and $\widetilde{x}$ if and only if $\widetilde{x+t}$ is suitable for $\zeta \cdot \widetilde{t}$ and $\widetilde{x}$ for any $\ell$-th root of unity $\zeta$.
\end{lem} 

\begin{proof}
This is a direct consequence of \cite[Lem.3.10]{lubicz-robert:isogenies}, saying that
$$\chainmultadd(\ell, \widetilde{x+t}, \zeta \cdot \widetilde{t}, \widetilde{x}) = \zeta^{\ell(\ell-1)}\cdot \chainmultadd(\ell,\widetilde{x+t}, \widetilde{t}, \widetilde{x}).$$
\end{proof}

Fix arbitrary affine lifts $\widetilde{x_1}, \dots ,\widetilde{x_r}$ of $x_1, \dots , x_r$ (for example by computing $\alpha_1 x, \dots , \alpha_r x$ in Mumford coordinates and then convert to theta coordinates). By Lemma \ref{lem:suitlifts}, computing suitable lifts of $x_1 + t
, \dots , x_r + t$ for $\widetilde{t}_{\text{right}}$ and $\widetilde{x_1}, \dots, \widetilde{x_r}$ is equivalent to computing suitable lifts of $x_1 + t
, \dots , x_r + t$ for $\widetilde{t} \, (=\zeta_t \cdot \widetilde{t}_{\text{right}})$ and $\widetilde{x_1}, \dots, \widetilde{x_r}$. We cannot perform the former computation since we do not know~$\widetilde{t}_{\text{right}}$, but we can perform the latter computation. This means that we can compute suitable lifts $\widetilde{x_1 + t}, \dots , \widetilde{x_r + t}$ of $x_1 + t, \dots , x_r + t$ for $\widetilde{t}_{\text{right}}$ and $\widetilde{x_1}, \dots , \widetilde{x_r}$ respectively, that differ from the \textbf{right} lifts of $x_1 + t, \dots , x_r + t$ by $\ell$-th roots of unity $\zeta_{x_1 + t}, \dots , \zeta_{x_r + t}$. Moreover, if we compute lifts of $x_s + u_st$ as
$$\widetilde{x_s + u_st} := \chainmultadd(u_s, \widetilde{x_s + t}, \widetilde{t}, \widetilde{x_s}),$$
then they differ from the \textbf{right} lifts of $x_s + u_st$ by a factor $\zeta_{x_s + t}^{u_s}$, meaning that
$$
\widetilde{x_s + u_st} = \zeta_{x_s + t}^{u_s} \cdot \widetilde{x_s + u_st}_{\text{right}}.
$$

\paragraph{Choice of lifts of $x_s + u_st$.}\label{par:choiceliftspt}
We use a different approach based on the Chinese Remainder Theorem and the fact that the order $N$ of $x$ is coprime to $\ell$. Suppose that we have fixed a lift~$\widetilde{0}_A$ of $0_A$ and that we have computed a lift $\widetilde{x}$ of $x$ that satisfies $\chainmult(N, \widetilde{x}) = \widetilde{0}_A$. This, via~$\widetilde{\widehat{f}}$, determines lifts $\widetilde{0}_B$ and $\widetilde{y}$ of $0_B$ and $y$,  respectively, satisfying $\chainmult(N,\widetilde{y}) = \widetilde{0}_B$. The lift~$\widetilde{0}_B$ then determines \textbf{right} lifts $\widetilde{t}_{\text{right}}$ and $\widetilde{x+t}_{\text{right}}$. We first compute an excellent lift $\widetilde{t}$ of $t$. We have~$\widetilde{t} = \zeta_t \cdot \widetilde{t}_{\text{right}}$ for an~$\ell$-th root of unity $\zeta_t$. We then compute a suitable lift $\widetilde{x+t}$ of $x+t$ for $\widetilde{t}$ (hence for $\widetilde{t}_{\text{right}}$) and $\widetilde{x}$. We have $\widetilde{x+t} = \zeta_{x+t} \cdot \widetilde{x+t}_{\text{right}}$ for $\zeta_{x+t}$ an $\ell$-th root of unity. Let $\widetilde{y + t'}$ be the lift of $y + t'$ induced by $\widetilde{\widehat{f}}$ and $\widetilde{x+t}$. It is not hard to see that 
\begin{equation}
\widetilde{y + t'} = \zeta_{x+t} \cdot (1, i_{t'}, 0) \cdot \widetilde{y}.
\end{equation}

Suppose $e$ and $f$ satisfy $eN + f\ell = 1$. Consider the isomorphism
$$c_{[\cdot , \cdot]}\colon \Z/N\Z \times \Z/ \ell \Z \to \Z/ N\ell \Z,\, (a,u)\mapsto c_{[a,u]} = ueN + af\ell.$$
For $s = 1, \dots, r$ let $a_s \bmod N$ be the integer such that $x_s = \alpha_s x = a_s x$.

Define lifts of $x_1, \dots , x_r$ as
\begin{equation}\label{def:x_s}
\widetilde{x_s} := \chainmult(c_{[a_s,0]}, \widetilde{x+t}) \text{ for } s=1,\dots , r.
\end{equation}
The lifts $\widetilde{x_1} , \dots , \widetilde{x_r}$ induce lifts $\widetilde{y_1} , \dots , \widetilde{y_r}$ of $y_1, \dots , y_r$ and using the compatibility of the affine isogeny~$\widetilde{\widehat{f}}$ with $\chainmult$, it is not hard to see that
\begin{equation}
\widetilde{y_s} = \chainmult(c_{[a_s, 0]}, \widetilde{y + t'}) \text{ for } s=1, \dots , r.
\end{equation}
The lifts $\widetilde{x_1} , \dots , \widetilde{x_r}$ also induce \textbf{right} lifts
$$\{ \widetilde{x_s + u_st}_{\text{right}} : s=1,\dots, r, \, u_s = 1, \dots , \ell-1 \},$$
which are given by
\begin{equation}\label{eq:rightlift_x_s+u_st}
\widetilde{x_s + u_st}_{\text{right}} = \widetilde{\widehat{f}}((1, u_si_{t'}, 0)\cdot \widetilde{y_s}) = \widetilde{\widehat{f}}((1, u_si_{t'}, 0)\cdot \chainmult(c_{[a_s, 0]}, \widetilde{y + t'})).
\end{equation}
Again, there is no chance we can determine the \textbf{right} lift of $x_s + u_st$. But we can prove the following.

\begin{prop}\label{prop:suitlifts}
Let $\widetilde{0}_A$ and $\widetilde{x}$ be fixed affine lifts of $0_A$ and $x$ respectively. Let $\widetilde{t}$ be a fixed excellent lift of $t$, that differs from the \textbf{right} lift of $t$ by an~$\ell$-th root of unity~$\zeta_t$. Let $\widetilde{x+t}$ be a fixed suitable lift of $x+t$ for $\widetilde{t}$ and $\widetilde{x}$, and suppose $\widetilde{x+t}$ differs from the \textbf{right} lift of $x+t$ by an~$\ell$-th root of unity~$\zeta_{x+t}$. Let $\widetilde{x_1}, \dots , \widetilde{x_r}$ be lifts of $x_1, \dots , x_r$ defined as in \eqref{def:x_s}. For $s=1, \dots, r$ and $u_s = 1, \dots , \ell-1$, define a lift of $x_s + u_st$ as
\begin{equation}
\widetilde{x_s + u_s t} := \chainmult(c_{[a_s, u_s]}, \widetilde{x+t}).
\end{equation}
Then, we have that
$$\widetilde{x_s + u_st} = \zeta_{x+t}^{u_s^2} \cdot \widetilde{x_s + u_st}_{\text{right}},$$
where the \textbf{right} lift of $x_s + u_st$ is as in \eqref{eq:rightlift_x_s+u_st}.
\end{prop}

\begin{proof}
Observe that 
\begin{align*}
\widetilde{y_s} & = \chainmult(c_{[a_s , 0]}, \widetilde{y+t'}) 
 = \chainmult(c_{[a_s,0]}, \zeta_{x+t} \cdot (1, i_{t'} , 0) \cdot \widetilde{y}) \\
& = \chainmult(c_{[a_s,0]}, (1, i_{t'}, 0) \cdot \widetilde{y}) \\
& \hspace{5mm} \text{(by Lemma 3.10 of \cite{lubicz-robert:isogenies} and the fact that $c_{[a_s, 0]}$ is congruent to 0 mod $\ell$)} \\
& = \chainmult(c_{[a_s, 0]}, \widetilde{y}) \text{ by Lemma \ref{lem:compatibility}}
\end{align*}

Now,
\begin{align*}
\widetilde{x_s + u_st} & = \chainmult(c_{[a_s,u_s]}, \widetilde{x+t}) \\
& =\chainmult(c_{[a_s,u_s]}, \widetilde{\widehat{f}}(\widetilde{y+t'})) \\
& =\widetilde{\widehat{f}}(\chainmult(c_{[a_s,u_s]}, \widetilde{y+t'})) \\
& =\widetilde{\widehat{f}}(\chainmult(c_{[a_s,u_s]}, \zeta_{x+t}\cdot (1,i_{t'}, 0) \cdot \widetilde{y})) \\
& =\widetilde{\widehat{f}}(\zeta_{x+t}^{c_{[a_s,u_s]}^2} \cdot \chainmult(c_{[a_s,u_s]}, (1,i_{t'}, 0) \cdot \widetilde{y})) \\
& =\widetilde{\widehat{f}}(\zeta_{x+t}^{u_s^2} \cdot (1, u_si_{t'}, 0) \cdot \chainmult(c_{[a_s,u_s]}, \widetilde{y})) \\
& \hspace{5.5mm} \text{(since $c_{[a_s,u_s]}$ congruent to $u_s \bmod \ell$ and Lemma \ref{lem:compatibility})} \\
& =\zeta_{x+t}^{u_s^2} \cdot \widetilde{\widehat{f}}( (1, u_si_{t'}, 0) \cdot \chainmult(c_{[a_s,0]}, \widetilde{y}))  \\
& \hspace{5.5mm} \text{(since $\chainmult(N,\widetilde{y}) = \widetilde{0}_B$)} \\
& =\zeta_{x+t} ^{u_s^2}\cdot \widetilde{\widehat{f}}( (1, u_si_{t'}, 0) \cdot \widetilde{y_s}) \\
& \hspace{5.5mm} \text{(by the above)} \\
& =\zeta_{x+t}^{u_s^2} \cdot \widetilde{x_s + u_st}_{\text{right}}.
\end{align*}

\end{proof}

\paragraph{Independence of the choice of a suitable lift of $x+t$.}

We prove the case $r=4$. The case $r=2$ is easier and can be proven in a similar way. 

Fix $\mathbf{k} \in K_1(\cM^{\star 4})$ and consider the sum
\begin{equation}\label{eq:sum2}
\sum_{\substack{\mathbf{t}'' \in K_1((\cM^\beta)^{\star 4})[\beta] \\ F(\mathbf{t}'') = 0}}
\prod_{s=1}^4 \theta^{\Theta_{\cM^\beta}}_{j_s+t''_s} (y_s),
\end{equation}
where $\mathbf{j} \in K_1((\cM^{\beta})^{\star 4})[n]$ is the unique index such that $F(\mathbf{j}) = \mathbf{k}$. Let $t' \in K_1(\cM^\beta)[\beta]$ be the unique preimage of $t$ under $\widehat{f}$.

Fix a lift $\widetilde{0}_A$ of $0_A$ and fix a lift $\widetilde{x}$ of $x$ that satisfies $\chainmult(N, \widetilde{x}) = \widetilde{0}_A$, where $N$ is the order of $x$ which is coprime to $\ell$. Compute an excellent lift $\widetilde{t}$ of $t$, that differs form $\widetilde{t}_{\text{right}}$ by an $\ell$-th root of unity $\zeta_t$, and compute a suitable lift $\widetilde{x+t}$ of $x+t$ for $\widetilde{t}$ (hence for $\widetilde{t}_{\text{right}}$) and $\widetilde{x}$, that differs from~$\widetilde{x+t}_{\text{right}}$ by an $\ell$-th root of unity $\zeta_{x+t}$. For $s=1, \dots , 4 $ and $u_s = 0, \dots , \ell-1$, define a lift of~$x_s + u_s t$ as
\begin{equation}\label{eq:ourlifts}
\widetilde{x_s + u_st} := \chainmult(c_{[a_s, u_s]}, \widetilde{x+t}).
\end{equation}
Proposition \ref{prop:suitlifts} tells us that for $s=1, \dots , 4$ and $u_s = 1, \dots , \ell-1$, the lift $\widetilde{x_s + u_st}$ differs from the \textbf{right} lift of $x_s + u_st$ by $\zeta_{x+t}^{u_s^2}$, i.e., 
$$\widetilde{x_s + u_st} = \zeta_{x+t}^{u_s^2} \cdot \widetilde{x_s + u_st}_{\text{right}}.$$
Here, \textbf{right} lift of $x_s + u_st$ means with respect to the lift $\widetilde{x_s} = \chainmult(c_{[a_s, 0]}, \widetilde{x+t})$ of $x_s$.
We want to evaluate \eqref{eq:sum2} by making the substitution
$$\theta^{\Theta_{\cM^\beta}}_{j_s+t''_s} (y_s) = \theta^{\Theta_\cL}_{\widehat{f}(j_s)}(\widetilde{x_s + u_st}),$$
where $u_s$ is such that $t_s'' = u_s t'$. If we knew the \textbf{right} lifts
$$\{\widetilde{x_s + u_st}_{\text{right}} : s=1, \dots , 4 \text{ and } u_s = 0, \dots, \ell-1\},$$
where $\widetilde{x_s + 0t}_{\text{right}}$ is just $\widetilde{x_s}$, then, by making the substitutions 
$$\theta^{\Theta_{\cM^\beta}}_{j_s+t''_s} (y_s) = \theta^{\Theta_\cL}_{\widehat{f}(j_s)}(\widetilde{x_s + u_st}_{\text{right}}),$$
we could compute \eqref{eq:sum2} correctly.

Proceeding as in Section~\ref{par:indep_of_excellent_lift}, we compute
$$\sum_{\substack{\mathbf{t}'' \in K_1((\cM^\beta)^{\star 4})[\beta] \\ F(\mathbf{t}'') = 0}}
\prod_{s=1}^4 \theta^{\Theta_{\cM^\beta}}_{j_s+t''_s} (y_s)$$
by computing
\begin{equation}\label{eq:y_isogF}
\displaystyle\sum_{t_1,t_2 \in G}\theta^{\Theta_{\cL}}_{i_1} (\widetilde{x_1+ \alpha_1 t_1 -\alpha_2 t_2})\theta^{\Theta_{\cL}}_{i_2} (\widetilde{x_2 + \alpha_2 t_1 +\alpha_1 t_2}) 
\theta^{\Theta_{\cL}}_{i_3} (\widetilde{x_3 + \alpha_3 t_1 -\alpha_4 t_2})  \theta^{\Theta_{\cL}}_{i_4} (\widetilde{x_4 + \alpha_4 t_1 +\alpha_3 t_2}),
\end {equation}
where $i_s=\widehat f(j_s)\in K_1(\cL)$. For $s=1,\dots, 4$, let $a_{s, t} \bmod \ell$ be the integer such that $\alpha_s t = a_{s,t}t$. Writing $t_1 = u_{1,t}t$ and $t_2 = u_{2,t}t$, with $0 \le u_{1,t} , u_{2,t} \le \ell-1$, we define integers 
\begin{itemize}
\item $u_1 = a_{1,t} u_{1,t}  - a_{2,t} u_{2,t}  \bmod \ell$, $u_2 = a_{2,t}u_{1,t} + a_{1,t}u_{2,t} \bmod \ell$,
\item $u_3 = a_{3,t}u_{1,t} - a_{4,t}u_{2,t} \bmod \ell$, $u_4 = a_{4,t}u_{1,t} - a_{3,t}u_{2,t} \bmod \ell$.
\end{itemize}
Then, we are reduced to computing 
\begin{equation}\label{eq:y_isogF2}
\displaystyle\sum_{0 \le u_{1,t}, u_{2,t} \le \ell-1}\theta^{\Theta_{\cL}}_{i_1} (\widetilde{x_1 + u_1t})\theta^{\Theta_{\cL}}_{i_2} (\widetilde{x_2 + u_2t}) 
\theta^{\Theta_{\cL}}_{i_3} (\widetilde{x_3 + u_3t})  \theta^{\Theta_{\cL}}_{i_4} (\widetilde{x_4 + u_4t}).
\end {equation}

\begin{lem}
Fix $u_{1,t}, u_{2,t} \in \{0, \dots , \ell-1 \}$. Then we have equality between the terms
$$\theta^{\Theta_{\cL}}_{i_1} (\widetilde{x_1 + u_1t})\theta^{\Theta_{\cL}}_{i_2} (\widetilde{x_2 + u_2t}) 
\theta^{\Theta_{\cL}}_{i_3} (\widetilde{x_3 + u_3t})  \theta^{\Theta_{\cL}}_{i_4} (\widetilde{x_4 + u_4t}) \text{ and }$$
$$\theta^{\Theta_{\cL}}_{i_1} (\widetilde{x_1 + u_1t}_{\text{right}})\theta^{\Theta_{\cL}}_{i_2} (\widetilde{x_2 + u_2t}_{\text{right}}) 
\theta^{\Theta_{\cL}}_{i_3} (\widetilde{x_3 + u_3t}_{\text{right}})  \theta^{\Theta_{\cL}}_{i_4} (\widetilde{x_4 + u_4t}_{\text{right}}).$$
That is to say, evaluating \eqref{eq:sum2} by substituting the lifts $\{ \widetilde{x_s + u_st} : s = 1, \dots , 4 \text{ and } u_s = 0, \dots , \ell-1\}$ as defined in \eqref{eq:ourlifts}, we correctly compute an affine lift of $(y, 0,\dots, 0) \in B^4(k)$ with respect to $\widetilde{\Theta}_{\cM^{\star 4}}$.
\end{lem}

\begin{proof}
By Proposition \ref{prop:suitlifts}, the term 
$$\theta^{\Theta_{\cL}}_{i_1} (\widetilde{x_1 + u_1t})\theta^{\Theta_{\cL}}_{i_2} (\widetilde{x_2 + u_2t}) 
\theta^{\Theta_{\cL}}_{i_3} (\widetilde{x_3 + u_3t})  \theta^{\Theta_{\cL}}_{i_4} (\widetilde{x_4 + u_4t})$$
differs from 
$$\theta^{\Theta_{\cL}}_{i_1} (\widetilde{x_1 + u_1t}_{\text{right}})\theta^{\Theta_{\cL}}_{i_2} (\widetilde{x_2 + u_2t}_{\text{right}}) 
\theta^{\Theta_{\cL}}_{i_3} (\widetilde{x_3 + u_3t}_{\text{right}})  \theta^{\Theta_{\cL}}_{i_4} (\widetilde{x_4 + u_4t}_{\text{right}})$$
by $\zeta_{x+t}^{u_1^2 + \cdots + u_4^2}$. But
$$u_1^2 + \cdots + u_4^2 = (a_{1,t}^2 + \cdots + a_{4,t}^2)(u_{1,t}^2 + u_{2,t}^2)$$
which is a multiple of $\ell$, since $a_{1,t}^2 + \cdots + a_{4,t}^2$ is given by the action of $\beta$ on $t$.
\end{proof}

\subsection{Theta coordinates for $(B, \cM, \Theta_\cM)$}
In Section \ref{sebseq:thetacoord_B^r} we computed the theta point $\left(\theta^{\widetilde{\Theta}_{\cM^{\star r}}}_{\mathbf{k}}(f(x),0,\dots,0)\right)_{\mathbf{k} \in K_1(\cM^{\star r})}$ for $x\in A(k)$. Analogous to Section \ref{subsec:metaplect} this does not allow us to recover the theta coordinates of $f(x)$ for $(B, \cM)$. The same metaplectic automorphism from Section \ref{par:unfolding} turns $\widetilde{\Theta}_{\cM^{\star r}}$ into a product theta structure and the transformation law for the theta coordinates from Proposition \ref{prop:transform} applies.
Let $S = \left( \begin{array}{cc} A & B \\C & D \end{array}\right)\in \Sp_{2gr}(\Z)$ be as in Proposition \ref{prop:transform} and let $\mathbf{e}' = \frac{1}{2} \diag ( A^t C)$ and $ \mathbf{e''} = \frac{1}{2} \diag( D^t B)$.
Then there exists a constant $\lambda \in \bar{k}^\times$ such that for all $\mathbf{b'} = (b_1',\dots, b_r')\in Rpr  \left ( \frac{1}{2}\Z^g/\Z^g \right )^r$ we have

\begin{eqnarray}
& \theta^{\Theta_{\cM}}_{b_1'}(f(x)) \cdot \theta^{\Theta_{\cM^{\star (r-1)}}}_{(b_2', \dots, b_r')}(0_{B^{r-1}}) = \theta^{\Theta_{\cM^{\star r}}}_{\mathbf{b'}}((f(x), 0, \dots , 0)) \\ \nonumber
&= \frac{\lambda}{2^{2r}}\sum_{\mathbf{a} \in \left ( \frac{1}{2}\Z^g / \Z^g \right )^r} \xi^2_{\mathbf{a},\mathbf{b}}\sum_{\mathbf{i}\in \left ( \frac{1}{2}\Z^g/\Z^g \right)^r}
\exp(4\pi i\mathbf{a}^t\mathbf{i})\theta^{\widetilde \Theta_{\cM^{\star r}}}_{\mathbf{b}+\mathbf{i}}((f(x), 0, \dots, 0))\theta^{\widetilde \Theta_{\cM^{\star r}}}_{\mathbf{i}}(0_{B^r}),
\end{eqnarray}

where $\mathbf{a},\mathbf{b},\mathbf{a'},\mathbf{b'}\in Rpr  \left ( \frac{1}{2}\Z^g/\Z^g \right )^r$ are related by $\scalefont{0.85}{\begin{pmatrix}
\mathbf{a'}\\
\mathbf{b'}
\end{pmatrix}=\begin{pmatrix} D & -C\\ -B & A \end{pmatrix}\cdot 
\begin{pmatrix}
\mathbf{a}-\mathbf{e'}\\
\mathbf{b}-
\mathbf{e''}
\end{pmatrix}}$, and 
$$
\xi_{\mathbf{a}, \mathbf{b}}=\exp \left(-\pi i (\mathbf{a}^t A B^t \mathbf{a}+\mathbf{b}^t C D^t \mathbf{b})-2\pi i(A^t \mathbf{a}+C^t \mathbf{b}+\mathbf{e'})^t\mathbf{e''}-2\pi i\mathbf{a}^t B C^t\mathbf{b}\right).
$$ 

We finally compute the projective point 
$$(\theta^{\Theta_\cM}_k ( f(x)))_{k \in K_1(\cM)}.$$

%
%
\section{Complexity Analysis}\label{sec:complexity}
The algorithms from Sections~\ref{sec:computing_theta_null_point} and \ref{sec:isogpts} depend on the following parameters: 
\begin{itemize}
\item the size $q$ of the finite field $k = \F_q$, 

\item the order $\ell$ of the kernel of the isogeny,  

\item the level $n$ of the theta functions that we use in the computation,  

\item the dimension $g$ of the abelian variety,  

\item the parameter $r$ (typically, $r$ is either 2 or 4), 

\item the order $N$ of the point $x \in A(\F_q)$ (for the algorithm of Section~\ref{sec:isogpts}).  
\end{itemize}
For the moment, we assume that the matrix $F$ is precomputed and that the (affine) theta null point $(\theta_i^{\Theta_{\cL}} (\widetilde{0}_A))_{i \in K_1(\cL)}$ as well as the (affine) theta points $(\theta^{\Theta_{\cL}}_{i}(\widetilde t))_{i \in K_1(\cL)}$ and $(\theta^{\Theta_{\cL}}_{i}(\widetilde x))_{i \in K_1(\cL)}$ are provided as input to the algorithm. We fix the following notations:
\begin{itemize}
\item $k_0$ denotes the field of definition of the affine theta coordinates of $0_A$,
\item $k_t$ denotes the field of definition of the affine theta coordinates of $t$,
\item $k_x$ denotes the field of definition of the affine theta coordinates of $x$,
\item $k_{x+t}$ denotes the field of definition of the affine theta coordinates of $x+t$.
\end{itemize}
It will follow from Section \ref{par:normal_addition} that $k_{x+t}$ is the composite field of $k_x$ and $k_t$.

In the case where $A$ is the Jacobian variety of a hyperelliptic curve $H$ defined over $\F_q$, we can explicitly determine the fields $k_0, k_t$ and $k_x$. Suppose that the Weierstrass points of $H$ have coordinates in $\F_{q^d}$. Then $k_0$ equals $\F_{q^{2d}}$. In general, if $z \in A(\F_{q^{d'}})$ is any point on the Jacobian of~$H$, then the theta coordinates of $z$ will be defined over the composite field of $\F_{q^{2d}}$ and $\F_{q^{d'}}$. Since~$x$ is $\F_q$-rational we have that $k_x = k_0$, and if $d'$ is the smallest integer such that $A[\ell] \subset A(\F_{q^{d'}})$, then $k_t$ is the composite field of $\F_{q^{2d}}$ and $\F_{q^{d'}}$.

\subsection{Computing the theta null point for $B$}
We first analyse the complexity of the algorithm from Section~\ref{sec:computing_theta_null_point}.  

\paragraph{Compute an excellent affine lift for $t \in G$.}\label{par:complexity-excellent}
To compute an excellent affine lift of $t$, we follow the procedure in Section~\ref{par:excellent}, that is, we take any affine lift $\widetilde{t}$ and compute $\chainmult(m+1, \widetilde{t}, \widetilde{0}_A)$ as well as 
$-\chainmult(m, \widetilde{t}, \widetilde{0}_A)$ and then solve an equation $\lambda_t^\ell = c$ to determine $\lambda_t$.

As discussed in \cite[p.1494]{lubicz-robert:isogenies}, a multiplication chain requires $O(\log m)$ chain additions (in the worst case, $3 \log m$). The complexity of each chain addition has been analyzed in 
\cite[\S 3.2]{lubicz-robert:isogenies} and has complexity 
$$
4 (4\ell n)^g \textbf{M}(k_t) + (4\ell n)^g \mathbf{A}(k_t) + (\ell n)^g \mathbf{D}(k_t)
$$ 
where, $\textbf{M}(k_t)$, $\textbf{A}(k_t)$ and $\textbf{D}(k_t)$ are the costs of multiplication, addition and division in the field $k_t$ respectively. 

Once we have obtained $\widetilde{t}_e$, we need to compute the excellent lifts 
$\widetilde t_e, \widetilde {2t}_e, \dots, \widetilde{(m+1) t}_e$. These then automatically yield lifts for $\widetilde{(m + 1)t_e}, \dots, \widetilde{ (\ell-1) t_e }$ (so we do not need to compute those) and hence, for all the points in $G$. In the case where $A$ is a hyperelliptic Jacobian, we have $A[\ell] \subset A(k_t)$, and since the $\ell$-th roots of unity $\mathbf{\mu}_\ell$ form a subgroup of  $k_t^\times$, the excellent lift $\widetilde{t}_e$ has affine coordinates in $k_t$.

\paragraph{Evaluating the right-hand side of equation \eqref{eq:0_r=4}.} 
For a given $\mathbf{k} \in K_1(\cM^{\star r})$, computing the right-hand side of \eqref{eq:0_r=4} requires $\ell^{r/2}$ times $(r-1)$ multiplications and one addition in the field $k_t$ as the theta coordinates of the suitable lifts of the points in the kernel have been computed in the previous step. There are $n^{gr}$ indices $\mathbf k$ for which we need to do this computation, thus leading to $n^{gr} \ell^{r/2} (r-1)$ multiplications and $n^{gr} \ell^{r/2}$ additions in 
$k_t$ with a total cost of 
$$
n^{gr} \ell^{r/2} ( (r-1) \textbf{M}(k_t) + \textbf{A}(k_t) ).
$$ 

\paragraph{Computing the symplectic transformation $\overline{S}$.}
The complexity depends only on $n$, $g$ and $r$. In practice, this can be speeded up if one finds a faster method for a symplectic transformation of a $2gr$-by-$2gr$ matrix with entries in $\Z / 2n\Z$ into a block-diagonal form (we only need to have a $2g$-by-$2g$ block that will then correspond to the single copy $(B, \cM, \Theta_{\cM})$ in the product 
$(B\times B^{r-1}, \cM \star \cM^{\star (r-1)}, \Theta_{\cM \star \cM^{\star (r-1)}})$). The brute-force method presented in Section~\ref{par:unfolding} requires testing $(2n)^{4g^2}$ matrices. 

\paragraph{Applying the transformation formula.}
The main cost is given by the number of multiplications and additions needed to compute the right-hand side of \eqref{eq:isS_gen}. 
For each element in $K_1(\cM^{\star r})$, one needs $\# K_1(\cM^{\star r}) \cdot \#K_1(\cM^{\star r}) = n^{2gr}$ multiplications and $n^{2gr}$ additions in the field $k_t$. Thus, the total cost of the coordinate transformation  is $n^{3gr}( \mathbf{M}(k_t) + \mathbf{A}(k_t))$.

\subsection{Computing the theta coordinates for $f(x)$}
For a point $x \in A(\F_q)$, suppose that we want to compute the coordinates of $f(x)$ with respect to the theta structure $\Theta_{\cM}$ on $(B, \cM)$, assuming that we are given the theta coordinates of $x$ for the theta structure $\Theta_{\cL}$ on $(A, \cL)$.

\paragraph{Computing a lift of $x + t$ using normal additions.}\label{par:normal_addition} Recall that the type of $\cL$ is $\delta = (n, \dots, n)$. Given affine lifts $(\theta^{\Theta_{\cL}}_{i}(\widetilde t))_{i \in K_1(\cL)}$ and $(\theta^{\Theta_{\cL}}_{i}(\widetilde x))_{i \in K_1(\cL)}$ of $t$ and $x$ respectively, we explain how to compute (affine) theta coordinates for $x + t$. 
We do that by using normal additions. That is, we compute for each $i \in K_1(\cL)$ the product 
$\theta_i^{\Theta_{\cL}}(\widetilde{x + t}) \theta_0^{\Theta_{\cL}}(\widetilde{x - t})$ as in \cite[p.81]{drobert:thesis}: if we write 
$i = 2 u$ where $u \in \Z(2\delta)$ then 
\begin{equation}\label{eq:normaladd}
\theta_{i}^{\Theta_{\cL}}(\widetilde{x + t}) \theta_{0}^{\Theta_{\cL}}(\widetilde{x - t}) = \theta_{u + u}^{\Theta_{\cL}}(\widetilde{x + t}) \theta_{u - u}^{\Theta_{\cL}}(\widetilde{x - t}) = \frac{1}{2^g} 
\sum_{\chi \in \widehat{\Z}(\delta)} U_{\chi, u}^{\cL^2}(\widetilde{x}) U_{\chi, u}^{\cL^2}(\widetilde{t}), 
\end{equation}
where 
\begin{equation}\label{eq:fourier}
U_{\chi, u}^{\cL^2}(\widetilde{x}) U_{\chi, u}^{\cL^2}(\widetilde{t}) = \frac{1}{U^{\cL^2}_{\chi, u}(\widetilde{0}_A)^2}
\left ( \sum_{s \in \Z(\delta)} 
\chi(s) \theta_{i + s}^{\Theta_{\cL}}(\widetilde{x}) \theta_{s}^{\Theta_{\cL}}(\widetilde{x}) \right ) \left ( \sum_{s \in \Z(\delta)} \chi(s) \theta_{i + s}^{\Theta_{\cL}}(\widetilde{t}) \theta_{s}^{\Theta_{\cL}}(\widetilde{t}) \right ).  
\end{equation}
Here, $\displaystyle U^{\cL^2}_{\chi, u}(\widetilde{0}_A)^2 = \sum_{s \in \Z(\delta)} \chi(s) \theta^{\Theta_{\cL}}_{i + s}(\widetilde{0}_A) \theta^{\Theta_{\cL}}_{s}(\widetilde{0}_A)$. 

\begin{rem}
As explained in \cite[p.81]{drobert:thesis}, one has to be a bit careful with the required non-vanishing of the denominator on the right-hand side of \eqref{eq:fourier}. Yet, according to \cite[Thm.4.4.4]{drobert:thesis}, there will always be $u' \in \Z(2\delta)$ with $2 u' = i$ for which the denominator will not vanish.  
\end{rem}

\noindent Assuming that $\theta_{0}^{\Theta_{\cL}}(\widetilde{x - t}) \ne 0$, we obtain projective theta coordinates for 
$x + t$.
For fixed $i \in \Z(\delta)$ and fixed $\chi \in \widehat{\Z}(\delta)$, the computation of  $U_{\chi, u}^{\cL^2}(\widetilde{x}) U_{\chi, u}^{\cL^2}(\widetilde{t})$ has a cost of $2n^g(\mathbf{M}(k_{x+t}) + \mathbf{A}(k_{x+t})) + \mathbf{M}(k_{x+t}) + \mathbf{D}(k_{x+t})$, and hence the computation of  $\theta_{i}^{\Theta_{\cL}}(\widetilde{x + t})$ has a cost of $n^g(2n^g(\mathbf{M}(k_{x+t}) + \mathbf{A}(k_{x+t})) + \mathbf{M}(k_{x+t}) + \mathbf{D}(k_{x+t}) + \mathbf{A}(k_{x+t}))$. Finally, we compute the projective theta coordinates of~$x+t$ at a cost of 
$$n^{2g}(2n^g(\mathbf{M}(k_{x+t}) + \mathbf{A}(k_{x+t})) + \mathbf{M}(k_{x+t}) + \mathbf{D}(k_{x+t}) + \mathbf{A}(k_{x+t})).$$

\paragraph{Computing affine lifts from Section~\ref{par:choiceliftspt}.} Suppose we are given an excellent lift $(\theta^{\Theta_{\cL}}_{i}(\widetilde t))_{i \in K_1(\cL)}$ of $t$ from the computation of the theta null point of $B$. Let $\widetilde{x + t}$ be an arbitrary affine lift of $x+t$ computed using the normal additions in the previous section. Compute a lift $\widetilde{x}$ of $x$ satisfying $\chainmult(N, \widetilde{x}) = \widetilde{0}_A$. This requires $\log N$ chain additions in the field $k_x$. As before, the cost is 
$$
3 \log N \left ( 4(4\ell n)^g \textbf{M}(k_x) + (4\ell n)^g \mathbf{A}(k_x) + (\ell n)^g \mathbf{D}(k_x) \right). 
$$
Next, compute a suitable lift of $x + t$ for $\widetilde{t}$ and $\widetilde{x}$ (which we need to compute out of the arbitrary affine lift of $x + t$). This requires a multiplication chain involving $\log \ell$ addition chains each with complexity $4 (4\ell n)^g \textbf{M}(k_{x+t}) + (4\ell n)^g \mathbf{A}(k_{x+t}) + (\ell n)^g \mathbf{D}(k_{x+t})$, hence a total number of 
$$
3 \log \ell \left ( 4(4\ell n)^g \textbf{M}(k_{x+t}) + (4\ell n)^g \mathbf{A}(k_{x+t}) + (\ell n)^g \mathbf{D}(k_{x+t}) \right)
$$
operations in the field $k_{x+t}$. Finally, the cost of computing each of the affine lifts $\widetilde{x_s + u_s t}$ defined in Proposition~\ref{prop:suitlifts} is at most 
$$
3 \log (N\ell) \left (4 (4\ell n)^g \textbf{M}(k_{x+t}) + (4\ell n)^g \mathbf{A}(k_{x+t}) + (\ell n)^g \mathbf{D}(k_{x+t}) \right),   
$$
since the scalar $c_{[a_s, u_s]}$ is at most $N \ell$.

\paragraph{Computing the right-hand side of \eqref{eq:y_isogF2}.}
To evaluate the sum in \eqref{eq:y_isogF2}, we need to loop over all indices $\mathbf{k} \in K_1(\cM^{\star r})$ (a total of $n^{gr}$ indices) and for 
each index, evaluate each of the $\ell^{r / 2}$ summands. Each summand then requires $(r-1)$ multiplications in the field $k_{x+t}$ as well as one addition, a total cost of 
$$
n^{gr} \ell^{r/2} ( (r-1) \mathbf{M}(k_{x+t}) + \mathbf{A}(k_{x+t})) 
$$
operations.
 
\paragraph{Applying the transformation formula.}
As for the computation of the theta null point for $B$, this step requires $n^{3gr}$ multiplications and additions in the field $k_{x+t}$, i.e. has a total cost of 
$$n^{3gr}(\mathbf{M}(k_{x+t}) + \mathbf{A}(k_{x+t}))$$
operations.

%
%
\section{Computational Examples}\label{sec:examples}

We have implemented the algorithm from Section \ref{sec:computing_theta_null_point} in Magma and used that implementation to compute the following example of isogenous abelian surfaces. Consider the curve 
$$H : y^2=x^5 + x^4 + 3x^3+22x^2 + 19x$$
over $\F_{23}$ and its Jacobian $J = \Jac(H)$. Then $J$ is ordinary and simple and the (irreducible) characteristic polynomial of the Frobenius endomorphism $\pi$ is given by $\chi_\pi(z) = z^4 + 14z^2 + 529$. The endomorphism algebra $\End^0(J) = \End(J) \otimes_\Z \Q$ is isomorphic to the quartic CM-field $K = \Q(\pi) = \Q[z] / (\chi_\pi)$, and the totally real subfield $K_0 \subset K$, consisting of the Rosati-stable elements of $\End^0(J)$, is generated by $\pi + \pi^\dagger$ over $\Q$. We have $\Z[\pi + \pi^\dagger]\subset \End(J)^+ \subset \O_{K_0}$.
The real endomorphism $\beta = -38(\pi + \pi^\dagger) + 215$ is totally positive and of real norm 17 (i.e. of degree $17^2$). Consider the 17-torsion point $t = (x^2 + u_1x + u_0, v_1x+v_0) \in J(\F_{23^{16}})$, where
\begin{scriptsize}
\begin{align*}
&u_1 = 10a^{15} + 9a^{14} + 17a^{13} + 5a^{12} + 14a^{11} + 19a^{10} + 14a^9 + 14a^8 + 5a^7 + 22a^6 + a^5 + 19a^4 + 13a^3 + 2a^2 + 15a + 7,\\
&u_0 = 6a^{15} + 11a^{14} + 17a^{13} + 19a^{12} + 10a^{11} + a^{10} + 21a^9 + 15a^8 + 18a^7 + 21a^6 + 5a^5 + 18a^4 + 4a^3 + 6a^2 + 3a + 19,\\
&v_1 = 19a^{15} + 11a^{14} + 18a^{13} + 3a^{12} + 20a^{11} + 11a^{10} + 8a^9 + a^8 + 19a^7 + 5a^6 + 14a^5 + 3a^4 + 4a^3 + 10a^2 + 22a + 22,\\
&v_0 = a^{15} + 10a^{14} + 11a^{13} + 22a^{12} + 3a^{11} + 14a^{10} + 21a^9 + 5a^8 + 9a^7 + 17a^5 + 20a^4 + 6a^3 + 8a^2 + 13a + 5
\end{align*}
\end{scriptsize}
and $a$ satisfies $a^{16} + 19a^7 + 19a^6 + 16a^5 + 13a^4 + a^3 + 14a^2 + 17a +5 = 0$. The subgroup $G = \langle t \rangle$ is Galois-stable, since $\pi(t) = [6] t$, and moreover we have $G \subset \ker(\beta)$.

The algorithm computes the hyperelliptic curve $H'$ over $\F_{23}$ with affine model
$$H' : y^2 = 5x^6 + 18x^5 + 18x^4 + 8x^3 + 20x,$$
whose Jacobian variety $J' = \Jac(H')$ is isomorphic (as a principally polarized abelian surface) to the quotient $J/G$. 
Hence $J$ and $J'$ are $\beta$-isogenous over $\F_{23}$ and the isogeny is given by
$$J \ra J / G = J'.$$
Indeed, the characteristic polynomial of the Frobenius endomorphism $\pi' \colon J' \ra J'$ equals $\chi_\pi$, but $H$ and $H'$ have Cardona-Quer-Nart-Pujola invariants (c.f. \cite{cardona-nart-pujolas} and \cite{cardona-quer}) given by $[16,12,17]$ and $[18,5,0]$ respectively, and hence the Jacobians $J$ and $J'$ are non isomorphic (as principally polarized abelian surfaces). The computation took 363.2 seconds on a 2.3 GHz Intel Core i7 CPU with 8 GB memory. 

\begin{acknowledgements}
We are grateful to Ernest H. Brooks and Enea Milio for carefully reading the draft of this paper and for suggesting various improvements. We thank Gaetan Bisson, Pierrick Gaudry, Arjen Lenstra, Chloe Martindale, Philippe Michel, Chris Skinner, Ben Smith, Marco Streng, Nike Vatsal, Ben Wesolowski and Alexey Zykin for helpful discussions.  
\end{acknowledgements}

\bibliographystyle{amsalpha}
\bibliography{bib,biblio-math,biblio-crypto}

\def\cprime{$'$}
\providecommand{\bysame}{\leavevmode\hbox to3em{\hrulefill}\thinspace}
\providecommand{\MR}{\relax\ifhmode\unskip\space\fi MR }
\providecommand{\MRhref}[2]{%
  \href{http://www.ams.org/mathscinet-getitem?mr=#1}{#2}
}
\providecommand{\href}[2]{#2}
\begin{thebibliography}{Mum67b}

\bibitem[BFT14]{bruin-flynn-testa}
N.~Bruin, E.~V. Flynn, and D.~Testa, \emph{Descent via {$(3,3)$}-isogeny on
  {J}acobians of genus 2 curves}, Acta Arith. \textbf{165} (2014), no.~3,
  201--223.

\bibitem[Bis11]{bisson}
G.~Bisson, \emph{Endomorphism rings in cryptography}, Ph.D. thesis, Loria,
  Nancy, 2011.

\bibitem[BJW17]{brooks-jetchev-wesolowski}
E.~Brooks, D.~Jetchev, and B.~Wesolowski, \emph{Isogeny graphs of ordinary
  abelian varieties}, available at {\tt http://arxiv.org/pdf/1609.09793v1.pdf}
  (2017).

\bibitem[BL04]{birkenhake-lange}
C.~Birkenhake and H.~Lange, \emph{Complex abelian varieties}, second ed.,
  Grundlehren der Mathematischen Wissenschaften [Fundamental Principles of
  Mathematical Sciences], vol. 302, Springer-Verlag, Berlin, 2004.

\bibitem[CE15]{couveignes-ezome}
J.-M. Couveignes and T.~Ezome, \emph{Computing functions on {J}acobians and
  their quotients}, LMS J. Comput. Math. \textbf{18} (2015), no.~1, 555--577.

\bibitem[CNP05]{cardona-nart-pujolas}
G.~Cardona, E.~Nart, and J.~Pujolas, \emph{Curves of genus two over fields of
  even characteristic}, {Mathematische Zeitschrift, 250:177--201} (2005).

\bibitem[Cos11]{Cosset}
R.~Cosset, \emph{Applications des fonctions theta a la cryptographie sur
  courbes hyperelliptiques}, Ph.D. thesis, Loria, Nancy, 2011.

\bibitem[CQ05]{cardona-quer}
G.~Cardona and J.~Quer, \emph{Field of moduli and field of definition for
  curves of genus 2}, {Lecture Notes Ser. Comput., 13:71--83} (2005).

\bibitem[CR11]{CossetRobert}
R.~Cosset and D.~Robert, \emph{Computing $(\ell,\ell)$-isogenies in polynomial
  time on jacobians of genus $2$ curves}, \url{http://eprint.
  iacr.org/2011/143}, 2011.

\bibitem[dFJP14]{defeo-jao-plut}
L.~de~Feo, D.~Jao, and J.~Plut, \emph{Towards quantum-resistant cryptosystems
  from supersingular elliptic curve isogenies}, J. Mathematical Cryptology
  \textbf{8} (2014), no.~3, 209--247.

\bibitem[DL08]{dolgachev-lehavi}
I.~Dolgachev and D.~Lehavi, \emph{On isogenous principally polarized abelian
  surfaces}, Curves and abelian varieties, Contemp. Math., vol. 465, Amer.
  Math. Soc., Providence, RI, 2008, pp.~51--69.

\bibitem[Fio16]{fiorentino}
A.~Fiorentino, \emph{Weber's formula for the bitangents of a smooth plane
  quartic}, {\tt https://arxiv.org/abs/1612.02049} (2016).

\bibitem[Fly15]{flynn:55}
E.~V. Flynn, \emph{Descent via {$(5,5)$}-isogeny on {J}acobians of genus 2
  curves}, J. Number Theory \textbf{153} (2015), 270--282.

\bibitem[FM02]{fouquet-morain}
M.~Fouquet and F.~Morain, \emph{Isogeny volcanoes and the {SEA} algorithm},
  Algorithmic number theory ({S}ydney, 2002), Lecture Notes in Comput. Sci.,
  vol. 2369, Springer, Berlin, 2002, pp.~276--291.

\bibitem[Igu72]{Igusa1972}
J.-I. Igusa, \emph{Theta functions}, Grundlehren der mathematischen
  Wissenschaf-ten, vol. 194, Springer, 1972.

\bibitem[JW15]{jetchev-wesolowski}
D.~Jetchev and B.~Wesolowski, \emph{On graphs of isogenies of principally
  polarizable abelian surfaces and the discrete logarithm problem}, {\tt
  http://arxiv.org/abs/1506.00522} (2015).

\bibitem[Koh96]{kohel:thesis}
D.~Kohel, \emph{Endomorphism rings of elliptic curves over finite fields}, PhD
  thesis, University of California, Berkeley (1996).

\bibitem[LR12a]{lauter-robert}
K.~Lauter and D.~Robert, \emph{Improved crt algorithm for class polynomials in
  genus 2}, IACR Cryptology ePrint Archive \textbf{2012} (2012), 443.

\bibitem[LR12b]{lubicz-robert:isogenies}
D.~Lubicz and D.~Robert, \emph{Computing isogenies between abelian varieties},
  Compos. Math. \textbf{148} (2012), no.~5, 1483--1515.

\bibitem[LR15]{lubicz-robert:pairings}
\bysame, \emph{A generalisation of {M}iller's algorithm and applications to
  pairing computations on abelian varieties}, J. Symbolic Comput. \textbf{67}
  (2015), 68--92.

\bibitem[Mil86]{milne:abvars}
J.\thinspace{}S. Milne, \emph{Abelian varieties}, Arithmetic geometry (Storrs,
  Conn., 1984), Springer, New York, 1986, pp.~103--150.

\bibitem[Mil15]{milio:hal1}
Enea Milio, \emph{{A quasi-linear time algorithm for computing modular
  polynomials in dimension 2}}, {LMS Journal of Computation and Mathematics}
  \textbf{18} (2015), 603--632.

\bibitem[MR17]{milio:hal2}
Enea Milio and Damien Robert, \emph{{Modular polynomials on Hilbert surfaces}},
  working paper or preprint, September 2017.

\bibitem[Mum66]{mumford:eq1}
D.~Mumford, \emph{On the equations defining abelian varieties. {I}}, Invent.
  Math. \textbf{1} (1966), 287--354.

\bibitem[Mum67a]{mumford:eq2}
\bysame, \emph{On the equations defining abelian varieties. {II}}, Invent.
  Math. \textbf{3} (1967), 75--135.

\bibitem[Mum67b]{mumford:eq3}
\bysame, \emph{On the equations defining abelian varieties. {III}}, Invent.
  Math. \textbf{3} (1967), 215--244.

\bibitem[Mum84]{mumford1984}
D~Mumford, \emph{Tata lectures on theta ii. jacobian theta functions and
  differential equations, with the collaboration of c. musili, m. nori, e.
  previato, m. stillman and h. umemura}, Progress in Mathematics, vol.~43,
  Boston, MA : Birkh\''{a}user Boston Inc., 1984.

\bibitem[Ric37]{richelot:37}
F.~Richelot, \emph{De transformatione integralium {A}belianorum primi ordinis
  commentatio}, J. Reine Angew. Math. \textbf{16} (1837), 221--284.

\bibitem[Rob10]{drobert:thesis}
D.~Robert, \emph{Fonctions th\^eta et applications \`a la cryptologie}, PhD
  thesis, Universit\'e Henri Poincar\'e - Nancy I (2010).

\bibitem[Smi09]{smith:genus3}
B.~Smith, \emph{Isogenies and the discrete logarithm problem in jacobians of
  genus 3 hyperelliptic curves}, J. Cryptology \textbf{22} (2009), no.~4,
  505--529.

\bibitem[V{\'e}l71]{velu}
J.~V{\'e}lu, \emph{Isog\'enies entre courbes elliptiques}, C. R. Acad. Sci.
  Paris S\'er. A-B \textbf{273} (1971), A238--A241.

\end{thebibliography}

\end{document}